\documentclass[a4paper,oneside,10pt]{scrartcl}\newcommand{\COLOUR}{Fuchsia!80!black}\newcommand{\STRETCH}{1.03}\usepackage[margin=1.25in]{geometry}\usepackage{amsmath,amssymb,amstext,amsthm,setspace,enumerate,mathtools}\usepackage[dvipsnames]{xcolor}\usepackage[linktoc=page,colorlinks=true,linkcolor=\COLOUR,urlcolor=\COLOUR,citecolor=\COLOUR]{hyperref}\usepackage{fancyvrb}\makeatletter\def\PY@reset{\let\PY@it=\relax \let\PY@bf=\relax \let\PY@ul=\relax \let\PY@tc=\relax \let\PY@bc=\relax \let\PY@ff=\relax}\def\PY@tok#1{\csname PY@tok@#1\endcsname}\def\PY@toks#1+{\ifx\relax#1\empty\else\PY@tok{#1}\expandafter\PY@toks\fi}\def\PY@do#1{\PY@bc{\PY@tc{\PY@ul{\PY@it{\PY@bf{\PY@ff{#1}}}}}}}\def\PY#1#2{\PY@reset\PY@toks#1+\relax+\PY@do{#2}}\expandafter\def\csname PY@tok@gd\endcsname{\def\PY@tc##1{\textcolor[rgb]{0.67,0.00,0.00}{##1}}}\expandafter\def\csname PY@tok@gu\endcsname{\let\PY@bf=\textbf\def\PY@tc##1{\textcolor[rgb]{0.50,0.00,0.50}{##1}}}\expandafter\def\csname PY@tok@gt\endcsname{\def\PY@tc##1{\textcolor[rgb]{0.67,0.00,0.00}{##1}}}\expandafter\def\csname PY@tok@gs\endcsname{\let\PY@bf=\textbf}\expandafter\def\csname PY@tok@gr\endcsname{\def\PY@tc##1{\textcolor[rgb]{0.67,0.00,0.00}{##1}}}\expandafter\def\csname PY@tok@cm\endcsname{\def\PY@tc##1{\textcolor[rgb]{0.13,0.55,0.13}{##1}}}\expandafter\def\csname PY@tok@vg\endcsname{\def\PY@tc##1{\textcolor[rgb]{0.00,0.41,0.55}{##1}}}\expandafter\def\csname PY@tok@m\endcsname{\def\PY@tc##1{\textcolor[rgb]{0.71,0.32,0.80}{##1}}}\expandafter\def\csname PY@tok@mh\endcsname{\def\PY@tc##1{\textcolor[rgb]{0.71,0.32,0.80}{##1}}}
\expandafter\def\csname PY@tok@cs\endcsname{\let\PY@bf=\textbf\def\PY@tc##1{\textcolor[rgb]{0.55,0.00,0.55}{##1}}}\expandafter\def\csname PY@tok@ge\endcsname{\let\PY@it=\textit}\expandafter\def\csname PY@tok@vc\endcsname{\def\PY@tc##1{\textcolor[rgb]{0.00,0.41,0.55}{##1}}}\expandafter\def\csname PY@tok@il\endcsname{\def\PY@tc##1{\textcolor[rgb]{0.71,0.32,0.80}{##1}}}\expandafter\def\csname PY@tok@go\endcsname{\def\PY@tc##1{\textcolor[rgb]{0.53,0.53,0.53}{##1}}}\expandafter\def\csname PY@tok@cp\endcsname{\def\PY@tc##1{\textcolor[rgb]{0.12,0.53,0.61}{##1}}}\expandafter\def\csname PY@tok@gi\endcsname{\def\PY@tc##1{\textcolor[rgb]{0.00,0.67,0.00}{##1}}}\expandafter\def\csname PY@tok@gh\endcsname{\let\PY@bf=\textbf\def\PY@tc##1{\textcolor[rgb]{0.00,0.00,0.50}{##1}}}\expandafter\def\csname PY@tok@s2\endcsname{\def\PY@tc##1{\textcolor[rgb]{0.80,0.33,0.33}{##1}}}\expandafter\def\csname PY@tok@nn\endcsname{\let\PY@ul=\underline\def\PY@tc##1{\textcolor[rgb]{0.00,0.55,0.27}{##1}}}\expandafter\def\csname PY@tok@no\endcsname{\def\PY@tc##1{\textcolor[rgb]{0.00,0.41,0.55}{##1}}}\expandafter\def\csname PY@tok@na\endcsname{\def\PY@tc##1{\textcolor[rgb]{0.40,0.55,0.00}{##1}}}\expandafter\def\csname PY@tok@nb\endcsname{\def\PY@tc##1{\textcolor[rgb]{0.40,0.55,0.00}{##1}}}\expandafter\def\csname PY@tok@nc\endcsname{\let\PY@bf=\textbf\def\PY@tc##1{\textcolor[rgb]{0.00,0.55,0.27}{##1}}}\expandafter\def\csname PY@tok@nd\endcsname{\def\PY@tc##1{\textcolor[rgb]{0.44,0.48,0.49}{##1}}}\expandafter\def\csname PY@tok@ne\endcsname{\let\PY@bf=\textbf\def\PY@tc##1{\textcolor[rgb]{0.00,0.55,0.27}{##1}}}\expandafter\def\csname PY@tok@nf\endcsname{\def\PY@tc##1{\textcolor[rgb]{0.00,0.55,0.27}{##1}}}\expandafter\def\csname PY@tok@si\endcsname{\def\PY@tc##1{\textcolor[rgb]{0.80,0.33,0.33}{##1}}}\expandafter\def\csname PY@tok@sh\endcsname{\let\PY@it=\textit\def\PY@tc##1{\textcolor[rgb]{0.11,0.49,0.44}{##1}}}
\expandafter\def\csname PY@tok@vi\endcsname{\def\PY@tc##1{\textcolor[rgb]{0.00,0.41,0.55}{##1}}}\expandafter\def\csname PY@tok@nt\endcsname{\let\PY@bf=\textbf\def\PY@tc##1{\textcolor[rgb]{0.55,0.00,0.55}{##1}}}\expandafter\def\csname PY@tok@nv\endcsname{\def\PY@tc##1{\textcolor[rgb]{0.00,0.41,0.55}{##1}}}\expandafter\def\csname PY@tok@s1\endcsname{\def\PY@tc##1{\textcolor[rgb]{0.80,0.33,0.33}{##1}}}\expandafter\def\csname PY@tok@gp\endcsname{\def\PY@tc##1{\textcolor[rgb]{0.33,0.33,0.33}{##1}}}\expandafter\def\csname PY@tok@ow\endcsname{\def\PY@tc##1{\textcolor[rgb]{0.55,0.00,0.55}{##1}}}\expandafter\def\csname PY@tok@sx\endcsname{\def\PY@tc##1{\textcolor[rgb]{0.80,0.42,0.13}{##1}}}\expandafter\def\csname PY@tok@bp\endcsname{\def\PY@tc##1{\textcolor[rgb]{0.40,0.55,0.00}{##1}}}\expandafter\def\csname PY@tok@c1\endcsname{\def\PY@tc##1{\textcolor[rgb]{0.13,0.55,0.13}{##1}}}\expandafter\def\csname PY@tok@kc\endcsname{\let\PY@bf=\textbf\def\PY@tc##1{\textcolor[rgb]{0.55,0.00,0.55}{##1}}}\expandafter\def\csname PY@tok@c\endcsname{\def\PY@tc##1{\emph{\textcolor[rgb]{0.60,0.60,0.53}{##1}}}}\expandafter\def\csname PY@tok@mf\endcsname{\def\PY@tc##1{\textcolor[rgb]{0.71,0.32,0.80}{##1}}}\expandafter\def\csname PY@tok@err\endcsname{\def\PY@tc##1{\textcolor[rgb]{0.65,0.09,0.09}{##1}}\def\PY@bc##1{\setlength{\fboxsep}{0pt}\colorbox[rgb]{0.89,0.82,0.82}{\strut ##1}}}\expandafter\def\csname PY@tok@mb\endcsname{\def\PY@tc##1{\textcolor[rgb]{0.71,0.32,0.80}{##1}}}\expandafter\def\csname PY@tok@ss\endcsname{\def\PY@tc##1{\textcolor[rgb]{0.80,0.33,0.33}{##1}}}\expandafter\def\csname PY@tok@sr\endcsname{\def\PY@tc##1{\textcolor[rgb]{0.11,0.49,0.44}{##1}}}\expandafter\def\csname PY@tok@mo\endcsname{\def\PY@tc##1{\textcolor[rgb]{0.71,0.32,0.80}{##1}}}\expandafter\def\csname PY@tok@kd\endcsname{\let\PY@bf=\textbf\def\PY@tc##1{\textcolor[rgb]{0.55,0.00,0.55}{##1}}}\expandafter\def\csname PY@tok@mi\endcsname{\def\PY@tc##1{\textcolor[rgb]{0.71,0.32,0.80}{##1}}}\expandafter\def\csname PY@tok@kn\endcsname{\let\PY@bf=\textbf\def\PY@tc##1{\textcolor[rgb]{0.55,0.00,0.55}{##1}}}\expandafter\def\csname PY@tok@kr\endcsname{\let\PY@bf=\textbf\def\PY@tc##1{\textcolor[rgb]{0.55,0.00,0.55}{##1}}}\expandafter\def\csname PY@tok@s\endcsname{\def\PY@tc##1{\textcolor[rgb]{0.80,0.33,0.33}{##1}}}\expandafter\def\csname PY@tok@kp\endcsname{\let\PY@bf=\textbf\def\PY@tc##1{\textcolor[rgb]{0.55,0.00,0.55}{##1}}}\expandafter\def\csname PY@tok@w\endcsname{\def\PY@tc##1{\textcolor[rgb]{0.73,0.73,0.73}{##1}}}
\expandafter\def\csname PY@tok@kt\endcsname{\let\PY@bf=\textbf\def\PY@tc##1{\textcolor[rgb]{0.65,0.65,0.65}{##1}}}\expandafter\def\csname PY@tok@sc\endcsname{\def\PY@tc##1{\textcolor[rgb]{0.80,0.33,0.33}{##1}}}\expandafter\def\csname PY@tok@sb\endcsname{\def\PY@tc##1{\textcolor[rgb]{0.80,0.33,0.33}{##1}}}\expandafter\def\csname PY@tok@k\endcsname{\let\PY@bf=\textbf\def\PY@tc##1{\bfseries{\textcolor[rgb]{0.55,0.00,0.55}{##1}}}}\expandafter\def\csname PY@tok@se\endcsname{\def\PY@tc##1{\textcolor[rgb]{0.80,0.33,0.33}{##1}}}\expandafter\def\csname PY@tok@sd\endcsname{\def\PY@tc##1{\textcolor[rgb]{0.80,0.33,0.33}{##1}}}\def\PYZbs{\char`\\}\def\PYZus{\char`\_}\def\PYZca{\char`\^}\def\PYZlt{\char`\<}\def\PYZgt{\char`\>}\def\PYZsh{\char`\#}\def\PYZpc{\char`\%}\def\PYZhy{\char`\-}\def\PYZsq{\char`\'}\def\PYZdq{\char`\"}\makeatother\newtheorem{lemma}{Lemma}\newtheorem{theorem}[lemma]{Theorem}\newtheorem{corollary}[lemma]{Corollary}\newtheorem{conjecture}[lemma]{Conjecture}\newcommand{\ssssss}{$\ast$}\newcommand{\tttttt}{---}\DeclareMathOperator{\Symm}{Symm}\DeclareMathOperator{\ind}{ind}\DeclareMathOperator{\End}{End}\DeclareMathOperator{\Ker}{ker}\DeclareMathOperator{\Gal}{Gal}\DeclareMathOperator{\GL}{GL}\DeclareMathOperator{\SL}{SL}\DeclareMathOperator{\Img}{im}\DeclareMathOperator{\GCD}{gcd}\DeclareMathOperator{\Ind}{ind}\newcommand{\SECTION}{}\newcommand{\SECTIONN}{}\newenvironment{Smatrix}{\tiny \left(\begin{smallmatrix}}{\end{smallmatrix}\right)}\newenvironment{Lmatrix}{\begin{Smatrix}}{\end{Smatrix}}\newenvironment{Xmatrix}{\tiny \begin{pmatrix}}{\end{pmatrix}}\usepackage{titlesec}\titleformat*{\subsection}{\large\sf}\titleformat*{\subsubsection}{\sf}\newcommand{\yyyyyyyyyy}{}\newcommand{\xxxxxxxxx}{\boldsymbol}\newcommand{\SECTIONNN}{\ \vspace{24pt}}\newcommand{\xxxxxxxxxx}{\boldsymbol}\usepackage{mathtools}\newcommand{\defeq}{\vcentcolon=}
    \title{\SECTIONNN\\ \yyyyyyyyyy Reduction modulo $\xxxxxxxxx p$ of two-dimensional crystalline representations of $\xxxxxxxxx {G_{\mathbb Q_p}}$ of slope less than three}\renewcommand{\SECTION}{}\newcommand{\PART}[1]{}
\author{Bodan Arsovski\setcounter{footnote}{1}\footnote{Department of Mathematics, Imperial College London, 180 Queen's Gate, London SW7 2AZ, United Kingdom}}\date{\it \small May, 2015}
\begin{document}\setstretch{\STRETCH}\maketitle\thispagestyle{empty}\newcommand{\NUMBERR}{3}\newcommand{\NUMBERRR}{38}\newcommand{\NUMBERRRRRR}{36}\newcommand{\LETTERSSS}{thirty-eight}\newcommand{\NUMBER}{6}\newcommand{\LETTERS}{six}
\renewcommand{\SECTIONN}{\newpage}%

\begin{abstract} \noindent {\sc Abstract.} We use the $p$-adic local Langlands correspondence for $\smash{\mathrm{GL}_2(\mathbb Q_p)}$ to find the reduction modulo $p$ of certain two-dimensional crystalline Galois representations, following the approach used in~\cite{b2} and~\cite{b4}. In particular, we resolve a conjecture of Breuil, Buzzard, and Emerton in the case when the slope is strictly between one and three, and prove partial results towards this conjecture for arbitrary slopes. Moreover, we partially classify the reduction modulo~$p$ of these representations when the slope is equal to one. \\[5pt] \noindent {\sc Keywords.} \emph{\small crystalline representations, local Langlands correspondence, modular forms.} \end{abstract}

\ \\[-22pt]

\SECTIONN {\tableofcontents}

\SECTION \section{Introduction}

\subsection{Definitions and notation}

Throughout this paper, we assume that $p > 2$ is an odd prime number, we let $K = \GL_2(\mathbb Z_p)$, and we let $Z$ be the centre of $G = \GL_2(\mathbb Q_p)$. If $R$ is a $\mathbb Z_p$-algebra, we define $\Symm^r(R)$ to be the space of homogeneous polynomials in $R[x,y]$ which have degree~$r$. There is the obvious right action of $K$ on $\Symm^r(R)$, which is defined in section~2 of~\cite{b2}, and which can be extended to an action of $KZ$ by letting $pI$ act trivially. We let $I(V)$ be the compact induction $\ind_{KZ}^G (V)$---the space of functions $f : G \to V$ which have compact support modulo~$Z$ and which satisfy $f(\kappa g) = \kappa f(g)$ for all $\kappa \in KZ$. There is a right action of $G$ on $I(V)$ defined by $fg (\gamma) = f(\gamma g)$.
If $V = \Symm^r(R^2)$ then there is an endomorphism $T$ of $I(V)$ which corresponds to the function $G \to \End_R(V)$ which is supported on $KZ {\tiny \begin{Smatrix} p &  0 \\ 0 & 1 \end{Smatrix}} KZ$ and sends ${\tiny \begin{Smatrix} p &  0 \\ 0 & 1 \end{Smatrix}}$ to the map $\psi (x,y) \mapsto \psi (px,y)$. We let $[g,v]$ be the unique element of $I(V)$ which is supported on $KZg^{-1}$ and which satisfies $[g,v]g^{-1} = v$. Then $g[h,v] = [gh,v]$ for $g,h \in G$, and $[g\kappa,v] = [g,v\kappa ]$ for $\kappa \in KZ$, and the $[g,v]$ span $I(V)$ as an abelian group. In section~2 of~\cite{b3} it is shown that $T$ can be written as
	\[\textstyle T[g,v] = \sum_{\lambda \in \mathbb F_p} \left[g {\begin{Lmatrix} p & [\lambda] \\ 0 & 1 \end{Lmatrix}},v(x,-[\lambda] x + py)\right] + \left[g {\begin{Lmatrix} 1 & 0 \\ 0 & p \end{Lmatrix}}, v(px,y)\right]. \]
Here $[\lambda]$ is the Teichm\"uller lift of $\lambda$ to $\mathbb Z_p$.  The same equation defines an endomorphism $T$ of $\smash{I(\det^s\otimes \Symm^r(\overline{\mathbb F}_p^2))}$. We denote $\smash{\Symm^r(\overline{\mathbb F}_p^2)}$ by $\sigma_r$. We denote by $\sigma_m(n)$ the twist $\det^n \otimes\, \sigma_m$. We denote by $\omega$ the $\bmod p$ reduction of the cyclotomic character, and we denote by $\omega_2$ the $\bmod p$ reduction of  a choice of a fundamental character of niveau~2. If $\smash{\chi : \mathbb Q_p^\times \to \overline{\mathbb F}_p^\times}$ is a character, if $\lambda \in \overline{\mathbb F}_p$, and if $0\leqslant \nu\leqslant p-1$, then we define $\smash{\pi(\nu, \lambda, \chi) =  (\chi \circ \det) \otimes \left(I(\sigma_{\nu})/(T-\lambda)\right)}$. If $R$ is a ring and $t \in R^\times$, then we let $\mu_t$ denote the map $\GL_1(\mathbb Q_p) \to R^\times$ which is trivial on $\mathbb Z_p^\times$ and which sends $p$ to $t$. The following classification of the modules $\pi(\nu, \lambda, \chi)$ is given in sections~2 and~3 of~\cite{b2}.

\begin{theorem}\label{t0.2}
\begin{enumerate}
	\item If $(\nu,\lambda) \not\in \{(0,\pm1),(p-1,\pm1)\}$, then $\pi(\nu, \lambda, \chi)$ is irreducible. The modules $\pi(\nu, \lambda, \chi)$ with $\lambda = 0$ are the supersingular representations of $G$; no Jordan--H\"older factor of $\pi(\nu, \lambda, \chi)$ with $\lambda \ne 0$ is supersingular.
	\item Each $\pi(\nu, \lambda, \chi)$ has finite length.
	\item The only isomorphisms between the $\pi(\nu, \lambda, \chi)$ are the following:
	\begin{enumerate}[\tttttt] 
		\item if $\lambda \ne 0$ and $(\nu,\lambda) \not\in \{(0,\pm1),(p-1,\pm1)\}$, then $\pi(\nu, \lambda, \chi) \cong \pi(\nu, -\lambda, \chi\mu_{-1})$;
		\item if $\lambda \not\in \{0,\pm1\}$, then $\pi(0, \lambda, \chi) \cong \pi(p-1, -\lambda, \chi)$;
		\item $\pi(\nu, 0, \chi) \cong \pi(\nu, 0, \chi\mu_{-1}) \cong \pi(p-1-\nu, 0, \chi\omega^\nu)\cong \pi(p-1-\nu, 0, \chi\omega^\nu\mu_{-1})$.
	\end{enumerate}
	\item If $\lambda \ne 0$, and if $\pi(\nu, \lambda, \chi)$ and $\pi(\nu', \lambda', \chi')$ have a common Jordan--H\"older factor, then $\lambda' \ne 0$, $\nu\equiv_{p-1} \nu'$, and $\chi/\chi'$ is unramified.
	\item If $0 \leqslant \nu \leqslant p- 1$, and if $F$ is a quotient of $I(\det^l \otimes \,\sigma_\nu)$ which has finite length as an $\overline{\mathbb F}_p [G]$-module, then every Jordan--H\"older factor of $F$ is a subquotient of $\pi(\nu, \lambda, \omega^l)$, for some $\lambda$ which possibly depends on the factor.
\end{enumerate}
\end{theorem}
If $\smash{f = \sum_{n \geqslant 1} a_n q^n}$ is an eigenform for the subgroup $\Gamma_1(N) \subseteq \SL_2(\mathbb Z)$, which has character $\psi$, then there is a $p$-adic Galois representation
	\[\textstyle \rho_{f} : \Gal(\overline{\mathbb Q}/\mathbb Q) \to \GL_2(\overline{\mathbb Q}_p).\]
If $\ell \ne p$ is a prime distinct from $p$, then the local structure at $\ell$ of the Galois representation $\rho_f$ associated with $f$ can be explicitly described. The local structure at $p$ of $\rho_f$ is in general more complicated. However, the reduction $\smash{\overline \rho_f|_{D_p}}$ modulo $p$ can be classified; if $\smash{\overline \rho_f|_{D_p}}$ is reducible then its semisimplifaction is the sum of two characters, and if $\smash{\overline \rho_f|_{D_p}}$ is irreducible then it is induced from characters of the absolute Galois group of the unramified quadratic extension of~$\mathbb Q_p$. When $N$ and $p$ are coprime, then $\smash{\overline \rho_f|_{D_p}}$ can be determined from the weight $k$, the coefficient $a_p=a$ of the $q$-series expansion, and the character $\psi$.
 If $r \geqslant 0$, then there is a uniquely determined two-dimensional crystalline representation $V_{r+2,a}$ of $\Gal(\overline{\mathbb Q}_p /\mathbb Q_p)$ which has Hodge--Tate weights $0$ and $r + 1$, determinant the cyclotomic character to the power of $r+1$, and $X^2 - aX + p^{r+1}$ as the characteristic polynomial of crystalline Frobenius on the contravariant Dieudonne module. A construction of this representation is given for example in~\cite{b3}. This is the local Galois representation at $p$ associated with the eigenform $\smash{f = \sum_{n \geqslant 1} a_n q^n}$ of weight $k=r+2$. We will denote by $\overline{V}_{r+2,a}$ the semisimplification of the reduction of $V_{r+2,a}$ modulo the maximal ideal of $\overline{\mathbb Z}_p$.
If $r \geqslant 0$, and $a \in \overline {\mathbb Q}_p$ is such that $v(a) > 0$, and if the roots of $X^2 - aX + p^{r+1}$ do not have ratio $p^{\pm 1}$ or $1$, then we define 
	\[\textstyle \smash{\textstyle \Pi_{r+2,a}=I(\Symm^{r}(\overline{\mathbb Q}_p^2))/(T-a)},\]
and we let $\Theta_{r+2,a}$ be the image of $\smash{I (\Symm^{r} (\overline{\mathbb Z}_p^2))}$ in $\Pi_{r+2,a}$. There is a natural surjective map 
	\[\textstyle \smash{I(\sigma_{r}) \twoheadrightarrow \overline{\Theta}_{r+2,a} = \Theta_{r+2,a} \otimes \overline {\mathbb F}_p}.\]
The following theorem is the $p$-adic local Langlands correspondence for $\GL_2(\mathbb Q_p)$, and is given, for example, in section~1 of~\cite{b3} and in section~2 of~\cite{b2}.

\begin{theorem}\label{t0.1}
If the roots of $X^2 - aX + p^{r+1}$ do not have ratio $1$ or $p^{\pm 1}$, then
	\begin{align*}
	& \textstyle
	\overline V_{r+2,a} \cong (\Ind(\omega_2^{\nu+1})) \otimes \chi
	\\ & \textstyle
	\quad\qquad\qquad \Longleftrightarrow \quad (\overline \Theta_{r+2,a})^{\mathrm{ss}} \cong \pi(\nu,0,\chi);
	\\ & \textstyle
	\overline V_{r+2,a} \cong (\mu_\lambda \omega^{\nu+1} \oplus \mu_{\lambda^{-1}}) \otimes \chi
	\\ & \textstyle
	\quad\qquad\qquad \Longleftrightarrow \quad (\overline \Theta_{r+2,a})^{\mathrm{ss}} \cong \pi(\nu,\lambda,\chi)^{\mathrm{ss}} \oplus \pi((p-3-\nu) \bmod {(p-1)},\lambda^{-1},\chi\omega^{\nu+1})^{\mathrm{ss}}.
	\end{align*}
\end{theorem}

If $N$ and $p$ are not coprime, say $N = N'p^{m}$ with $\GCD(N',p) = 1$, and if the $p$-part $\chi$ of the character $\psi$ has conductor $p^{m}$, then again the reduction $\smash{\overline \rho_f|_{D_p}}$ can be determined from the weight $k$, the $p^{\text{th}}$ coefficient of the $q$-series expansion, and the character $\psi$. We will denote by $\overline{V}_{r+2,a,\chi}$ the semisimplification of this reduction modulo the maximal ideal of $\overline{\mathbb Z}_p$.

\subsection{Statements of the main results}

\subsubsection{In the interior of the weight space}

Theorem~3.2.1 in~\cite{b6}, theorem~1.6 in~\cite{b2}, and corollary~4.7 in~\cite{b4} completely determine $\overline V_{r+2,a}$ in the following cases (with $p > 2$):
\begin{enumerate}
	\item $0 \leqslant r \leqslant 2p-1$;
	\item $r > 2p-1$ and $v(a) < 1$;
	\item $r > 2p-1$ and $v(a) > \lfloor \frac r{p-1}\rfloor$.
\end{enumerate}
In this paper we are concerned with the remaining case: when $r \geqslant 2p$ and $\lfloor \frac r{p-1}\rfloor \geqslant v(a) \geqslant 1$. The papers~\cite{b8} and~\cite{b9} give surveys of the known results and conjectures about $p$-adic slopes of modular forms. In particular, the following conjecture is given in section~4 of~\cite{b9}.

\begin{conjecture}[Breuil, Buzzard, Emerton]\label{c0000}
If $r$ is even and $v(a) \not\in \mathbb Z$, then $\overline V_{r+2,a}$ is irreducible.
\end{conjecture}

We prove partial results towards this conjecture, in addition to giving a partial classification of $\overline V_{r+2,a}$ in the case when $v(a)=1$. Let $r$ be a positive integer such that $r = t(p - 1) + s$, with $t \geqslant 0$ and $s \in \{1,\ldots,p-1\}$. Let $a \in \overline {\mathbb Q}_p$ will be such that and $v(a) \geqslant 1$, and such that the roots of $X^2 - aX + p^{r+1}$ do not have ratio $1$ or $p^{\pm 1}$. In particular, due to a theorem on Newton polygons, the second condition will always hold true whenever $r > 2v(a)$. We will interchangeably refer to the coefficient $a_p$ by $a$, and to the valuation $v(a)$ by $v$. %
 In section~\ref{Sect3}, we give a partial classification of $\overline V_{r+2,a}$ in the case when~$v(a)=1$. More specifically, we show the following theorem.

\begin{theorem}\label{t0.0}
Let $v = 1$, and $s \not \in \{1,3\}$.
\begin{enumerate}[1.] 
	\item If $p \nmid r - s$, then $\overline V_{r+2,a}$ is one of $\{\mu_\lambda \omega^{s} \oplus \mu_{\lambda^{-1}}\omega, \,\Ind(\omega_2^{s+p})\}$, where $\lambda = \overline{a/p} \cdot \frac{s}{s-r}$.
	\item If $p \mid r - s$, then $\overline V_{r+2,a}$ is one of $\{\Ind(\omega_2^{s+1}), \,\Ind(\omega_2^{s+p})\}$.
\end{enumerate}
\end{theorem}

In section~\ref{Sect4}, we prove special cases of conjecture~\ref{c0000}, when the slope is small. More specifically, we show the following theorems.

\begin{theorem}\label{t0.0''0}
If $r$ is even and $2 > v > 1$, then $\overline V_{r+2,a}$ is irreducible.
\end{theorem}

\begin{theorem}\label{t0.0''}
If $r$ is even and $3 > v > 2$, then $\overline V_{r+2,a}$ is irreducible.
\end{theorem}

\begin{theorem}\label{t0.0''__}
Let $4 > v > 3$ and let $r$ be even. Then $\overline V_{r+2,a}$ is irreducible whenever
	\begin{align*}& \textstyle r \not\equiv \left\{3p + 1,3p + 3, 4p,  4p + 2, 5p + 1, 6p, s, s+p-1, s+2p-2\right\} \pmod{p(p-1)}.\end{align*}
\end{theorem}

Let $x^{(n)} = x(x+1)\cdots(x+n-1)$ denote the rising factorial power.

\begin{conjecture}\label{c1111}
If $r$ is even and $v \not\in \mathbb Z$, then $\overline V_{r+2,a}$ is irreducible in the following cases.
\begin{enumerate}[1.] 
\item If $s \in\{2,\ldots,2\lfloor v\rfloor\}$, and if the extra condition $p \nmid (r-s)^{(2\lfloor v\rfloor+1)}$ holds true.
\item If $s \not\in\{2,\ldots,2\lfloor v\rfloor\}$, and if the extra condition $p \nmid (r-s)^{(\lfloor v\rfloor)}$ holds true. In fact, in this case
	\[\smash{\overline V_{r+2,a} \cong \Ind(\omega_2^{s+(p-1)\lfloor v\rfloor+1})}.\]
\item If $s \not\in\{2,\ldots,2\lfloor v\rfloor\}$.
\end{enumerate}
\end{conjecture}

\begin{theorem}\label{iygutfrd} We have the following partial results towards conjecture~\ref{c1111}.
\begin{enumerate}[1.] 
\item The first part of conjecture~\ref{c1111} is true %
 when $\NUMBER>v$.
\item The second part of conjecture~\ref{c1111} is true  %
 when  $\NUMBERRR>v$.
\item The third part of conjecture~\ref{c1111} is true %
 when $\NUMBERR>v$.
\end{enumerate}
\end{theorem}

By combining these theorems with previously known results about conjecture~\ref{c0000} (for $\smash{v > \lfloor \frac r{p-1}\rfloor}$), we can deduce the following corollary.

\begin{corollary}\label{c1}
Conjecture~\ref{c0000} is true in the following cases:
\begin{enumerate}[\ssssss]
	\item $3 > v > 1$;
	\item $4 > v > 3$, and
		\[\textstyle r \not\equiv \left\{3p + 1,3p + 3, 4p,  4p + 2, 5p + 1, 6p, s, s+p-1, s+2p-2\right\} \pmod{p(p-1)};\]
	\item $\NUMBER > v$, and $s \in\{2,\ldots,2\lfloor v\rfloor\}$, and $p \nmid (r-s)^{(2\lfloor v\rfloor+1)}$;
	\item $\NUMBERRR>v$, and $s \not\in\{2,\ldots,2\lfloor v\rfloor\}$, and $p \nmid (r-s)^{(\lfloor v\rfloor)}$.
\end{enumerate}
In particular, Conjecture~\ref{c0000} is true when $\NUMBER > v$ and $p \nmid (r-s)^{(2\lfloor v\rfloor+1)}$.
\end{corollary}

\subsubsection{Near the boundary of the weight space}

No local results are known when $N = N'p^{m}$ with $\GCD(N',p) = 1$ and $m > 0$, though a  result similar to conjecture~\ref{c0000} is expected to hold true when the $p$-part of the character has conductor $p^{m}$. In particular, the following conjecture is given in section~4 of~\cite{b9} (the full statement of conjecture~4.2.1 in~\cite{b9} includes the case when $p=2$, but in this paper we are only interested in the case when $p$ is odd).

\begin{conjecture}[Buzzard, Gee]\label{c00000}
If $m \geqslant 2$ and $(p-1)p^{m-2} v(a) \not\in \mathbb Z$, then $\overline V_{r+2,a,\chi}$ is irreducible.
\end{conjecture}

There are global results, in~\cite{b18,b21,b19,b20}, which show that $(p-1)p^{m-2-\delta_{p=2}} v(\alpha) \in \mathbb Z$ when $\alpha$ is an eigenvalue of $U_p$ on a space of modular forms of level $2^m,3^m,5^2,7^2$, respectively. More results close to the boundary of weight space are proven in~\cite{b21,b22}.

\subsection{Assumptions}

In the remainder of this paper, we will make the following assumptions:
\begin{enumerate}[\tttttt]
	\item $r$ will be a positive integer such that $r = t(p - 1) + s$, with $t \geqslant 0$ and $s \in \{1,\ldots,p-1\}$;
	\item $a \in \overline {\mathbb Q}_p$ will be such that and $v(a) \geqslant 1$, and such that the roots of $X^2 - aX + p^{r+1}$ do not have ratio $p^{\pm 1}$ or $1$. The second condition will always hold true whenever $r > 2v(a)$.
\end{enumerate}

\SECTION \PART{The interior of weight space}

\section{Auxiliary lemmas}

\subsection[Properties of the maps $\Psi_\alpha$]{Properties of the maps $\Psi_\alpha$}

We define the map $\textstyle \Psi : \sigma_r  \to \sigma_{p-1-s} (s)$ as follows:
	\begin{align*}
	\textstyle \Psi \textstyle : \sigma_r & \textstyle \to \sigma_{p-1-s} (s)
	\\ 
	\textstyle f & \textstyle \mapsto \sum_{u,v \in \mathbb F_p} f(u,v) (vX-uY)^{p-1-s}.
	\end{align*}
It can be shown that this map is surjective and $\GL_2(\mathbb F_p)$-equivariant, and it induces a map on the induction $\smash{I(\sigma_r)}$. We will denote this induced map also by $\Psi$.

\begin{lemma}\label{l1}
Let $t \geqslant 2$ be an integer, and suppose that $r = t(p-1) + s$, with $s \in \{1,\ldots,p-1\}$. Suppose moreover that $l\in\{1,\ldots,t\}$ and $i \in \{0,\ldots,s-1\}$. Then $\textstyle\smash{\Psi x^{r-i}y^i = \Psi x^iy^{r-i} = 0}$, and $\smash{\Psi x^{l(p-1)}y^{r-l(p-1)} = X^{p-1-s}}$.
\end{lemma}

\begin{proof}
If $0 \leqslant i \leqslant r$, then
	\begin{align*}
	\textstyle[X^jY^{p-j-1-s}]\Psi x^iy^{r-i} 
	& \textstyle
	= \sum_{u,v \in \mathbb F_p} (-1)^{j+s} {p-1-s \choose j}u^{i+p-j-1-s}v^{r-i+j}
	\\ & \textstyle
	= (-1)^{j+s} {p-1-s \choose j} \sum_{u \in \mathbb F_p} u^{i+p-j-1-s} \sum_{v \in \mathbb F_p}v^{r-i+j}
	\\ & \textstyle
	= (-1)^{j+s} {p-1-s \choose j} \,\Xi_{i,j},
	\end{align*}
where for convenience we use the notation $0^0 = 1$. This implies that $\smash{\textstyle\sum_{u \in \mathbb F_p} u^{\xi} = \delta_{\xi \equiv_{p-1} 0} \delta_{\xi \ne 0}}$, where $\delta_P = 1$ if $P$ holds true and $\delta_P = 0$ otherwise. Then
	\begin{align*}
	\textstyle \Xi_{ij} 
	& \textstyle 
	= \delta_{i+p-j-1-s \equiv_{p-1} 0}  \delta_{i+p-j-1-s \ne 0} \delta_{r-i+j\equiv_{p-1} 0} \delta_{r-i+j \ne 0}
	\\ & \textstyle
	=  \delta_{j \equiv_{p-1} i - s} \delta_{(i,j) \not\in \{(0,p-1-s), (r,0)\}}
	\\ &  \textstyle
	=  \delta_{j \equiv_{p-1} i - s} \delta_{i \not\in \{0,r\}}.
	\end{align*}
Consequently,
	\[\textstyle[X^jY^{p-j-1-s}]\Psi x^iy^{r-i} = (-1)^{j+s} {p-1-s \choose j} \delta_{j \equiv_{p-1} i - s} \delta_{i \not\in \{0,r\}}.\]
Since $j \in \{0,\ldots,p-1-s\}$, then $\delta_{j \equiv_{p-1} i - s} \delta_{i \not\in \{0,r\}} = 0$ whenever $i \in \{0,\ldots,s-1\}$, so $\Psi x^iy^{r-i} = 0$. Since $\Psi$ is $\GL_2(\mathbb F_p)$-equivariant, then $\textstyle\Psi x^{r-i}y^i = 0$ whenever $i \in \{0,\ldots,s-1\}$ as well. Finally, if $l\in\{1,\ldots,t\}$, then
	\[\textstyle\Psi x^{l(p-1)}y^{r-l(p-1)} = (-1)^{p-1} {p-1-s \choose p-1-s} X^{p-1-s} = X^{p-1-s},\]
which completes the proof.
\end{proof}

\begin{lemma}\label{l1.1}
Let $t \geqslant 1$ be an integer, and suppose that $r = t(p-1) + s$, with $s \in \{1,\ldots,p-1\}$. In $I(\sigma_{p-1-s}(s))$,
	\begin{align*}
	\textstyle T : [1,X^{p-1-s}] & \mapsto \textstyle \sum_{\mu \in \mathbb F_p} \left[{\begin{Lmatrix} p & [\mu] \\ 0 & 1 \end{Lmatrix}},X^{p-1-s}\right] + \delta_{s=p-1} \left[{\begin{Lmatrix} 1 & 0 \\ 0 & p \end{Lmatrix}},X^{p-1-s}\right],
	\end{align*}
where $\delta_{s=p-1}=1$ if $s = p-1$, and $\delta_{s=p-1}=0$ otherwise.
\end{lemma}

\begin{proof}
Immediate from the definition of $T$.
\end{proof}

Let $h \geqslant 0$ be an integer, and suppose that  $r \geqslant h(p+1)$. We define the map
	\begin{align*}
	\textstyle \Psi_h \textstyle : \overline\theta^h\sigma_{r-h(p+1)}\subseteq \sigma_r & \textstyle \to \sigma_{(2h-r \bmod p-1)} (r-h)
	\\ 
	\textstyle \overline\theta^h f & \textstyle \mapsto \overline\theta^h \Psi f.
\end{align*}
This map is surjective and $\GL_2(\mathbb F_p)$-equivariant, due to the fact that $\overline\theta^h\sigma_{r-h(p+1)} \cong \sigma_{r-h(p+1)}(h)$, and it induces a map on
$\smash{I(\overline\theta^h\sigma_{r-h(p+1)})}$. We will denote this induced map also by $\Psi_h$.

\subsection[Properties of the kernel of reduction $X( r+2, a)$]{Properties of the kernel of reduction $X( r+2, a)$}

The following statement is a technical lemma about modules.

\begin{lemma}\label{l4.1.1}
Suppose that $A,B,C$ are modules and $\beta : A \to B$ and $\gamma : A \to C$ are surjective maps. Suppose that $B' \subseteq B$ is a submodule such that, for all $b' \in B'$, there exists a $c' \in \Ker\gamma$ such that $\beta(c') = b'$. Then there is a series $C = C_2 \supseteq C_1 \supseteq C_0 = \{0\}$ of length~two, whose factor $C/C_1$ is isomorphic to a quotient of $B/B'$ and whose other factor $C_1$ is isomorphic to $\Ker\beta/(\Ker \beta \cap \Ker\gamma)$.
\end{lemma}

\begin{proof}
There is the composition map $A \stackrel{\gamma}{\longrightarrow} C \stackrel{q}{\longrightarrow} C/\gamma(\Ker\beta)$. This map is surjective, and its kernel contains $\Ker\beta$. Let $C_1 = \gamma(\Ker\beta)$. Then $C_1 \cong \Ker\beta/\Ker (\gamma|_{\Ker\beta})= \Ker\beta/(\Ker \beta \cap \Ker\gamma)$. Moreover, we have $\Ker \gamma q = \Ker\beta + \Ker\gamma$. Because of the condition that, for all $b' \in B'$, there exists a $c' \in \Ker\gamma$ such that $\beta(c') = b'$, we know that $\Ker\beta + \Ker\gamma \supseteq \Ker\beta + \beta^{-1}(B')$. Hence $C/C_1 \cong A/(\Ker\beta + \Ker\gamma)$ is isomorphic to a quotient of $A/(\Ker\beta + \beta^{-1}(B'))$, which is isomorphic to $B/B'$ by several applications of the isomorphism theorems. 
\end{proof}

Let $X(r+2,a)$ denote the kernel of the surjection $I (\sigma_r) \twoheadrightarrow \overline{\Theta}_{r+2,a}$. The following statements can be proven by the same method used to show lemmas~4.1 and~4.3 in \cite{b2}.

\begin{lemma}\label{l4.1}
\begin{enumerate}[1.]
	\item Let $r \geqslant 2(p + 1)$, and suppose that $a \in \overline {\mathbb Q}_p$ is such that $2 > v(a) \geqslant 1$. Denote by $\overline\theta$ the polynomial $xy^p-x^py \in \smash{\overline{\mathbb F}_p[x,y]}$. Then $X(r+2,a)$ contains  $I(Y_r)$, where $Y_r$ is the subrepresenation of $\sigma_r$ generated by $\smash{\overline\theta^2 \sigma_{r-2(p+1)}}$ and $y^r$.
	\item Let $r \geqslant (m+1)(p + 1)$, and suppose that $a \in \overline {\mathbb Q}_p$ is such that $m+1 > v(a) > m$. Then $X(r+2,a)$ contains the induction of the subrepresenation of $\sigma_r$ generated by $\smash{\overline\theta^{m+1} \sigma_{r-(m+1)(p+1)}}$ and $\{y^r,xy^{r-1},\ldots,x^my^{r-m}\}$.
\end{enumerate}
\end{lemma}

Before proving lemma~\ref{l4.1}, we will first show two auxiliary results.

\begin{lemma}\label{l7x}
Let $r,m,n,c \geqslant 0$ be integers, and let $a \in \overline {\mathbb Q}_p$ be such that $m+1 > v(a) > m$. Suppose that $n \geqslant m+1$, and $r \geqslant np + c$. Then there is some $\textstyle\phi_{r,m,n,c,g}$ such that 
	\begin{align*}
		\textstyle (T - a) \phi_{r,m,n,c,g} & \textstyle \equiv 
			 \sum_{j}  (-1)^{j} {{n}  \choose j} p^{r-j(p-1)-c} \left[g{\begin{Lmatrix} p & 0 \\ 0 & 1 \end{Lmatrix}},x^{j(p-1)+c}y^{r-j(p-1)-c}\right]
	\\ & \textstyle \qquad
		+ \sum_{j}  (-1)^{j} {{n}  \choose j} p^{j(p-1)+c} \left[g{\begin{Lmatrix} 1 & 0 \\ 0 & p \end{Lmatrix}},x^{j(p-1)+c}y^{r-j(p-1)-c}\right]
	\\ & \textstyle \qquad
		- a \sum_{j}  (-1)^{j} {{n-\alpha}  \choose j} \left[g,\theta^\alpha x^{j(p-1)+c-\alpha}y^{r-\alpha p-j(p-1)-c}\right] \pmod {p^{{n}}},
	\end{align*}
for all $0\leqslant \alpha \leqslant n$.
\end{lemma}

\begin{proof}
For $0 \leqslant j \leqslant n$, let $\smash{\varphi_{c,j,g} = \left[g,x^{j(p-1)+c}y^{r-j(p-1)-c}\right]}$, and let $\smash{\textstyle\phi_{n,c,g}^{(0)} = \sum_{j}  (-1)^{j} {{n}  \choose j} \varphi_{c,j,g}}$. Then 
	\begin{align*}
		\textstyle (T - a) \phi_{n,c,g}^{(0)} & \textstyle \equiv \sum_{\lambda \ne 0} \sum_{z\geqslant 0} p^z  \sum_{j}  (-1)^{j} {{n}  \choose j} {r - j(p-1)-c \choose z} \left[g  {\begin{Lmatrix} p & [\lambda] \\ 0 & 1 \end{Lmatrix}},(-[\lambda])^{r-c-z} x^{r-z}y^z\right] 
	\\ & \textstyle \qquad
			+ \sum_{j}  (-1)^{j} {{n}  \choose j} p^{r-j(p-1)-c} \left[g{\begin{Lmatrix} p & 0 \\ 0 & 1 \end{Lmatrix}},x^{j(p-1)+c}y^{r-j(p-1)-c}\right]
	\\ & \textstyle \qquad
			+ \sum_{j}  (-1)^{j} {{n}  \choose j} p^{j(p-1)+c} \left[g{\begin{Lmatrix} 1 & 0 \\ 0 & p \end{Lmatrix}},x^{j(p-1)+c}y^{r-j(p-1)-c}\right]
	\\ & \textstyle \qquad
			- a \sum_{j}  (-1)^{j} {{n}  \choose j} \left[g,x^{j(p-1)+c}y^{r-j(p-1)-c}\right] 
	\\ & \textstyle \equiv 
			\sum_{j}  (-1)^{j} {{n}  \choose j} p^{r-j(p-1)-c} \left[g{\begin{Lmatrix} p & 0 \\ 0 & 1 \end{Lmatrix}},x^{j(p-1)+c}y^{r-j(p-1)-c}\right]
	\\ & \textstyle \qquad
			+ \sum_{j}  (-1)^{j} {{n}  \choose j} p^{j(p-1)+c} \left[g{\begin{Lmatrix} 1 & 0 \\ 0 & p \end{Lmatrix}},x^{j(p-1)+c}y^{r-j(p-1)-c}\right]
	\\ & \textstyle \qquad - a \sum_{j}  (-1)^{j} {{n}  \choose j} \left[g,x^{j(p-1)+c}y^{r-j(p-1)-c}\right] \pmod {p^{{n}}}.
	\end{align*}
Note that, for all $0\leqslant \alpha \leqslant n$, 
	\begin{align*}
		& \textstyle a \sum_{j}  (-1)^{j} {{n}  \choose j} \left[g,x^{j(p-1)+c}y^{r-j(p-1)-c}\right] \\& \textstyle \qquad =
		a \sum_{u,w}  (-1)^{u+w} {{n-\alpha}  \choose u} {\alpha \choose w} \left[g,x^{(u+w)(p-1)+c}y^{r-(u+w)(p-1)-c}\right]
		\\ & \textstyle \qquad =
		a \sum_{u}  (-1)^{u} {{n-\alpha}  \choose u}
		\left[g,x^{u(p-1)+c-\alpha}y^{r-\alpha p-u(p-1)-c}\sum_w{ \alpha \choose w} x^{w(p-1)+\alpha}y^{\alpha p-w(p-1)}\right]
		\\ & \textstyle\qquad  =
		a \sum_{u}  (-1)^{u} {{n-\alpha}  \choose u} \left[g,\theta^\alpha x^{u(p-1)+c-\alpha}y^{r-\alpha p-u(p-1)-c}\right].
	\end{align*}
Consequently, $\smash{\phi_{r,m,n,c,g} = \phi_{n,c,g}^{(0)}}$ has the desired properties.
\end{proof}

\begin{lemma}\label{l7xx}
Let $r,m,n,\alpha \geqslant 0$ be integers, and let $a \in \overline {\mathbb Q}_p$ be such that $m+1 > v(a) > m$. Suppose that $n \geqslant m+1$, and $n \geqslant \alpha$, and $r \geqslant np + \alpha$. Then there are some $\textstyle\phi_{r,m,n,\alpha,g}^\ast$ and $\textstyle\phi_{r,m,n,\alpha,g}^{\ast\ast}$ such that
	\begin{align*}
		\textstyle (T - a) \phi_{r,m,n,\alpha,g}^\ast & \textstyle \equiv 
        \sum_{j}  (-1)^{j} {{n}  \choose j} p^{j(p-1)+\alpha} \left[g{\begin{Lmatrix} 1 & 0 \\ 0 & p \end{Lmatrix}},x^{j(p-1)+\alpha}y^{r-j(p-1)-\alpha}\right]
	\\ & \textstyle \qquad - a \sum_{j}  (-1)^{j} {{n-\alpha}  \choose j} \left[g,\theta^\alpha x^{j(p-1)}y^{r-\alpha (p+1)-j(p-1)}\right] \pmod {p^{{n}}},
	\end{align*}
and
	\begin{align*}
		\textstyle (T - a) \phi_{r,m,n,\alpha,g}^{\ast\ast} & \textstyle \equiv 
        \sum_{j}  (-1)^{j} {{n}  \choose j} p^{j(p-1)+\alpha} \left[g{\begin{Lmatrix} p & 0 \\ 0 & 1 \end{Lmatrix}},y^{j(p-1)+\alpha}x^{r-j(p-1)-\alpha}\right]
	\\ & \textstyle \qquad - (-1)^\alpha a \sum_{j}  (-1)^{j} {{n-\alpha}  \choose j} \left[g,\theta^\alpha y^{j(p-1)}x^{r-\alpha (p+1)-j(p-1)}\right] \pmod {p^{{n}}}.
	\end{align*}
\end{lemma}

\begin{proof}
Take  $\smash{\phi_{r,m,n,\alpha,g}^{\ast} = \phi_{r,m,n,\alpha,g}}$ and $\smash{\phi_{r,m,n,\alpha,g}^{\ast\ast} = \phi_{r,m,n,\alpha,g{\begin{Smatrix} 0 & 1 \\ 1 & 0 \end{Smatrix}}}}$.
\end{proof}

\begin{proof}[Proof of lemma~\ref{l4.1}]
First, we will show that if $r,m\geqslant 0$ are such that $r \geqslant (m+1)(n+1)$, and if $m+1>v(a)\geqslant m$, then $X(r+2,a)$ contains the induction of   $\smash{\overline\theta^{m+1} \sigma_{r-(m+1)(p+1)}}$. Then, we will show that if $r,m\geqslant 0$ are such that $r \geqslant (m+1)(n+1)$, and if $m+1\geqslant v(a)>m$, then $X(r+2,a)$ contains the induction of the  subrepresentation of $\sigma_r$ generated by  $\{y^r,xy^{r-1},\ldots,x^my^{r-m}\}$. These two results will imply the two claims stated in lemma~\ref{l4.1}. For the first part, note that, for any $f \in \overline{\mathbb Z}_p[x,y]$ of degree $r-(m+1)(p+1)$, we have $\smash{[1,\theta^{m+1}f] = \sum_j \lambda_j\phi_{m+1,c_j,1}^{(0)}}$ for some $\lambda_j$ and some $r-m-1\geqslant c_j \geqslant m+1$, so
	\begin{align*}
		\textstyle (T - a) \left[1,\theta^{m+1}f\right] & \textstyle = (T - a) \sum_j \lambda_j\phi_{m+1,c_j,1}^{(0)}
        \equiv -a\left[1,\theta^{m+1}f\right]  \pmod{p^{m+1}}.
	\end{align*}
Consequently, $\smash{\overline{(T - a) \left[1,-a^{-1}\theta^{m+1}f\right]}} = \left[1,\overline{\theta}^{m+1}f\right]$ is in the kernel of reduction. For the second part, note that, for $0\leqslant \alpha\leqslant m$, 
	\begin{align*}
		\textstyle (T - a) \left(p^{-\alpha}\phi_{r,m,m+1,\alpha,{\begin{Smatrix} p & 0 \\ 0 & 1 \end{Smatrix}}}^\ast\right) & \textstyle 
        \equiv \left[1,x^{\alpha}y^{r-\alpha}\right]  \pmod{p^{v(a)-m}}.
	\end{align*}
Consequently, $\smash{\left[1,x^{\alpha}y^{r-\alpha}\right]}$  is in the kernel of reduction, for all $0\leqslant \alpha\leqslant m$.
\end{proof}

The subrepresentation $\Ker\Psi$ is isomorphic to $W_r$, where $W_r$ is the subrepresentation of $\sigma_r$ generated by $\smash{\overline \theta \sigma_{r - (p+1)} }$ and $\{y^r, xy^{r-1}, \ldots, x^{s-1}y^{r-s+1}\}$. We showed this in the proof of lemma~\ref{l1}, where we stated an explicit matrix representation of $\Psi$. In fact, $W_r$ is generated by $\smash{\overline \theta \sigma_{r - (p+1)}}$ and $y^r$ only; indeed, $\smash{W_r' = \langle\overline \theta \sigma_{r - (p+1)},y^r\rangle_K \subseteq W_r}$, and $\sigma_{r} /W_r \cong \sigma_{p-1-s} $ since $\Psi$ is surjective, and also $\sigma_r /W_r' \cong \sigma_{p-1-s}(s)$ due to corollary~5.1 in~\cite{b2}. Consequently,
	\[\textstyle \smash{I(\Ker \Psi) = I\left(\langle\overline \theta \sigma_{r - (p+1)} ,y^r\rangle_{\mathrm{GL}_2(\mathbb F_p)}\right)}.\]
The following lemma gives a linear basis for $\langle y^r,xy^{r-1},\ldots,x^my^{r-m} \rangle_{\mathrm{GL}_2(\mathbb F_p)} \subseteq \sigma_r$.

\begin{lemma}\label{lemma5}
Let $r \geqslant (m+1)(p + 1)$, and let $p > m+1$. Then the subrepresentation $S_m$ of $\sigma_r$ generated by $\{y^r,xy^{r-1},\ldots,x^my^{r-m}\}$ is linearly spanned by
	\begin{align*}
	& \textstyle \{y^r,xy^{r-1},\ldots,x^my^{r-m}\} \cup \{x^r,x^{r-1}y\ldots,x^{r-m}y^m\}
	\\ & \textstyle\qquad
	\cup\left\{\smash{\sum_{r-m>s(p-1)+i>m} {r-m \choose s(p-1) + j} x^{s(p-1) + i} y^{r-s(p-1) - i}}\,\big|\, i,j \in \mathbb Z \text{ with } i - j \in \{0,\ldots,m\}\right\}
	\\ & \textstyle\qquad
	\cup\left\{\smash{\sum_{r-m>s(p-1)+i>m} {r-m \choose s(p-1) + j} y^{s(p-1) + i} x^{r-s(p-1) - i}}\,\big|\, i,j \in \mathbb Z \text{ with } i - j \in \{0,\ldots,m\}\right\}.
	\end{align*}
\end{lemma}

\begin{proof}
Let $S_m'$ denote the subrepresentation of $\sigma_r$ generated by $x^my^{r-m}$. Then $S_m'$ is a subrepresentation of $S_m$. Moreover, $\smash{\sum_{\mu\in\mathbb F_p^\times} \lambda_\mu(x+\mu y)^my^{r-m} = \sum_{i=0}^{m} \nu_i x^iy^{m-i}}$ is in $S_m'$, where \[\textstyle \smash{\nu_i = \sum_{\mu\in\mathbb F_p^\times} \lambda_\mu \mu^i = \sum_{s=0}^{p-1} \lambda_s t^{is}} = f(t^i),\] where $t$ is a generator for $\smash{\mathbb F_p^\times}$. Since the number of $\nu_i$ is $m+1 \leqslant  p-1 = \deg f$, the coefficients of $f$ can be chosen in a way that $\nu_i = \delta_{i=\alpha}$. Consequently, $x^\alpha y^{r-\alpha}$ is in $S_m'$, for all $0 \leqslant \alpha \leqslant m$, which implies that $S_m' = S_m$. So $S_m$ is equal to the linear span of $(ax+cy)^m(bx+dy)^{r-m}$. Since either $ad \ne 0$ or $bc \ne 0$, then $S_m$ is equal to the linear span of $(x+ay)^m(bx+y)^{r-m}$ and $(x+ay)^m(bx+y)^{r-m}{\begin{Smatrix} 0 & 1 \\ 1 & 0 \end{Smatrix}}$. By a similar argument as in the previous paragraph, the linear span of $(x+ay)^m(bx+y)^{r-m}$ is the same as the linear span of $x^iy^{m-i}(bx+y)^{r-m}$, for $i \in \{0,\ldots,m\}$ and $\smash{b \in \mathbb F_p}$. By another application of this argument, the linear span of $x^iy^{m-i}(bx+y)^{r-m}$ is the same as the linear span of $y^r,xy^{r-1},\ldots,x^my^{r-m}$, and $\textstyle \smash{\sum_{s} {r-m\choose s(p-1)+j} x^{s(p-1)+i} y^{r-s(p-1)-i}}$, for $i,j \in \mathbb Z$ such that $i - j \in \{0,\ldots,m\}$.  Similarly, the linear span of  $(x+ay)^m(bx+y)^{r-m}{\begin{Smatrix} 0 & 1 \\ 1 & 0 \end{Smatrix}}$ is the same as the linear span of $x^r,yx^{r-1},\ldots,y^mx^{r-m}$, and $\textstyle \smash{\sum_{s} {r-m\choose s(p-1)+j} y^{s(p-1)+i} x^{r-s(p-1)-i}}$, for $i,j \in \mathbb Z$ such that $i - j \in \{0,\ldots,m\}$. The proof of the lemma can be completed by combining these two results.
\end{proof}

The module $I_h$ is defined in \cite{b1} for each integer $h$, as the module of $\mathbb F_p$-valued functions on $\mathbb F_p^2$ which are homogeneous of degree $h$ and which vanish at the origin. By definition, $\smash{I_h = I_{\widetilde h}}$, where $\smash{\widetilde h}$ is the integer in $\{1,\ldots,p-1\}$ which is congruent to $h$ modulo~$p-1$.

\begin{lemma}\label{lemma6}
Whenever $p > h > 0$ and $t \in \mathbb Z$, the module $I_h(t)$ is semi-simple if and only if $h=p-1$. If $p-1>h>0$, then the only possible quotients of $I_h(t)$ are $I_h(t),\sigma_{p-1-h}(t+h)$, and the trivial quotient. If $h=p-1$, then $I_{h}(t) = I_{p-1}(t) = \sigma_{p-1}(t) \oplus \sigma_{0}(t)$, so there is the additional possible quotient $\sigma_{0}(t)$.
\end{lemma}

\begin{proof}
Due to parts~(a) and~(c) of lemma~3.2 in~\cite{b1}, $I_h(t)$ has a series 
	\[\smash{\textstyle I_h(t) = N_2 \supseteq N_1 \supseteq N_0 = \{0\}},\]
whose factors are $\textstyle \smash{N_2/N_1 \cong \sigma_{p-1-h}(t+h)}$ and $\textstyle \smash{N_1 \cong \sigma_{h}(t)}$. Then $I_h(t)$ is semi-simple if and only if $I_h(t) = N_1 \oplus M$, where $M \cong \sigma_{p-1-h}(t+h)$. Since $M$ inherits the action of $I_h(t)$, this is possible only when $h = p-1$, in which case indeed $I_h(t) = I_{p-1}(t) = \sigma_{p-1}(t) \oplus \sigma_{0}(t)$. Hence, if $p-1>h>0$ then the only possible quotients of $I_h(t)$ are $I_h(t),\sigma_{p-1-h}(t+h)$, and the trivial quotient, and if $h=p-1$ then $I_{h}(t) = I_{p-1}(t) = \sigma_{p-1}(t) \oplus \sigma_{0}(t)$ so there is the additional possible quotient $\sigma_{0}(t)$.
\end{proof}

\begin{lemma}\label{l4.1LL}
Let $r \geqslant (m+1)(p+1)$, and suppose that $a \in \overline {\mathbb Q}_p$ is such that $m+1>v(a)>m$. Then $\smash{\overline \Theta_{r+2,a}}$ is a quotient of $\textstyle \sigma_r/\smash{\langle \overline\theta^{m+1} \sigma_{r-(m+1)(p+1)},y^r,xy^{r-1},\ldots,x^my^{r-m}\rangle_{\mathrm{GL}_2(\mathbb F_p)}}$. This module has a series whose factors are $\textstyle \smash{I_r,I_{r-2}(1),\ldots,I_{r-2m}(m)}$, and consequently $\smash{\overline \Theta_{r+2,a}}$ has a series whose factors are $M_0,\ldots,M_{m}$, where $M_\alpha$ is a quotient of a submodule of $\smash{I_{r-2\alpha}(\alpha)}$, for each $0 \leqslant \alpha \leqslant m$.
\end{lemma}

\begin{proof}
Follows from the facts that $\textstyle \sigma_r/\smash{\overline\theta^{m+1} \sigma_{r-(m+1)(p+1)}}$ has a series whose factors are \[N_\alpha = \textstyle \smash{\overline\theta^{\alpha}  \sigma_{r-\alpha(p+1)}/\overline\theta^{\alpha+1} \sigma_{r-(\alpha+1)(p+1)} \cong I_{r-2\alpha}(\alpha)},\] for $0 \leqslant \alpha \leqslant m$, and each $N_\alpha$ has a submodule $N'_\alpha\subseteq N_\alpha$ such that 
	\begin{align*}
		N_\alpha' & \cong \smash{\sigma_{\widetilde{r-2\alpha}}(\alpha)},
		\\ N_\alpha/N_\alpha' & \cong  \smash{\sigma_{(2\alpha-r \bmod p-1)}(r-\alpha)},
	\end{align*}
due to parts~(a) and~(c) of lemma~3.2 in~\cite{b1}.
\end{proof}

\subsection[Identities involving binomial sums]{Identities involving binomial sums}

\begin{lemma}\label{lemma7}
Let $t \geqslant 0$ be an integer, and suppose that $r = t(p-1) + s$, with $s \in \{1,\ldots,p-1\}$. Then:
\begin{enumerate}[1.]
	\item\label{423423432} $\textstyle \smash{\sum_{j=1}^t {r \choose j(p-1)} \equiv \frac{t}{s}p \pmod{p^2}}$;
	\item \[\textstyle \smash{\sum_{j=1}^t {r \choose j(p-1)} \equiv \frac{t}{s}p + p^2\left(tA_s + t^2B_s\right) \pmod{p^3}},\] where $A_s$ and $B_s$ depend only on $s$;
	\item If $s \ne 1$, then $\textstyle \smash{\sum_{j=1}^t j{r \choose j(p-1)} \equiv 0 \pmod{p}}$;
	\item If $s\ne 1$ and if $p\mid t$, then $\textstyle \smash{\sum_{j=1}^t p^{-1}j{r \choose j(p-1)} \equiv 0 \pmod{p}}$;
	\item If $A \geqslant 0$, then $\textstyle \smash{\sum_{j} {r \choose j(p-1)+A} \equiv {s \choose A \bmod p-1} (1 + \delta_{s=p-1}\delta_{A \equiv_{p-1} 0})\pmod{p}}$;
	\item If $A \in \mathbb Z$, and $R \geqslant 0$, and $M_{R,A} = \smash{\sum_{j} {R \choose j(p-1)+A}}$, then
	\[\textstyle M_{R,A} = 
			\left\{\begin{array}{rl} \sum_{i=0}^{R-1} {R-1-i \choose A-1} M_{i,0} & \text{if } A > 0, \\ 
            \sum_{i=0}^{-A} (-1)^i {-A \choose i} M_{R-i,0} & \text{if } A < 0. \end{array}\right.\]
\end{enumerate}
\end{lemma}

\begin{proof}
\begin{enumerate}[1.]
\item
For all $r \geqslant 1$, define $M_r$ by the equation $\textstyle M_r = \smash{\sum_{j=1}^t {r \choose j(p-1)}}$. We want to show that $M_r = \frac{t}{s} p + O(p^2)$, where the valuation of the expression $O(p^2)$ is at least~$2$. We are going to prove this by induction. Note that $M_r = 0$ for $1 \leqslant i \leqslant p -1$, and $M_p \textstyle\smash{  = {p \choose p-1} = \frac 11p}$,
so the claim holds true when $1 \leqslant r \leqslant p$. Consequently, the base case of the induction holds true. Moreover, note that
	\begin{align*}
	M_r & = \textstyle (p-1)^{-1} \sum_{\mu \in \mathbb F_p^\times} (1 + [\mu])^r - 1 - \delta_{s=p-1},
	\end{align*}
for all $r \geqslant 1$. Let $f$ be the polynomial of degree $p-2$ whose roots are the numbers $1 + [\mu]$, so that $\textstyle \smash{ f(x) = \frac{1}{x}((x-1)^{p-1} - 1)= a_{p-2}x^{p-2} + a_{p-3}x^{p-3} + \cdots + a_0}$, and the coefficients of $f$ are given by $\smash{a_{p-1-m} = (-1)^{m-1} {p-1 \choose m-1}}$. Then $N_r = M_r + 1 + \delta_{s=p-1}$ must satisfy the recurrence
$\textstyle N_{r+p-2} = -a_{p-3} N_{r+p-3}- \cdots  -a_0 N_r$, for all $r \geqslant 1$. As $a_{p-3} + \cdots + a_0 = f(1)-1 = -2$, then 
	\[\textstyle M_{r+p-2} - 1 + (-1)^{s-1} {p-1 \choose s-1} = {p-1 \choose 1} M_{r+p-3} - \cdots + {p-1 \choose p-2} M_r,\]
for all $r \geqslant 1$. Note that $\smash{(-1)^{s-1} {p-1 \choose s-1} = \frac{(s-1)!}{(s-1)!} - \frac{(s-1)!}{(s-1)!} \left(1^{-1} + \cdots + (s-1)^{-1}\right) p + O(p^2)}$. Define $M_r' = \frac{t}{s}$. Then
	\begin{align*}
	\textstyle M_{r+p-2} & \textstyle = \left(-M_{r+p-3}' - \cdots - M_r' + \frac{1}{1} + \cdots + \frac{1}{s-1}\right)p + O(p^2) \textstyle \\ & \textstyle =  \left(-M_{r+p-3}' - \cdots - M_r' -\frac{1}{s} - \cdots - \frac{1}{p-1}\right)p + O(p^2).
	\end{align*}
If $s = 1$, then, by the induction hypothesis,
	\begin{align*}
	\textstyle M_{r+p-2} & \textstyle =  \left(-M_{r+p-3}' - \cdots - M_r' -\frac{1}{s} - \cdots - \frac{1}{p-1}\right)p + O(p^2)
	\\ & \textstyle =
	\left(-\frac{t}{p-2} \cdots - \frac{t}{1} - \frac{1}{1} - \cdots - \frac{1}{p-1}\right)p + O(p^2)
	\textstyle =
	\frac{t}{p-1} p + O(p^2).
	\end{align*}
If $s > 1$, then, by the induction hypothesis,
	\begin{align*}
	\textstyle M_{r+p-2} 
	& \textstyle =
	\left(-M_{r+p-3}' - \cdots - M_r' -\frac{1}{s} - \cdots - \frac{1}{p-1}\right)p + O(p^2)
	\\ & \textstyle =
	\left(-\frac{t+1}{s-2}- \cdots  - \frac{t+1}{1} - \frac{t}{p-1} \cdots - \frac{t}{s+1} - \frac{t}{s} -\frac{1}{s} - \cdots - \frac{1}{p-1}\right)p + O(p^2)
	\\ & \textstyle =
	\left(-\frac{t+1}{s-2}- \cdots  - \frac{t+1}{1} - \frac{t+1}{p-1} \cdots - \frac{t+1}{s+1} - \frac{t+1}{s}\right)p + O(p^2) \textstyle = \frac{t+1}{s-1}p + O(p^2).
	\end{align*}
This completes the proof by induction that $M_r = \frac{t}{s} p + O(p^2)$, for all $r \geqslant 1$, which is equivalent to the first claim stated in the lemma.
\item
Suppose that $\alpha_r$ is such that $\textstyle \smash{M_r \equiv \frac{t}{s}p + p^2\alpha_r \pmod{p^3}}$. These constants are unique up to addition of an $O(p)$ term, in the sense that $\alpha_r'$ is such that $\textstyle \smash{M_r \equiv \frac{t}{s}p + p^2\alpha_r' \pmod{p^3}}$ if and only if $\alpha_r' = \alpha_r + O(p)$. We know that there is the recurrence relation
	\[\textstyle M_{r+p-2} - 1 + (-1)^{s-1} {p-1 \choose s-1} = {p-1 \choose 1} M_{r+p-3} - \cdots + {p-1 \choose p-2} M_r,\]
for all $r \geqslant 1$. Consequently,
	\begin{align*}
	& \textstyle \frac{t+1}{s-1} p + p^2 (\alpha_{r+p-2} + \cdots + \alpha_{r})
	\\ & \textstyle \qquad = p\left(\frac11 + \cdots + \frac{1}{s-1}\right) - p^2\sum_{1\leqslant i < j\leqslant s-1} \frac{1}{ij}
	\\ & \textstyle \qquad\qquad
	- \sum_{i=1}^{s-2}  \frac {t+1}{s-1-i} \left(p - p^2\sum_{1\leqslant k\leqslant i}\frac 1k\right)
	- \sum_{i=s-1}^{p-2}  \frac {t}{p+s-2-i} \left(p - p^2\sum_{1\leqslant k\leqslant i}\frac 1k\right) + O(p^3)
	\\ & \textstyle \qquad = 
	\frac{t+1}{s-1}p -pt\sum_{k=1}^{p-1}\frac1k - p^2\sum_{1\leqslant i < j\leqslant s-1} \frac{1}{ij}
	+p^2 \sum_{i=1}^{s-2}  \frac {t+1}{s-1-i} \sum_{1\leqslant k\leqslant i}\frac 1k
	\\ & \textstyle \qquad\qquad 
	+p^2\sum_{i=s-1}^{p-2}  \frac {t}{p+s-2-i}\sum_{1\leqslant k\leqslant i}\frac 1k + O(p^3)
	\\ & \textstyle \qquad = \frac{t+1}{s-1}p + p^2 (A_{s}' + tB_{s}') + O(p^3),
	\end{align*}
where
	\begin{align*}
	\textstyle A_{s}'
	& \textstyle = -\sum_{1\leqslant i < j\leqslant s-1} \frac{1}{ij}+  \sum_{i=1}^{s-2}  \frac {1}{s-1-i} \sum_{1\leqslant k\leqslant i}\frac 1k,
	\\ \textstyle B_{s}'
	& \textstyle = -\frac{1}{p}\sum_{k=1}^{p-1}\frac1k +  \sum_{i=1}^{s-2}  \frac {1}{s-1-i} \sum_{1\leqslant k\leqslant i}\frac 1k  +\sum_{i=s-1}^{p-2}  \frac {1}{p+s-2-i}\sum_{1\leqslant k\leqslant i}\frac 1k.
	\end{align*}
Therefore, $\alpha_{r+p-1}-\alpha_r = A_s'' + tB_s'' + O(p)$. %
 Hence there is the recurrence relation
	\[\textstyle \alpha_{r+3(p-1)}-3\alpha_{r+2(p-1)} + 3\alpha_{r+(p-1)}-\alpha_r = O(p).\]
This implies that there are constants $A_{s},B_{s},\gamma_s$, such that 
	\[\textstyle \alpha_r = \alpha_{t(p-1)+s} = \gamma_s + tA_s + t^2B_s + O(p).\]
Since $M_{1} = \cdots = M_{p-1} = 0$ and $M_p = p$, we have $\alpha_1 = \cdots = \alpha_{p} = O(p)$, which implies that $\gamma_s = 0$. This completes the proof of the second part of the lemma.
\item
For all $r \geqslant 1$, define $L_r$ by $\textstyle L_r = \smash{\sum_{j=1}^t j {r \choose j(p-1)}}$. We want to show that $p \mid L_r$, when $s \ne 1$. Note that
	\begin{align*}
	\textstyle - r \sum_{u \in \mathbb F_p^\times} u (1 + u)^{r-1} & = \textstyle - r \sum_{m=0}^{r-1} {r-1 \choose m} \sum_{u \in \mathbb F_p^\times} u^{m+1} 
	\\ & = 
	\textstyle -  \sum_{m=1}^{r} m{r \choose m} \sum_{u \in \mathbb F_p^\times} u^{m} = L_r + r \delta_{s=p-1},
	\end{align*}
in $\mathbb F_p$. Consequently,
	\begin{align*}
	\textstyle  L_r & = \textstyle - r \sum_{u \in \mathbb F_p^\times} u (1 + u)^{r-1} - r \delta_{s=p-1} \\&\textstyle = - r \left(\sum_{u \in \mathbb F_p^\times} u (1 + u)^{s-1} + \delta_{s=p-1}\right) = \frac rs L_s,
	\end{align*}
in $\mathbb F_p$. As $L_s = \text{empty sum} = 0$, then $L_r = 0$ in $\mathbb F_p$ as well, when $s \ne 1$.
\item
By the first part of this proof, the claim we want to show is equivalent to the congruence $\textstyle \sum_{j=1}^t j(p-1){r \choose j(p-1)} \equiv 0 \pmod{p^2}$. Note that
	\begin{align*}
	\textstyle  \textstyle \sum_{j=1}^t j(p-1){r \choose j(p-1)} &  \textstyle \equiv \textstyle r \sum_{j=1}^t {r-1 \choose j(p-1)-1}\\ & \equiv \textstyle r \sum_{j=1}^t {r \choose j(p-1)} - r \sum_{j=1}^t {r-1 \choose j(p-1)}  \equiv 0 \pmod {p^2},
	\end{align*}
the last part being true due to the second part of this lemma and the fact that $s > 1$.
\item
Let $r = t(p-1)+s$, with $t \geqslant 0$ and $s \in \{1,\ldots,p-1\}$, and let $\nu = A \bmod p-1$. Then, in $\mathbb F_p$,
	\[\textstyle \smash{\sum_{j} {r \choose j(p-1)+A} = \textstyle - \sum_{u \in \mathbb F_p^\times} u^{-\nu} (1 + u)^r = - \sum_{u \in \mathbb F_p^\times} u^{-\nu} (1 + u)^s = {s \choose \nu} (1 + \delta_{s = p -1, \nu = 0})}.\]
\item
Follows from a repeated application of the identity $M_{R,B} = M_{R-1,B} + M_{R-1,B-1}$, which holds true whenever $R,B \geqslant 0$. 
\end{enumerate}
\end{proof}

\begin{lemma}\label{l7}
Let $r,L,b,N \geqslant 0$ be integers, and suppose that $r \geqslant (L+b)N$. Then:
\begin{enumerate}[1.]
	\item $\textstyle \sum_{j} (-1)^{j-b} {L \choose j-b} {r - jN \choose u} = \delta_{u=L}  N^L$, for all $0 \leqslant u \leqslant L$.
\end{enumerate}
\end{lemma}

\begin{proof}
Let $t = \lfloor r/N\rfloor -b$. For all $u,L \geqslant 0$, define $M_{u,L}$ by the equation
	\[\textstyle \smash{\textstyle M_{u,L} =  \sum_{j} (-1)^{j-b} {L \choose j-b} {r - jN \choose u}},\]
and let $f(x,y)$ denote the polynomial $\smash{\textstyle f(x,y) = \sum_{r\geqslant u\geqslant 0, t\geqslant L\geqslant 0} M_{u,L} x^L y^u}$. Then
	\begin{align*}
	f(x,y) & = \textstyle \sum_{r\geqslant u\geqslant 0, t\geqslant L\geqslant 0} M_{u,L} x^L y^u
	= \textstyle \sum_{j\geqslant 0,r\geqslant u\geqslant 0, t\geqslant L\geqslant 0} x^L   {L \choose j-b} (-1)^{j-b} {r - jN \choose u} y^u 
	\\ & = \textstyle
	\sum_{j\geqslant 0, t\geqslant L\geqslant 0} x^L  {L \choose j-b}(-1)^{j-b} (1+y)^{r - jN} 
	\\ & = \textstyle
	(1+y)^{r - (L+b)N} \sum_{j\geqslant 0, t\geqslant L\geqslant 0} x^L  {L \choose j-b}(-1)^{j-b} (1+y)^{(L - (j-b))N} 
	\\ & = \textstyle
	(1+y)^{r - (L+b)N} \sum_{t\geqslant L\geqslant 0} x^L ((1+y)^{N}-1)^L 
	\\ & = \textstyle
	(1 - x((1+y)^{N}-1))^{-1}(1+y)^{r- (L+b)N} (1 - (x((1+y)^{N}-1))^{t+1}) %
	\\ & = \textstyle
	(1 - Nxy-xy^2h_3(y))^{-1}(1+h_1(y)) (1-N^{t+1}x^{t+1}y^{t+1} - x^{t+1}y^{t+2}h_4(y)),
	\end{align*}
where the $h_i$ are polynomials, so
	\begin{align*}
	\lim_{\epsilon\to 0} f(\epsilon^{-1}x,\epsilon)
	& =  \textstyle  \lim_{\epsilon\to 0}  (1  - Nx + O(\epsilon))^{-1} (1+O(\epsilon)) (1 - N^{t+1}x^{t+1} + O(\epsilon))
	\\ & = \textstyle 
	\lim_{\epsilon\to 0} (1 - Nx + O(\epsilon))^{-1} (1 - N^{t+1}x^{t+1} + O(\epsilon))
    \\ & \textstyle  =
    (1 - Nx)^{-1} (1 - N^{t+1}x^{t+1}) = \sum_{j=0}^t   N^jx^j.
	\end{align*}
Since this is still a polynomial in $x$, it follows that $M_{u,L} = 0$ whenever $0\leqslant u < L$. Moreover,
	\begin{align*}
	[x^Ly^L] f(x,y)  & =  \textstyle [x^L] \lim_{\epsilon\to 0} f(\epsilon^{-1}x,\epsilon) = N^L,
	\end{align*}
whenever $0\leqslant L \leqslant t$.
\end{proof}

\begin{lemma}\label{lemma742}
\begin{enumerate}[1.]
	\item Let $A,z,w \geqslant 0$ be integers. Then $\textstyle \smash{{z \choose w} = \sum_{v} (-1)^{w-v} {A+w-v-1 \choose w-v} {z+A \choose v}}$.
	\item Let $A,i,w \geqslant 0$ be integers. Then $\textstyle \smash{{i \choose w} \equiv \sum_{v} (-1)^{v} {A-v \choose w-v} {i(p-1)+A \choose v} \pmod{p}}$.
\end{enumerate}
\end{lemma}

\begin{proof}
\begin{enumerate}[1.]
\item Let $L_z = \smash{\sum_w X^w \sum_{v} (-1)^{w-v} {A+w-v-1 \choose w-v} {z+A \choose v}}$. Then
	\begin{align*}
	\textstyle L_z & = \textstyle \smash{\sum_v  (-1)^v {z+A \choose v} \sum_{w\geqslant v} X^w (-1)^{w} {A+w-v-1 \choose w-v}  = \sum_v  X^v {z+A \choose v} \sum_{u} (-X)^{u} {A+u-1 \choose u}} \\& = \textstyle \smash{(1 + X)^{z+A} (1-(-X))^{-A} = (1+X)^z = \sum_{w} X^w {z \choose w}}.
	\end{align*}
Consequently, $\smash{{z \choose w} = [X^w] \sum_{w} X^w {z \choose w} = [X^w] L_z = \sum_{v} (-1)^{w-v} {A+w-v-1 \choose w-v} {z+A \choose v}}$.
\item 
Let $L_z' = \smash{\sum_w X^w \sum_{v} (-1)^{v} {A-v \choose w-v} {A-i \choose v}}$. Then
	\begin{align*}
	\textstyle L_z' & = \textstyle \smash{\sum_v  (-1)^v {A-i \choose v} \sum_{w\geqslant v} X^w  {A-v \choose w-v}  = \sum_v  (-X)^v {A-i \choose v} \sum_{u} X^{u} {A-v \choose u}} \\& = \textstyle \smash{(1+X)^{i} \sum_v   {A-i \choose v}  (-X)^v(1+X)^{A-i-v} = (1+X)^i = \sum_{w} X^w {i \choose w}}.
	\end{align*}
Consequently, $\smash{{i \choose w} = [X^w] \sum_{w} X^w {i \choose w} = [X^w] L_z' =  \sum_{v} (-1)^{v} {A-v \choose w-v} {A-i \choose v} }$.
\end{enumerate}
\end{proof}

\begin{lemma}\label{lemma8}
Let $r,\alpha,\beta,\gamma\geqslant 0$ be integers, and suppose that $m \geqslant \alpha(p+1) + \beta(p-1) + \gamma$, and $\beta \geqslant \alpha > 0$, and $C_j \in \mathbb Z$. Then
	\[\textstyle \smash{f = \sum_{j=0}^\beta C_j x^{\alpha+\gamma+j(p-1)}y^{r-\alpha-\gamma-j(p-1)} = (-1)^\alpha \sum_{j=0}^{\beta-\alpha} C_j^{(\alpha)} \theta^\alpha x^{\gamma+j(p-1)}y^{r-\alpha(p+1)-\gamma-j(p-1)}},\]
for some $C_j^{(\alpha)} \in \mathbb Z$, if and only if $\sum_{j=0}^\beta C_j {j \choose w} = 0$ for all $0 \leqslant w < \alpha$. Moreover, in that case $C_{\beta-\alpha}' =  C_{\beta}$, and $C_0' = (-1)^\alpha C_0$, and 
	\[\textstyle \smash{C_j^{(\alpha)} = (-1)^\alpha\sum_{i=j+1}^\beta {i - j - 1 \choose \alpha - 1} C_i},\]
and consequently
	\[\textstyle \smash{\sum_{j=0}^{\beta-\alpha} C_j^{(\alpha)} = (-1)^\alpha\sum_{j=0}^\beta C_j {j \choose \alpha}}.\]
\end{lemma}

\begin{proof}
We will prove this by induction on $\alpha > 0$. In the base case $\alpha = 1$, the induction hypothesis is equivalent to the fact that
	\[\textstyle \smash{f = \sum_{j=0}^\beta C_j x^{1+\gamma+j(p-1)}y^{r-1-\gamma-j(p-1)} = -\sum_{j=0}^{\beta-1} C_j^{(1)} \theta x^{\gamma+j(p-1)}y^{r-(p+1)-\gamma-j(p-1)}},\]
for some  $C_j^{(1)} \in \mathbb Z$, if and only if $\smash{\sum_{j=0}^\beta C_j = 0}$, and $\smash{C_j^{(1)} = -\sum_{i=j+1}^\beta C_i}$. Indeed, if the right side is expanded into the left side, then $\smash{\sum_{j=0}^\beta C_j = 0}$ and $\smash{C_j^{(1)} = -\sum_{i=j+1}^\beta C_i}$, and conversely if $\smash{\sum_{j=0}^\beta C_j = 0}$ then the left side can be grouped into the right side, and $\smash{C_j^{(1)} = -\sum_{i=j+1}^\beta C_i}$ due to  uniqueness of the coefficients. Suppose that the induction hypothesis holds true for all $\alpha' < \alpha$. Then
	\[\textstyle \smash{f = \sum_{j=0}^\beta C_j x^{\alpha+\gamma+j(p-1)}y^{r-\alpha-\gamma-j(p-1)} = (-1)^\alpha \sum_{j=0}^{\beta-\alpha} C_j^{(\alpha)} \theta^\alpha x^{\gamma+j(p-1)}y^{r-\alpha(p+1)-\gamma-j(p-1)}},\]
for some $C_j^{(\alpha)} \in \mathbb Z$, if and only if the following four conditions hold true:
	\begin{align*}
	\textstyle \sum_{j=0}^\beta C_j {j \choose w} & \textstyle= 0 \text{ for all } 0 \leqslant w < \alpha-1, \\
	\textstyle\smash{C_j^{(\alpha-1)}} & \textstyle= -(-1)^\alpha\sum_{i=j+1}^\beta {i - j - 1 \choose \alpha - 2} C_i, \\
	\textstyle\smash{\sum_{j=0}^\beta C_j^{(\alpha-1)} }& \textstyle= 0, 
	\\ \textstyle\smash{C_j^{(\alpha)}} &\textstyle = -\sum_{i=j+1}^\beta C_i^{(\alpha-1)}.
	\end{align*}
These conditions can be combined into the folowing two conditions: 
	\begin{align*}
	\textstyle \sum_{j=0}^\beta C_j {j \choose w} & \textstyle= 0 \text{ for all } 0 \leqslant w < \alpha, \\
	\textstyle C_j^{(\alpha)}  & \textstyle= (-1)^\alpha\sum_{s=j+2}^\beta C_s \sum_{i\geqslant j+1}  {s - i - 1 \choose \alpha - 2}
	\\ & \textstyle \qquad = (-1)^\alpha\sum_{s=j+2}^\beta  {s - j - 1 \choose \alpha - 1} C_s = (-1)^\alpha\sum_{s=j+1}^\beta  {s - j - 1 \choose \alpha - 1} C_s.
	\end{align*}
This completes the proof by induction. The facts that $C_{\beta-\alpha}' =  C_{\beta}$ and $C_0' = (-1)^\alpha C_0$ can be deduced by comparing the top and bottom coefficients in the two expressions of $f$.
\end{proof}

\begin{lemma}\label{lemma7wefwfwerfwf}
Let $r,m,L,l,w\geqslant 0$ be integers. Suppose that $r \geqslant (m+1)(p + 1)$, and $r \equiv_{p-1} 2L$, and $1 \leqslant L \leqslant m$, and $p > m+1$. Then
	\begin{align*}
	\textstyle  \sum_{i} {r - m + l \choose i(p-1) + l} {(p-1)i \choose w} & \textstyle \equiv \sum_{v=0}^{w} \kappa_v {r - m + l \choose v} \pmod{p},
	\end{align*}
where $\smash{\kappa_v = (-1)^{w-v} {l+w-v-1 \choose w-v}  \eta(2L-m+l-v,l-v) }$, where
	\begin{align*}
	\textstyle \eta(X,Y) = 
	\left\{\begin{array}{rl}{X \choose Y} & \text{if } X \geqslant 1, \\
    (1+\delta_{X = Y = 0}) {X-1 \choose Y}& \text{if } X < 1 \text{ and }  Y \geqslant 0, \\
    {X-1 \choose X-Y} & \text{if } X < 1 \text{ and }  Y < 0.  \end{array}\right.
	\end{align*}
Similarly,
	\begin{align*}
	\textstyle  \sum_{i} {r - m + l \choose i(p-1) + l} {i \choose w} & \textstyle \equiv \sum_{v=0}^{w} \kappa_v' {r - m + l \choose v} \pmod{p},
	\end{align*}
where $\smash{\kappa_v' = (-1)^{v} {l-v \choose w-v}  \eta(2L-m+l-v,l-v) }$.
\end{lemma}

\begin{proof}
Note that $\textstyle \smash{{i(p-1) \choose w} = \sum_{v} (-1)^{w-v}  {l+w-v-1 \choose w-v} {i(p-1)+l \choose v}}$, due to lemma~\ref{lemma7}. Consequently,
	\begin{align*}
	\textstyle \sum_i{r-m+l \choose i(p-1) +l} {i(p-1) \choose w} & \textstyle =\smash{ \sum_i{r-m+l \choose i(p-1) +l} \sum_{v} (-1)^{w-v}  {l+w-v-1 \choose w-v} {i(p-1)+l \choose v}}
	\\[5pt] &  \textstyle  =\smash{ \sum_{v}  (-1)^{w-v} {l+w-v-1 \choose w-v}{r-m+l \choose v} \sum_i {r-m+l-v \choose i(p-1) + l-v}   }
	\\[5pt] &  \textstyle \equiv \smash{ \sum_{v}  (-1)^{w-v} {l+w-v-1 \choose w-v}{r-m+l \choose v} (1+\delta_{2L+l=m+v}\delta_{l=v}) {\smash{\widetilde {2L-m+l-v}} \choose l-v \bmod p-1}}
	\\[5pt] &  \textstyle \equiv \smash{ \sum_{v}  (-1)^{w-v} {l+w-v-1 \choose w-v}{r-m+l \choose v}  \eta(2L-m+l-v,l-v) \pmod{p}.}
	\end{align*}
The proof of the second part is similar, and uses the second part of lemma~\ref{lemma7}.
\end{proof}

 \section{General theorems}

\subsection{The main theorems}

The main results in this section are the following two theorems.

\begin{theorem}\label{l7Z}
Let $r,m,\alpha\geqslant 0$ be such that $r \geqslant (m+1)(p + 1)$, and $m \geqslant \alpha$, and $r \equiv_{p-1} 2L$, with $L \in \{1,\ldots,m\}$. Let $a \in \overline {\mathbb Q}_p$ be such that $m+1 > v(a) > m$. Define the $(\alpha+1) \times (m+1+\mu)$ matrix $\smash{\mathfrak M^{(r,m,\alpha)}}$, where $\omega_1 < \cdots <\omega_\mu$ are such that
	\[\textstyle \{\omega_1,\ldots,\omega_\mu\} = \left\{j \,|\, j(p-1)+\alpha \in [\alpha,m] \cup [r-m,r-\alpha]\right\},\]
by
	\begin{align*}
	\textstyle \mathfrak M^{(r,m,\alpha)}_{u,j} & \textstyle = {\omega_j \choose u-1} - \delta_{u=\alpha+1}((-1)^\alpha\delta_{j=1}+\delta_{j=\mu}\delta_{L \equiv_{p-1} \alpha}),
	\end{align*}
whenever $1\leqslant j\leqslant \mu$, and
	\begin{align*}
	\textstyle \mathfrak M^{(r,m,\alpha)}_{u,l+m-\alpha+1+\mu} & \textstyle = \sum_{r - m > i(p-1) + \alpha > m} {r-\alpha+l \choose i(p-1) +l} {i \choose u-1}.
	\end{align*}
Suppose that $(0,\ldots,0,1)^T$ is in the range of $\smash{\mathfrak M^{(r,m,\alpha)}}$ over $\mathbb F_p$.  Then (the map induced by) $\Psi_\alpha$, restricted to $X(r+2,a)$, is a surjective map
	\[\textstyle \Psi_\alpha|_{X(r+2,a)} :  \smash{I(\overline\theta^\alpha\sigma_{r-\alpha(p+1)}) \cap X(r+2,a) \subseteq I(\sigma_r) \twoheadrightarrow I(\sigma_{(2\alpha-r \bmod p-1)} (r-\alpha))}.\] 
\end{theorem}

\begin{theorem}\label{theorem26}
Let $r,m,\alpha\geqslant 0$ be integers such that $r > m(p+1)$, and $m > \alpha$, and $r \equiv_{p-1} 2L$, with $\smash{L \in \left\{1,\ldots,\frac{p-1}{2}\right\}}$. Let $a \in \overline {\mathbb Q}_p$ be such that $m+1 > v(a) > m$. For $w \geqslant 0$, define
	\[\textstyle \mathfrak m_{w} (C_1,\ldots,C_\alpha) = \sum_{r - 2\alpha > i(p-1) > 0} \sum_{l=0}^{\alpha} C_l {r - \alpha + l \choose i(p-1)+l} {i \choose w}, \]
where $C_0 = 1$. Denote by $\mathfrak v_{w} (C_1,\ldots,C_\alpha)$ the valuation $v(\mathfrak m_{w} (C_1,\ldots,C_\alpha))$. Suppose that the following three conditions hold true:
\begin{enumerate}[1.]
	\item $\smash{\textstyle \mathfrak m_{w} (C_1,\ldots,C_\alpha)} = 0$, for all $0 \leqslant w < \alpha$;
	\item $\smash{\textstyle \min\{v(a)-\alpha,\mathfrak v_{\alpha} (C_1,\ldots,C_\alpha)\} \leqslant  \mathfrak v_{w} (C_1,\ldots,C_\alpha)}$, for all $\alpha < w \leqslant 2m+1-\alpha$;
	\item If $L \equiv_{p-1} \alpha$ and $\smash{v(a)-\alpha < \mathfrak v_{\alpha} (C_1,\ldots,C_\alpha)}$, then $\smash{\sum_{l=0}^\alpha C_l {r-\alpha+l \choose \alpha} \equiv (-1)^{\alpha+1} \pmod {p}}$.
\end{enumerate}
\begin{enumerate}[\tttttt] 
	\item If $\smash{\textstyle \mathfrak v_{\alpha} (C_1,\ldots,C_\alpha) < v(a)-\alpha}$, then
	\[\textstyle \Psi_\alpha|_{X(r+2,a)} : \smash{I(\overline\theta^\alpha\sigma_{r-\alpha(p+1)}) \cap X(r+2,a) \subseteq I(\sigma_r) \twoheadrightarrow I(\sigma_{(2\alpha-r \bmod p-1)} (r-\alpha))}\]
	is a surjection;
	\item If $\smash{\textstyle \mathfrak v_{\alpha} (C_1,\ldots,C_\alpha) > v(a)-\alpha}$, then \[\textstyle \Psi_\alpha \left(  \smash{I(\overline\theta^\alpha\sigma_{r-\alpha(p+1)}) \cap X(r+2,a)}\right) \supseteq \smash{T \left( I(\sigma_{(2\alpha-r \bmod p-1)} (r-\alpha))\right)}.\]
\end{enumerate}
\end{theorem}

\subsection{Proof of theorem~\ref{l7Z}}

\subsubsection{Lemmas}

\begin{lemma}\label{l7xxx}
Let $r,m,n,\alpha,\beta \geqslant 0$ be integers, and let $a \in \overline {\mathbb Q}_p$ be such that $m+1 > v(a) > m$. Suppose that $n \geqslant m+1$, and $n \geqslant \alpha$, and $p-1 > \beta$, and $r \geqslant np + \alpha$. Then there is some $\textstyle\tau_{r,m,n,\alpha,\beta}^\star$ such that
	\begin{align*}
		& \textstyle (T - a) \tau_{r,m,n,\alpha,\beta}^\star 
		\\ & \textstyle  \qquad \equiv 
         \sum_{i}  \left[1,x^{i(p-1)+\alpha-\beta}y^{r-i(p-1)-\alpha+\beta}\right] 
         \sum_{j < i}   (-1)^{j} {{n}  \choose j} {r-j(p-1)-\alpha \choose (i-j)(p-1)-\beta} p^{j(p-1)+\alpha} (p-1)
                \\ & \textstyle \qquad \qquad + (p-1) (-1)^{\alpha+1} \delta_{r \equiv_{p-1} 2\alpha}  
                \sum_{j}  (-1)^{j} {{n}  \choose j} p^{j(p-1)+\alpha} \left[1,y^{j(p-1)+\alpha}x^{r-j(p-1)-\alpha}\right]
	\\ & \textstyle \qquad \qquad - a \sum_{j}  (-1)^{j} {{n-\alpha}  \choose j} \sum_\mu [\mu]^\beta \left[{\begin{Lmatrix} p & [\mu] \\ 0 & 1 \end{Lmatrix}},\theta^\alpha x^{j(p-1)}y^{r-\alpha (p+1)-j(p-1)}\right] 
	\\ & \textstyle \qquad \qquad  -  a \sum_{j}  (-1)^{j} {{n-\alpha}  \choose j} \delta_{r \equiv_{p-1} 2\alpha} \left[{\begin{Lmatrix} 1 & 0 \\ 0 & p \end{Lmatrix}},\theta^\alpha y^{j(p-1)}x^{r-\alpha (p+1)-j(p-1)}\right] \pmod {p^{{m+1}}}.
	\end{align*}
\end{lemma}

\begin{proof}
Take
	\[\textstyle \smash{\tau_{r,m,n,\alpha,\beta}^\star = \sum_\mu [\mu]^\beta \phi_{r,m,n,\alpha,{\begin{Smatrix} p & [\mu] \\ 0 & 1 \end{Smatrix}}}^{\ast}  - p\delta_{\beta=0} \phi_{r,m,n,\alpha,{\begin{Smatrix} p & 0 \\ 0 & 1 \end{Smatrix}}}^{\ast} - (p-1) (-1)^\alpha \delta_{r \equiv_{p-1} 2\alpha} \phi_{r,m,n,\alpha,{\begin{Smatrix} 1 & 0 \\ 0 & p \end{Smatrix}}}^{\ast\ast}}.\]
Then
	\begin{align*}
			& \textstyle (T - a) \tau_{r,m,n,\alpha,\beta}^\star 
			\\ & \textstyle \qquad\equiv 
			\sum_{j}  (-1)^{j} {{n}  \choose j} p^{j(p-1)+\alpha} \sum_\mu [\mu]^\beta \left[{\begin{Lmatrix} 1 & [\mu] \\ 0 & 1 \end{Lmatrix}},x^{j(p-1)+\alpha}y^{r-j(p-1)-\alpha}\right]
							- p\delta_{\beta=0} \phi_{r,m,n,\alpha,{\begin{Smatrix} p & 0 \\ 0 & 1 \end{Smatrix}}}^{\ast}
					\\ & \textstyle \qquad\qquad + (p-1) (-1)^{\alpha+1} \delta_{r \equiv_{p-1} 2\alpha}\sum_{j}  (-1)^{j} {{n}  \choose j} p^{j(p-1)+\alpha} \left[1,y^{j(p-1)+\alpha}x^{r-j(p-1)-\alpha}\right]
	 \\ & \textstyle\qquad \qquad - a \sum_{j}  (-1)^{j} {{n-\alpha}  \choose j} \sum_\mu [\mu]^\beta \left[{\begin{Lmatrix} p & [\mu] \\ 0 & 1 \end{Lmatrix}},\theta^\alpha x^{j(p-1)}y^{r-\alpha (p+1)-j(p-1)}\right] 
	 \\ & \textstyle\qquad \qquad  -  a \sum_{j}  (-1)^{j} {{n-\alpha}  \choose j} \delta_{r \equiv_{p-1} 2\alpha} \left[{\begin{Lmatrix} 1 & 0 \\ 0 & p \end{Lmatrix}},\theta^\alpha y^{j(p-1)}x^{r-\alpha (p+1)-j(p-1)}\right] \pmod {p^{{m+1}}}.
	\end{align*}
Note that
	\begin{align*}
			& \textstyle \sum_{j}  (-1)^{j} {{n}  \choose j} p^{j(p-1)+\alpha} \sum_\mu [\mu]^\beta \left[{\begin{Lmatrix} 1 & [\mu] \\ 0 & 1 \end{Lmatrix}},x^{j(p-1)+\alpha}y^{r-j(p-1)-\alpha}\right] 
			\\ & \textstyle \qquad = \sum_{j}  (-1)^{j} {{n}  \choose j} p^{j(p-1)+\alpha} \sum_\mu [\mu]^\beta \left[1,x^{j(p-1)+\alpha}(y+[\mu]x)^{r-j(p-1)-\alpha}\right]
					\\ &  \textstyle \qquad  = \sum_{i,j}  (-1)^{j} {{n}  \choose j} {r-j(p-1)-\alpha \choose i} p^{j(p-1)+\alpha} \sum_\mu [\mu]^{i+\beta} \left[1,x^{j(p-1)+\alpha+i}y^{r-j(p-1)-\alpha-i}\right]
					\\ &  \textstyle \qquad  = \sum_{i,j}  (-1)^{j} {{n}  \choose j} {r-j(p-1)-\alpha \choose i(p-1)-\beta} p^{j(p-1)+\alpha} (p-1+\delta_{i(p-1)=\beta}) \\& \textstyle \qquad\qquad\times \left[1,x^{(i+j)(p-1)+\alpha-\beta}y^{r-(i+j)(p-1)-\alpha+\beta}\right]
					\\ &  \textstyle \qquad  = \sum_{i,j}  (-1)^{j} {{n}  \choose j} {r-j(p-1)-\alpha \choose i(p-1)-\beta} p^{j(p-1)+\alpha} (p-1+\delta_{i=\beta=0}) \\& \textstyle \qquad\qquad\times \left[1,x^{(i+j)(p-1)+\alpha-\beta}y^{r-(i+j)(p-1)-\alpha+\beta}\right]
					\\ &  \textstyle \qquad  = \sum_{i,j}  (-1)^{j} {{n}  \choose j} {r-j(p-1)-\alpha \choose (i-j)(p-1)-\beta} p^{j(p-1)+\alpha} (p-1+\delta_{i=j}) \\& \textstyle \qquad\qquad\times \left[1,x^{i(p-1)+\alpha-\beta}y^{r-i(p-1)-\alpha+\beta}\right]
									\\ &  \textstyle \qquad  \equiv \sum_{i}  \left[1,x^{i(p-1)+\alpha-\beta}y^{r-i(p-1)-\alpha+\beta}\right] \sum_{j < i}   (-1)^{j} {{n}  \choose j} {r-j(p-1)-\alpha \choose (i-j)(p-1)-\beta} p^{j(p-1)+\alpha} (p-1)
									\\ & \textstyle \qquad\qquad +\smash{p\delta_{\beta=0} \phi_{r,m,n,\alpha,{\begin{Smatrix} p & 0 \\ 0 & 1 \end{Smatrix}}}^{\ast}} \pmod {p^{{m+1}}}.
	\end{align*}
The proof of the lemma can be completed by combining these two congruences.
\end{proof}

\begin{lemma}\label{l7z}
Let $r,m,n,\alpha\geqslant 0$ be integers, and let $a \in \overline {\mathbb Q}_p$ be such that $m+1 > v(a) > m$. Suppose that $n \geqslant m+1$, and $n \geqslant \alpha$, and $r \geqslant np + \alpha$. Then, for any $C_1,\ldots,C_\alpha \in \mathbb Z_p$, there is some $\textstyle\tau_{r,m,n,\alpha,C_1,\ldots,C_\alpha}^{\star\star}$ such that
	\begin{align*}
			& \textstyle (T - a) \tau_{r,m,n,\alpha,C_1,\ldots,C_\alpha}^{\star\star} \\& \textstyle \qquad\equiv 
			 \sum_{i} D_i \left[1,x^{i(p-1)+\alpha}y^{r-i(p-1)-\alpha}\right]
	 \\ & \textstyle \qquad\qquad - a \sum_{j}  (-1)^{j} {{n-\alpha}  \choose j} \sum_\mu \left[{\begin{Lmatrix} p & [\mu] \\ 0 & 1 \end{Lmatrix}},\theta^\alpha x^{j(p-1)}y^{r-\alpha (p+1)-j(p-1)}\right] 
	 \\ & \textstyle\qquad \qquad  -  a \sum_{j}  (-1)^{j} {{n-\alpha}  \choose j} \delta_{r \equiv_{p-1} 2\alpha} \left[{\begin{Lmatrix} 1 & 0 \\ 0 & p \end{Lmatrix}},\theta^\alpha y^{j(p-1)}x^{r-\alpha (p+1)-j(p-1)}\right] \pmod {p^{{m+1}}},
	\end{align*}
where
	\begin{align*}
			\textstyle (p-1)^{-1} D_i & \textstyle =  \smash{\sum_{j < i}   (-1)^{j} {{n}  \choose j} p^{j(p-1)+\alpha} \left( {r-j(p-1)-\alpha \choose (i-j)(p-1)}+ \sum_{l=1}^\alpha C_l {r-j(p-1)-\alpha+l \choose (i-j)(p-1)+l}\right)}
					\\ & \textstyle \qquad + (-1)^{\alpha+1} \delta_{r \equiv_{p-1} 2\alpha} (-1)^{(r - 2\alpha)/(p-1)-i} {{n}  \choose (r - 2\alpha)/(p-1)-i} p^{r-i(p-1)-\alpha}.
	\end{align*}
\end{lemma}

\begin{proof}
Write
	\[\textstyle \smash{\tau' = \sum_\mu [\mu]^\beta \phi_{r,m,n,\alpha,{\begin{Smatrix} p & [\mu] \\ 0 & 1 \end{Smatrix}}}^{\ast}  - p\delta_{\beta=0} \phi_{r,m,n,\alpha,{\begin{Smatrix} p & 0 \\ 0 & 1 \end{Smatrix}}}^{\ast} }\]
and
	\[\textstyle \smash{\tau'' = - (p-1) (-1)^\alpha \delta_{r \equiv_{p-1} 2\alpha} \phi_{r,m,n,\alpha,{\begin{Smatrix} 1 & 0 \\ 0 & p \end{Smatrix}}}^{\ast\ast}},\]
so that $\smash{\tau_{r,m,n,\alpha,\beta}^\star = \tau' + \tau''}$. Take
	\[\textstyle \smash{\tau_{r,m,n,\alpha,C_1,\ldots,C_\alpha}^{\star\star} = \tau_{r,m,n,\alpha,0}^\star + \sum_{l=1}^{\alpha} C_l p^l \tau_{r,m,n,\alpha-l,(-l \bmod p-1)}'}.\]
It can be checked that $\smash{\tau_{r,m,n,\alpha,C_1,\ldots,C_\alpha}^{\star\star}}$ satisfies all of the requirements in the statement of the theorem.
\end{proof}

\begin{corollary}\label{l7zz}
Let $r,m,\alpha\geqslant 0$ be integers, and let $a \in \overline {\mathbb Q}_p$ be such that $m+1 > v(a) > m$. Suppose that $m \geqslant \alpha$, and $r \geqslant (m+1)(p + 1)$. Then, for any $F_j,C_0,\ldots,C_{m} \in \mathbb Z_p$, there is some $\textstyle\smash{\tau_{r,m,\alpha,F_j,C_0,\ldots,C_{m}}^{\dagger}}$ such that
	\begin{align*}
		& \textstyle (T - a) \smash{\tau_{r,m,\alpha,F_j,C_0,\ldots,C_{m}}^{\dagger}} 
		\\ & \textstyle\qquad  \equiv 
         \sum_{r - m > i(p-1)+\alpha > m} E_i \left[1,x^{i(p-1)+\alpha}y^{r-i(p-1)-\alpha}\right]
         \\ & \textstyle \qquad\qquad +  \sum_{j(p-1)+\alpha \in [\alpha,m] \cup [r-m,r-\alpha]} F_j \left[1,x^{j(p-1)+\alpha}y^{r-j(p-1)-\alpha}\right] \pmod {p^{{v(a)-m}}},
	\end{align*}
where $\smash{\textstyle E_i \textstyle = \sum_{l=\alpha-m}^{\alpha} C_{l+m-\alpha} {r-\alpha+l \choose i(p-1) +l}}$.
\end{corollary}

\begin{proof}
Define $\tau'$ as in the proof of lemma~\ref{l7z}. Take 
	\begin{align*}
		\textstyle  \smash{\smash{\tau_{r,m,\alpha,F_j,C_0,\ldots,C_m}^{\dagger}}} & \textstyle =
		\smash{(p-1)^{-1}p^{-m} \sum_{l=\alpha-m}^{\alpha} C_{l+m-\alpha} p^{l+m-\alpha} \tau_{r,m,m+1,\alpha-l,(-l \bmod p-1)}'}
	\\ & \qquad \textstyle + \sum_{u=0}^{m} X_u p^{-u} \phi^\ast_{r,m,m+1,u,{\begin{Smatrix} p & 0 \\ 0 & 1 \end{Smatrix}}} + \sum_{u=0}^{m} Y_u p^{-u} \phi^{\ast\ast}_{r,m,m+1,u,{\begin{Smatrix} 1 & 0 \\ 0 & p \end{Smatrix}}},
	\end{align*}
with suitably chosen constants $X_u, Y_u \in \mathbb Z_p$.
\end{proof}

\subsubsection{Completing the proof of theorem~\ref{l7Z}}

\begin{proof}[Proof of theorem~\ref{l7Z}]
Let $n=m+1$. Since $(0,\ldots,0,1)^T$ is in the range of $\smash{\mathfrak M^{(r,m,\alpha)}}$ over $\mathbb F_p$, the constants $F_j, C_l \in \mathbb Z_p$ in lemma~\ref{l7zz} can be chosen to  be such that
	\begin{align*}
	& \textstyle \sum_{j = 1}^\mu F_{\omega_j} \left({\omega_j \choose u-1} - \delta_{u=\alpha+1}((-1)^\alpha\delta_{j=1}+\delta_{j=\mu}\delta_{L \equiv_{p-1} \alpha})\right) 
	\\ & \textstyle \qquad + \sum_{l=\alpha-m}^\alpha \sum_{r - m > i(p-1)+\alpha > m} C_l {r-\alpha+l \choose i(p-1) +l} {i \choose u-1} = \delta_{u=\alpha+1},
	\end{align*}
or, equivalently,
	\begin{align*}
	& \textstyle \sum_{j = 1}^\mu F_{\omega_j} {\omega_j \choose u-1} + \sum_{l=\alpha-m}^\alpha \sum_{r - m > i(p-1)+\alpha > m} C_l {r-\alpha+l \choose i(p-1) +l} {i \choose u-1} 
	\\  & \textstyle \qquad = \delta_{u=\alpha+1} (1 + (-1)^\alpha F_{\omega_1} + \delta_{L = 2\alpha} F_{\omega_\mu}).
	\end{align*}
This implies that the constants $F_j,E_i \in \mathbb Z_p$ are such that
	\begin{align*}
			&\textstyle  \sum_{r - m > i(p-1)+\alpha > m} E_i \left[1,x^{i(p-1)+\alpha}y^{r-i(p-1)-\alpha}\right]
			 \\  & \textstyle \qquad +  \sum_{j(p-1)+\alpha \in [\alpha,m] \cup [r-m,r-\alpha]} F_j \left[1,x^{j(p-1)+\alpha}y^{r-j(p-1)-\alpha}\right]   \\ & \textstyle \qquad\qquad  = 
			 \sum_{i}E_i' \left[1,\theta^\alpha x^{i(p-1)}y^{r-\alpha(p+1)-i(p-1)}\right],
	\end{align*}
for some constants $E_i' \in \mathbb Z_p$ which add up to 
	\[\textstyle \smash{1 + (-1)^\alpha F_{\omega_1} + \delta_{L \equiv_{p-1} \alpha} F_{\omega_\mu}}.\]
Moreover, 
	\begin{align*}
		\textstyle E_0' & \textstyle =(-1)^\alpha F_{\omega_1}, \\
		\textstyle E_{(r-\alpha(p+1))/(p-1)}' & \textstyle=F_{\omega_\mu} \text{ whenever } L \equiv_{p-1} \alpha.
	\end{align*}
This implies that $\textstyle (T - a) \smash{\tau_{r,m,n,\alpha,F_j,C_0,\ldots,C_{m}}^{\dagger}}$ is in the domain of (the map induced by) $\Psi_\alpha$, and
	\begin{align*}
			\textstyle \overline{(T - a) \smash{\tau_{r,m,n,\alpha,F_j,C_0,\ldots,C_{m}}^{\dagger}}} & \textstyle = 
			 \sum_i E_i' \left[1,\overline \theta^\alpha x^{i(p-1)}y^{r-\alpha(p+1)-i(p-1)}\right],
	\end{align*}
which implies that $\textstyle\smash{\Psi_\alpha(\overline{(T - a) \smash{\tau_{r,m,n,\alpha,F_j,C_0,\ldots,C_{m}}^{\dagger}}}) \textstyle =  \left[1,X^{(2\alpha - r \bmod p-1)}\right]}$.  The fact that the module $\sigma_{(2\alpha-r \bmod p-1)} (r-\alpha)$ is simple and hence generated by $X^{(2\alpha - r \bmod p-1)}$ completes the proof of the theorem.
\end{proof}

\emph{Remark.} Theorem~\ref{l7Z} can be proved similarly by using lemma~\ref{lemma5} instead of corollary~\ref{l7zz}.

\subsection{Proof of theorem~\ref{theorem26}}

\begin{proof}[Proof of theorem~\ref{theorem26}]
We know, from the proof of lemma~\ref{l7x}, that
	\begin{align*}
		\textstyle (T - a) \phi_{\alpha+1,\alpha,1}^{(0)} & \textstyle \equiv \sum_{z = \alpha+1}^{m} p^z  \sum_{j}  (-1)^{j} {{\alpha+1}  \choose j} {r - j(p-1)-\alpha \choose z} \sum_{\lambda \ne 0} (-[\lambda])^{r-\alpha-z} \left[  {\begin{Lmatrix} p & [\lambda] \\ 0 & 1 \end{Lmatrix}}, x^{r-z}y^z\right] 
	\\ & \textstyle \qquad
			+ \sum_{j}  (-1)^{j} {{\alpha+1}  \choose j} p^{r-j(p-1)-\alpha} \left[{\begin{Lmatrix} p & 0 \\ 0 & 1 \end{Lmatrix}},x^{j(p-1)+\alpha}y^{r-j(p-1)-\alpha}\right]
	\\ & \textstyle \qquad
			+ \sum_{j}  (-1)^{j} {{\alpha+1}  \choose j} p^{j(p-1)+\alpha} \left[{\begin{Lmatrix} 1 & 0 \\ 0 & p \end{Lmatrix}},x^{j(p-1)+\alpha}y^{r-j(p-1)-\alpha}\right]
	\\ & \textstyle \qquad
			- a \sum_{j}  (-1)^{j} {{1}  \choose j} \left[1,\theta^\alpha x^{j(p-1)}y^{r-\alpha(p+1)-j(p-1)}\right] 
	\\ & \textstyle \equiv \sum_{z = \alpha+1}^{m} p^z  \sum_{j}  (-1)^{j} {{\alpha+1}  \choose j} {r - j(p-1)-\alpha \choose z} \sum_{\lambda \ne 0} (-[\lambda])^{r-\alpha-z} \left[  {\begin{Lmatrix} p & [\lambda] \\ 0 & 1 \end{Lmatrix}}, x^{r-z}y^z\right] 
	\\ & \textstyle \qquad
			+ \sum_{m\geqslant j(p-1)+\alpha \geqslant 0}  (-1)^{j} {{\alpha+1}  \choose j} p^{j(p-1)+\alpha} \left[{\begin{Lmatrix} 1 & 0 \\ 0 & p \end{Lmatrix}},x^{j(p-1)+\alpha}y^{r-j(p-1)-\alpha}\right]
	\\ & \textstyle \qquad
			- a \sum_{j}  (-1)^{j} {{1}  \choose j} \left[1,\theta^\alpha x^{j(p-1)}y^{r-\alpha(p+1)-j(p-1)}\right] 
	\\ & \textstyle \equiv \sum_{z = \alpha+1}^{m} \sum'_{j_z \in S_z}  A_{j_z} p^z \left[g_{j_z}, x^{z}y^{r-z}\right] 
			+ p^{\alpha} \left[{\begin{Lmatrix} 1 & 0 \\ 0 & p \end{Lmatrix}},x^{\alpha}y^{r-\alpha}\right]
	\\ & \textstyle \qquad
			- a \sum_{j}  (-1)^{j} {{1}  \choose j} \left[1,\theta^\alpha x^{j(p-1)}y^{r-\alpha(p+1)-j(p-1)}\right]\pmod {p^{{m+1}}},
	\end{align*}
where, for each $z$, the sum $\smash{\sum'_{j_z \in S_z}}$ is over some finite set of indices $S_z$, and each $A_{j_z}$ is a constant, and each $g_{j_z}$ is an element of $G$. Note that, for each $z$ such that $m \geqslant z \geqslant \alpha+1$, we have
	\[\textstyle p^z \left[g_{j_z}, x^{z}y^{r-z}\right] = a[g_{j_z}',\theta^zh_{j_z}] + O(p^{m+1}) = a[g_{j_z}',\theta^{m+1} h_{j_z}'] + O(p^{m+1}),\]
for some polynomials $h_{j_z}, h_{j_z}'$. It can be shown by the same algebraic manipulation used in the proof of lemma~\ref{l7xxx} that
	\[\textstyle \sum_{\mu} p^{\alpha} \left[{\begin{Lmatrix} p & [\mu] \\ 0 & 1 \end{Lmatrix}}{\begin{Lmatrix} 1 & 0 \\ 0 & p \end{Lmatrix}},x^{\alpha}y^{r-\alpha}\right] = \sum_{i>0} p^\alpha(p-1) {r - \alpha \choose i(p-1)}[1,x^{i(p-1)+\alpha}y^{r-i(p-1)-\alpha}] + O(p^{m+1}).\]
Thus, by letting 
	\[\textstyle \tau_{\dagger} = p^{-\alpha}(p-1)^{-1}\left(\sum_\mu  \phi_{\alpha+1,\alpha,{{\begin{Lmatrix} p & [\mu] \\ 0 & 1 \end{Lmatrix}}}}^{(0)} + (-1)^\alpha \delta_{L \equiv_{p-1} \alpha} \phi_{\alpha+1,\alpha,{{\begin{Lmatrix} 0 & 1 \\ 1 & 0 \end{Lmatrix}}}}^{(0)}\right)\]
and $a' = p^{-\alpha}(p-1)^{-1}a$, we get the equation
	\begin{align*}
		\textstyle (T - a) \tau_\dagger & \textstyle \equiv  \sum_{i>0} {r - \alpha \choose i(p-1)} [1,x^{i(p-1)+\alpha}y^{r-i(p-1)-\alpha}]
	\\ & \textstyle \qquad
			- a'  \sum_{j}  (-1)^{j} {{1}  \choose j} \sum_\mu \left[{\begin{Lmatrix} p & [\mu] \\ 0 & 1 \end{Lmatrix}},\theta^\alpha x^{j(p-1)}y^{r-\alpha(p+1)-j(p-1)}\right] 
	\\ & \textstyle \qquad
			- a'  \sum_{j}  (-1)^{j} {{1}  \choose j} \delta_{L \equiv_{p-1} \alpha} \left[{\begin{Lmatrix} 1 & 0 \\ 0 & p \end{Lmatrix}},\theta^\alpha y^{j(p-1)}x^{r-\alpha(p+1)-j(p-1)}\right] 
	\\ & \textstyle \qquad
			+ a' \sum'_{j' \in S'}  A_{j'} \left[g_{j'}, \theta^{\alpha+1} h_{j'}'\right] \pmod {p^{{m+1-\alpha}}}.
	\end{align*}
Let us denote this equation by Eq.~$\alpha$. By adding constant multiples of expressions which are similar to Eq.~$1$ through Eq.~$(\alpha-1)$ with $\smash{\sum_\mu}$ replaced by $\smash{\sum_\mu [\mu]^{l}}$, we can get an analogue of lemma~\ref{l7z}:
	\begin{align*}
		\textstyle (T - a) \tau_\dagger & \textstyle \equiv  \sum_{r - 2\alpha > i(p-1) > 0} D_i [1,x^{i(p-1)+\alpha}y^{r-i(p-1)-\alpha}]
	\\ & \textstyle \qquad
			- a'  \sum_{j}  (-1)^{j} {{1}  \choose j} \sum_\mu \left[{\begin{Lmatrix} p & [\mu] \\ 0 & 1 \end{Lmatrix}},\theta^\alpha x^{j(p-1)}y^{r-\alpha(p+1)-j(p-1)}\right] 
	\\ & \textstyle \qquad
			- (-1)^{\alpha+1} \left(\sum_{l=0}^\alpha C_l {r - \alpha + l \choose \alpha}\right) \\ & \textstyle \qquad\qquad\times a'  \sum_{j}  (-1)^{j} {{1}  \choose j} \delta_{L \equiv_{p-1} \alpha} \left[{\begin{Lmatrix} 1 & 0 \\ 0 & p \end{Lmatrix}},\theta^\alpha y^{j(p-1)}x^{r-\alpha(p+1)-j(p-1)}\right] 
	\\ & \textstyle \qquad
			+ a' \sum'_{j' \in S'}  A_{j'} \left[g_{j'}, \theta^{\alpha+1} h_{j'}'\right] \pmod {p^{{m+1-\alpha}}},
	\end{align*}
where $D_i = \sum_{l=0}^\alpha C_l {r - \alpha+l \choose i(p-1)+l}$, and $C_0 = 1$, and $v(a') = v(a) - \alpha$. We can restrict the first sum to the range $r - 2\alpha \geqslant i(p-1) > 0$ by noting that
	\[\textstyle [1,x^{i(p-1)+\beta}y^{r-i(p-1)-\beta}] = O(p^{m-\beta}) = O(p^{m+1-\alpha}),\]
for all $\beta$ such that $\alpha > \beta \geqslant 0$. Let
	\begin{align*}
			\smash{\tau_0}  & \textstyle=  
			 \sum_{r - 2\alpha > i(p-1) > 0} (c_i - D_i) \left[1,x^{i(p-1)+\alpha}y^{r-i(p-1)-\alpha}\right],
	\end{align*}
where 
	\[\textstyle c_i = (-1)^{i-1} {\alpha \choose i-1} \sum_{r - 2\alpha > j(p-1) > 0} D_j {j \choose \alpha}.\]
We will denote by $\smash{\sum'}$ the restricted sum $\smash{\sum_{r - 2\alpha > i(p-1) > 0}}$. Then
	\[\textstyle \sum' c_i \left[1,x^{i(p-1)+\alpha}y^{r-i(p-1)-\alpha}\right] = \left(\sum' D_i {i \choose \alpha}\right) \left[1,\theta^\alpha x^{p-1}y^{r-(\alpha+1)(p+1)+2}\right],\]
and
	\[\textstyle v(c_i) \geqslant v \left(\sum_{i>0} D_i {i \choose \alpha}\right) = v \left(\sum_{i>0} c_i {i \choose \alpha}\right).\]
Finally, note that
	\begin{align*}
		 \textstyle (T - a) \smash{\smash{\tau_0}} & \textstyle \equiv 
			 \sum_{\lambda \ne 0} \sum_{z \geqslant 0} p^z \sum' (c_i - D_i) {r - i(p-1) - \alpha \choose z} \left[{\begin{Lmatrix} p & [\lambda] \\ 0 & 1 \end{Lmatrix}},(-[\lambda])^{r-\alpha-z}x^{r-z}y^z\right]
	 \\ & \textstyle \qquad - a  \sum' (c_i - D_i)  \left[1, x^{i(p-1)+\alpha}y^{r-i(p-1)-\alpha}\right] 
	 \\ & \textstyle \equiv \sum_{z = \alpha+1}^{2m+1-\alpha}   \sum'  (c_i - D_i) {r - i(p-1) - \alpha \choose z} \sum_{\lambda \ne 0}  (-[\lambda])^{r-\alpha-z} p^z \left[{\begin{Lmatrix} p & [\lambda] \\ 0 & 1 \end{Lmatrix}},x^{r-z}y^z\right]
		\\ & \textstyle \qquad	 - a  \sum'  (c_i - D_i)  \left[1, x^{i(p-1)+\alpha}y^{r-i(p-1)-\alpha}\right] \pmod {p^{{2m+2-\alpha}}}.
	\end{align*}
Write $\smash{\mathcal V = v\big(\smash{\sum' D_i {i \choose \alpha}}\big)}$ and $\smash{\mathcal V_z = v\left(\sum' (c_i - D_i) {r - i(p-1) - \alpha \choose z}\right)}$, and note that $\mathcal V_z \geqslant \mathcal V$, for all $\alpha < z \leqslant 2m+1-z$, by assumption. Therefore, 
	\begin{align*}
			& \textstyle  \sum_{z = m+1}^{2m+1-\alpha}   \sum' (c_i - D_i) {r - i(p-1) - \alpha \choose z} \sum_{\lambda \ne 0}  (-[\lambda])^{r-\alpha-z} p^z \left[{\begin{Lmatrix} p & [\lambda] \\ 0 & 1 \end{Lmatrix}},x^{r-z}y^z\right]
	 \\ & \textstyle \qquad \equiv \sum_{z = m+1}^{2m+1-\alpha}  \mathcal V_z p^z X_z \pmod {p^{{2m+2-\alpha}}},
	\end{align*}
where $X_z$ is some expresion such that $v(X_z) \geqslant 0$. This implies that
	\begin{align*}
			& \textstyle  -a^{-1}\sum_{z = m+1}^{2m+1-\alpha}   \sum' (c_i - D_i) {r - i(p-1) - \alpha \choose z} \sum_{\lambda \ne 0}  (-[\lambda])^{r-\alpha-z} p^z \left[{\begin{Lmatrix} p & [\lambda] \\ 0 & 1 \end{Lmatrix}},x^{r-z}y^z\right]
	 \\ & \textstyle \qquad = \sum_{z = m+1}^{2m+1-\alpha}  \mathcal V_z X_z' + O(p^{{m+1-\alpha}}) = \mathcal V(p^{m+1-v(a)}) + O(p^{{m+1-\alpha}}).
	\end{align*}
Similarly,
	\begin{align*}
			& \textstyle  \sum_{z = \alpha+1}^{m}   \sum' (c_i - D_i) {r - i(p-1) - \alpha \choose z} \sum_{\lambda \ne 0}  (-[\lambda])^{r-\alpha-z} p^z \left[{\begin{Lmatrix} p & [\lambda] \\ 0 & 1 \end{Lmatrix}},x^{r-z}y^z\right]
	 \\ & \textstyle \qquad = \sum_{z = \alpha+1}^{m}  \mathcal V_z \sum_{\lambda \ne 0}  (-[\lambda])^{r-\alpha-z} p^z \left[{\begin{Lmatrix} p & [\lambda] \\ 0 & 1 \end{Lmatrix}},x^{r-z}y^z\right] + O(p^{2m+2-\alpha})
	 \\ & \textstyle \qquad = \sum_{z = \alpha+1}^{m}  \mathcal V_z \left(\sum_{\lambda \ne 0}  (-[\lambda])^{r-\alpha-z} a \left[{\begin{Lmatrix} p & [\lambda] \\ 0 & 1 \end{Lmatrix}},\theta^z h_z\right] + O(p^{m+1})\right) + O(p^{2m+2-\alpha})
	 \\ & \textstyle \qquad = \mathcal V \sum_{z = \alpha+1}^{m} Y_z \sum_{\lambda \ne 0}  (-[\lambda])^{r-\alpha-z} a \left[{\begin{Lmatrix} p & [\lambda] \\ 0 & 1 \end{Lmatrix}},\theta^z h_z\right] + \mathcal V O(p^{m+1}) + O(p^{2m+2-\alpha}),
	\end{align*}
where $h_z$ is some polynomial, and $Y_z$ is some expresion such that $v(Y_z) \geqslant 0$. This implies that
	\begin{align*}
			& \textstyle  -a^{-1}\sum_{z = \alpha+1}^{m}   \sum' (c_i - D_i) {r - i(p-1) - \alpha \choose z} \sum_{\lambda \ne 0}  (-[\lambda])^{r-\alpha-z} p^z \left[{\begin{Lmatrix} p & [\lambda] \\ 0 & 1 \end{Lmatrix}},x^{r-z}y^z\right]
	 \\ & \textstyle \qquad = \mathcal V \sum_{j\in S} [g_j,\theta^{\alpha+1} h_j'] + \mathcal V O(p^{m+1-v(a)}) + O(p^{m+1-\alpha}),
	\end{align*}
where $g_j$ is some element of $G$, and $h_j'$ is some polynomial. This implies that
	\begin{align*}
			& \textstyle (T - a) (\smash{%
			\tau_\dagger-a^{-1}\tau_0})  \\ & \textstyle \qquad \equiv 
			 \sum' c_i \left[1,x^{i(p-1)+\alpha}y^{r-i(p-1)-\alpha}\right] +  \mathcal V \sum_{j\in S} [g_j,\theta^{\alpha+1} h_j'] + \mathcal V O(p^{\delta})
	 	\\ & \textstyle \qquad\qquad
			- a'  \sum_{j}  (-1)^{j} {{1}  \choose j} \sum_\mu \left[{\begin{Lmatrix} p & [\mu] \\ 0 & 1 \end{Lmatrix}},\theta^\alpha x^{j(p-1)}y^{r-\alpha(p+1)-j(p-1)}\right] 
	\\ & \textstyle \qquad\qquad
			- (-1)^{\alpha+1} \left(\sum_{l=0}^\alpha C_l {r - \alpha + l \choose \alpha}\right) \\ & \textstyle \qquad\qquad\qquad\times a'  \sum_{j}  (-1)^{j} {{1}  \choose j} \delta_{L \equiv_{p-1} \alpha} \left[{\begin{Lmatrix} 1 & 0 \\ 0 & p \end{Lmatrix}},\theta^\alpha y^{j(p-1)}x^{r-\alpha(p+1)-j(p-1)}\right] 
	\\ & \textstyle \qquad\qquad
			+ a' \sum'_{j' \in S'}  A_{j'} \left[g_{j'}, \theta^{\alpha+1} h_{j'}'\right] 
	 \\ & \textstyle \qquad \equiv 
			 \left(\sum' D_i {i \choose \alpha}\right) \left[1,\theta^\alpha x^{p-1}y^{r-(\alpha+1)(p+1)+2}\right] +  \mathcal V \sum_{j\in S} [g_j,\theta^{\alpha+1} h_j'] + \mathcal V O(p^{\delta})
		 	\\ & \textstyle \qquad\qquad
			- a'  \sum_{j}  (-1)^{j} {{1}  \choose j} \sum_\mu \left[{\begin{Lmatrix} p & [\mu] \\ 0 & 1 \end{Lmatrix}},\theta^\alpha x^{j(p-1)}y^{r-\alpha(p+1)-j(p-1)}\right] 
	\\ & \textstyle \qquad\qquad
			- (-1)^{\alpha+1} \left(\sum_{l=0}^\alpha C_l {r - \alpha + l \choose \alpha}\right) \\ & \textstyle \qquad\qquad\qquad\times a'  \sum_{j}  (-1)^{j} {{1}  \choose j} \delta_{L \equiv_{p-1} \alpha} \left[{\begin{Lmatrix} 1 & 0 \\ 0 & p \end{Lmatrix}},\theta^\alpha y^{j(p-1)}x^{r-\alpha(p+1)-j(p-1)}\right] 
	\\ & \textstyle \qquad\qquad
			+ a' \sum'_{j' \in S'}  A_{j'} \left[g_{j'}, \theta^{\alpha+1} h_{j'}'\right] 
	 \pmod {p^{{m+1-\alpha}}},
	\end{align*}
where $\delta = m+1-\alpha>0$. If $\smash{\textstyle \mathfrak v_{\alpha} (C_1,\ldots,C_\alpha) < v(a)-\alpha}$, then
	\[\textstyle \textstyle \Psi_\alpha \left(\overline{(T - a) (\mathfrak m_{\alpha} (C_1,\ldots,C_\alpha))^{-1}(\smash{
	\tau_\dagger-a^{-1}\tau_0})}\right) = [1,X^{(2\alpha-r \bmod p-1)}]\]
is in the image under $\Psi_\alpha$ of the kernel of reduction, and generates $\smash{I(\sigma_{(2\alpha-r \bmod p-1)} (r-\alpha))}$. If, on the other hand, $\smash{\textstyle \mathfrak v_{\alpha} (C_1,\ldots,C_\alpha) > v(a)-\alpha}$, then
	\[\textstyle \textstyle \textstyle \Psi_\alpha \left(\overline{(T - a) (a')^{-1}(\smash{
	\tau_\dagger-a^{-1}\tau_0})}\right) = T[1,X^{(2\alpha-r \bmod p-1)}]\]
is in the image under $\Psi_\alpha$ of the kernel of reduction, and generates $\smash{T \left(I(\sigma_{(2\alpha-r \bmod p-1)} (r-\alpha))\right)}$. 
\end{proof}

\emph{Remark.} In the statement of theorem~\ref{theorem26}, we can replace the binomials ${i \choose w}$ by $f_w(i)$, where $f_w$ is any polynomial of degree~$w$ whose top coefficient is a unit in $\mathbb Z_p$. This can be proved by a simple application of row operations. %

 \section{The case when $\xxxxxxxxxx {v(a) = 1}$} \label{Sect3}

\subsection{The main theorem}

We follow the approach outlined in~\cite{b2,b4}. The main result in this section is theorem~\ref{t1}.
\begin{theorem}\label{t1}
          Let $t \geqslant 1$ be an integer, and let $r = t(p-1) + s$, with $s \in \{2,\ldots,p-1\}$. Let $a \in \overline {\mathbb Q}_p$ be such that $v(a) = 1$. Then, for any \[\textstyle \gamma \in \left(T-\overline{p/a} \cdot \frac{t}{s}\right)(I(\sigma_{p-1-s}(s))),\] there is a $\tau$ such that $(T - a) \tau$ and $p^2\tau$ are integral, and $\textstyle\Psi (\overline{\textstyle (T - a) \tau}) = \gamma$.
\end{theorem}

Before giving the proof of theorem~\ref{t1}, we will first prove a lemma.
 
\subsection{A lemma}

\begin{lemma}\label{l2''}
Let $t \geqslant 1$ be an integer, and let $a \in \overline {\mathbb Q}_p$ be such that $v(a) = 1$. Let $r = t(p-1) + s$, with $s \in \{2,\ldots,p-1\}$. Then there is some $\tau$ such that $(T - a) \tau$ and $p^2\tau$ are integral, and
	\[\textstyle  \Psi (\overline{\textstyle (T - a) \tau}) \textstyle  = \left( T -  \overline{p/a} \cdot \frac{t}{s} \right) \left[1, X^{p-1-s}\right].\]
\end{lemma}

\begin{proof}
Let $l \in \{1,\ldots,t\}$, and let $\smash{\varphi_{l,g} = \left[g,y^r - x^{l(p-1)}y^{r-l(p-1)}\right]}$. Then
	\begin{align*}
	\textstyle (T - a) \varphi_{l,g} & \textstyle \equiv \sum_{\lambda \ne 0}  \big[g  {\begin{Lmatrix} p & [\lambda] \\ 0 & 1 \end{Lmatrix}},\left\{(-[\lambda])^{r} - (-[\lambda])^{r-l(p-1)} \right\} x^{r} 
	\\ &
	\textstyle
	\qquad
	+ \left\{r(-[\lambda])^{r-1} - (r-l(p-1))(-[\lambda])^{r-l(p-1)-1}\right\} px^{r-1}y\big]
	\\ &
	\textstyle
	\qquad
	+ \left[g{\begin{Lmatrix} 1 & 0 \\ 0 & p \end{Lmatrix}},y^r\right]- \left[g,a(y^r - x^{l(p-1)}y^{r-l(p-1)})\right] \pmod {p^2},
	\end{align*} 
Note that, for $\lambda \ne 0$, $\smash{(-[\lambda])^{r} = (-[\lambda])^{r-l(p-1)}}$  and $ \smash{(-[\lambda])^{r-1} = (-[\lambda])^{r-l(p-1)-1} = (-[\lambda])^{s-1}}$. Hence
	\begin{align*}
	 \textstyle (T - a) \varphi_{l,g} & \textstyle \equiv \sum_{\lambda \ne 0}  \left[g  {\begin{Lmatrix} p & [\lambda] \\ 0 & 1 \end{Lmatrix}},(-[\lambda])^{s-1}lp(p-1)x^{r-1}y\right] 
	 \\ &
	 \textstyle
	 \qquad
	 + \left[g{\begin{Lmatrix} 1 & 0 \\ 0 & p \end{Lmatrix}},y^r\right]- \left[g,a(y^r - x^{l(p-1)}y^{r-l(p-1)})\right] \pmod {p^2}.
	\end{align*}
Let $\smash{\textstyle  \tau' = \sum_{\mu \in \mathbb F_p} \varphi_{l,\tiny \begin{Smatrix} p & [\mu] \\ 0 & 1 \end{Smatrix}}}$. Then
	\begin{align*}
	 \textstyle (T - a) \tau' & \textstyle \equiv -pl \sum_{\lambda \ne 0,\mu}  \left[  {\begin{Lmatrix} p^2 & p[\lambda]+[\mu] \\ 0 & 1 \end{Lmatrix}},(-[\lambda])^{s-1} x^{r-1}y\right] 
	  \textstyle
	 +  \sum_{\mu}\left[{\begin{Lmatrix} 1 & [\mu] \\ 0 & 1 \end{Lmatrix}},y^r\right]
	 \\ &
	 \textstyle \qquad - \sum_{\mu} a \left[{\begin{Lmatrix} p & [\mu] \\ 0 & 1 \end{Lmatrix}}, y^r - x^{l(p-1)}y^{r-l(p-1)}\right] \pmod {p^2}.
	\end{align*}
Note that 
	\begin{align*}
	 \textstyle \sum_{\mu} \left[{\begin{Lmatrix} 1 & [\mu] \\ 0 & 1 \end{Lmatrix}},y^r\right] & \textstyle  \equiv \sum_{\mu} \left[1,(y + [\mu]x)^r\right]
	 \textstyle \equiv \sum_{i,\mu} {r \choose i} [\mu]^i \left[1, x^iy^{r-i}\right] \pmod {p^2},
	\end{align*}
and that $\sum_\mu [\mu]^i = 0$ if $i \not\in \{0,p-1,\ldots,t(p-1),(t+1)(p-1)\}$, and $\sum_\mu [\mu]^i = p-1+\delta_{i=0}$ otherwise. Consequently,
	\begin{align*}
	 \textstyle \sum_{\mu} \left[{\begin{Lmatrix} 1 & [\mu] \\ 0 & 1 \end{Lmatrix}},y^r\right] & \textstyle \equiv p{r \choose 0}  \left[1, y^{r}\right] + (p-1)\delta_{s=p-1}  {r \choose (t+1)(p-1)}  \left[1, x^{r}\right] \textstyle \\&\textstyle\qquad + (p-1) \sum_{j=1}^t {r \choose j(p-1)}  \left[1, x^{j(p-1)}y^{r-j(p-1)}\right] \pmod {p^2},
	\end{align*}
which implies that
	\begin{align*}
	 \textstyle (T - a) \tau' & \textstyle \equiv - lp\sum_{\lambda \ne 0,\mu}  \left[  {\begin{Lmatrix} p^2 & p[\lambda]+[\mu] \\ 0 & 1 \end{Lmatrix}},(-[\lambda])^{s-1} x^{r-1}y\right] 
	  \\ & \textstyle \qquad 
	 +p(p-1)\sum_{j=1}^t p^{-1} {r \choose j(p-1)}  \left[1, x^{j(p-1)}y^{r-j(p-1)}\right] + p  \left[1, y^{r}\right] 
	 \\ & \textstyle \qquad 
	  + (p-1) \delta_{s=p-1} \left[1, x^{r}\right]
	  - p\sum_{\mu} \frac ap \left[{\begin{Lmatrix} p & [\mu] \\ 0 & 1 \end{Lmatrix}}, y^r - x^{l(p-1)}y^{r-l(p-1)}\right]\pmod {p^2}.
	\end{align*}
Note also that
	\begin{align*}
	\textstyle (T - a) \varphi_{l,{\tiny \begin{Smatrix} 0 & 1 \\ p & 0 \end{Smatrix}}} & \textstyle  \textstyle \equiv \sum_{\lambda \ne 0}  \left[{\begin{Lmatrix} 0 & 1 \\ p^2 & p[\lambda]+1 \end{Lmatrix}},(-[\lambda])^{-1}lp(p-1)x^{r-1}y\right] 
	 \\ &
	 \textstyle
	 \qquad
	 + \left[{\begin{Lmatrix} 0 & p \\ p & 0 \end{Lmatrix}},y^r\right] - a \left[{\begin{Lmatrix} 0 & 1 \\ p & 0 \end{Lmatrix}},y^r - x^{l(p-1)}y^{r-l(p-1)}\right] 
	 \\ & \textstyle \equiv \sum_{\lambda \ne 0}  \left[{\begin{Lmatrix} 0 & 1 \\ p^2 & p[\lambda]+1 \end{Lmatrix}},(-[\lambda])^{-1}lp(p-1)x^{r-1}y\right] 
	 \\ &
	 \textstyle
	 \qquad
	 + \left[1,x^r\right] - a \left[{\begin{Lmatrix} 1 & 0 \\ 0 & p \end{Lmatrix}},x^r - y^{l(p-1)}x^{r-l(p-1)}\right] 
	 \pmod {p^2}.
	\end{align*}
Let $\tau'' = \tau'+\delta_{s=p-1} \varphi_{l,{\tiny \begin{Smatrix} 0 & 1 \\ p & 0 \end{Smatrix}}}$, so that $\tau''$ is integral. Then
	\begin{align*}
	\textstyle (T - a) \tau'' & \textstyle \equiv - lp\sum_{\lambda \ne 0,\mu}  \left[  {\begin{Lmatrix} p^2 & p[\lambda]+[\mu] \\ 0 & 1 \end{Lmatrix}},(-[\lambda])^{s-1} x^{r-1}y\right] + p \left[1, y^{r}\right] + p  \delta_{s=p-1} \left[1, x^{r}\right]
	 \\ & \textstyle \qquad 
	 +p(p-1)\sum_{j=1}^t p^{-1} {r \choose j(p-1)}  \left[1, x^{j(p-1)}y^{r-j(p-1)}\right] 
	  \\ & \textstyle \qquad - p\sum_{\mu} \frac ap \left[{\begin{Lmatrix} p & [\mu] \\ 0 & 1 \end{Lmatrix}}, y^r - x^{l(p-1)}y^{r-l(p-1)}\right] 
	 \\ &\qquad \textstyle +p\delta_{s=p-1} \sum_{\lambda \ne 0}  \left[{\begin{Lmatrix} 0 & 1 \\ p^2 & p[\lambda]+1 \end{Lmatrix}},(-[\lambda])^{-1}l(p-1)x^{r-1}y\right] 
	 \\ &
	 \textstyle
	 \qquad
	 - p\delta_{s=p-1}  \frac ap \left[{\begin{Lmatrix} 1 & 0 \\ 0 & p \end{Lmatrix}},x^r - y^{l(p-1)}x^{r-l(p-1)}\right] 
	   \pmod {p^2}.
	\end{align*} 
Let $\tau''' = \frac{p}{a} \sum_{j=1}^t b_j \varphi_{j,1}$, where $b_j \in \mathbb Q_p$ are such that $pb_j \in \mathbb Z_p$ and $\sum_{j=1}^t b_j = 0$. Then $p\tau'''$ is integral, and
	\begin{align*}
	 \textstyle (T - a) \tau''' & \textstyle \equiv - \sum_{\lambda \ne 0}  p \left(\sum_{j=1}^tjb_j\right)\left[{\begin{Lmatrix} p & [\lambda] \\ 0 & 1 \end{Lmatrix}},(-[\lambda])^{s-1}\frac{p}{a} x^{r-1}y\right] 
	   \\ & \textstyle \qquad
	 + p \sum_{j=1}^t b_j \left[1,x^{j(p-1)}y^{r-j(p-1)}\right] \pmod {p^2}.
	\end{align*}
In particular, this equation holds true when $b_j = c_j - p^{-1}{r \choose j(p-1)}$, where $c_j\in \mathbb Z_p$ are such that  $\sum_{j=1}^t c_j = \frac ts$. Consequently, if
	\begin{align*}
	  \tau'''' & \textstyle = \tau'' + \tau''' =  \sum_{\mu \in \mathbb F_p} \varphi_{l,\tiny \begin{Smatrix} p & [\mu] \\ 0 & 1 \end{Smatrix}}+ \delta_{s=p-1} \varphi_{l,{\tiny \begin{Smatrix} 0 & 1 \\ p & 0 \end{Smatrix}}}  +  \frac{p}{a} \sum_{j=1}^t \left(c_j - p^{-1}{r \choose j(p-1)}\right) \varphi_{j,1},
	\end{align*}
then  $p\tau''''$ is integral, and
	\begin{align*}
	\textstyle (T - a) \tau'''' & \textstyle \equiv - lp\sum_{\lambda \ne 0,\mu}  \left[  {\begin{Lmatrix} p^2 & p[\lambda]+[\mu] \\ 0 & 1 \end{Lmatrix}},(-[\lambda])^{s-1} x^{r-1}y\right]   + p \left[1, y^{r}\right] + p  \delta_{s=p-1} \left[1, x^{r}\right]
	 \\ & \textstyle \qquad 
	 +p\delta_{s=p-1} \sum_{\lambda \ne 0}  \left[{\begin{Lmatrix} 0 & 1 \\ p^2 & p[\lambda]+1 \end{Lmatrix}},(-[\lambda])^{-1}l(p-1)x^{r-1}y\right] 
	 \\ & \textstyle \qquad - \sum_{\lambda \ne 0}  p \left(\sum_{j=1}^tj\left(c_j - p^{-1}{r \choose j(p-1)}\right)\right)\left[{\begin{Lmatrix} p & [\lambda] \\ 0 & 1 \end{Lmatrix}},(-[\lambda])^{s-1}\frac{p}{a} x^{r-1}y\right] 
	 \\ & \textstyle \qquad 
	 +p(p-1)\sum_{j=1}^t c_j \left[1, x^{j(p-1)}y^{r-j(p-1)}\right]
	  - p\sum_{\mu} \frac ap \left[{\begin{Lmatrix} p & [\mu] \\ 0 & 1 \end{Lmatrix}}, y^r - x^{l(p-1)}y^{r-l(p-1)}\right] 
	 \\ &\qquad \textstyle 
	 - p\delta_{s=p-1}  \frac ap \left[{\begin{Lmatrix} 1 & 0 \\ 0 & p \end{Lmatrix}},x^r - y^{l(p-1)}x^{r-l(p-1)}\right] 
	   \pmod {p^2}.
	\end{align*}
Note that 
	\[\textstyle \smash{v\left(\sum_{j=1}^tj\left(c_j - p^{-1}{r \choose j(p-1)}\right)\right)} \geqslant 0,\]
due to lemma~\ref{lemma7}. Let $\tau = a^{-1}\tau''''$. Then $(T-a)\tau$ and $p^2\tau$ are integral. Moreover, $\Psi y^r = 0$ and $\Psi x^{r-1} y = 0$,
since $s \ne 1$.  Consequently,
	\begin{align*}
	 \textstyle  \Psi (\overline{\textstyle (T - a) \tau})  & \textstyle 
	 = \left(  T - \overline{p/a} \sum_{j=1}^t c_j\right) \left[1, X^{p-1-s}\right]= \left(  T - \overline{p/a} \cdot \frac{t}{s}\right) \left[1, X^{p-1-s}\right]. 
	\end{align*}
\end{proof}

\begin{proof}[Proof of  theorem~\ref{t1}]
Follows from lemma~\ref{l2''} %
 due to the fact that the module $\sigma_{p-1-s}(s)$ is simple and hence generated by $X^{p-1-s}$.
\end{proof}

\subsection{Completing the proof of theorem~\ref{t0.0}}

\begin{lemma}\label{l4.2}
Let $t \geqslant 2$ be an integer, and let $a \in \overline {\mathbb Q}_p$ be such that $v(a) = 1$. Suppose that $r = t(p-1) + s$, with $s \in \{2,4,\ldots,p-1\}$. If $p \nmid t$, then $\smash{\overline \Theta_{r+2,a}}$ is one of the following: 
\begin{enumerate}[\tttttt] 
	\item either $\smash{\overline \Theta_{r+2,a}^{\mathrm{ss}} \cong  \pi(\widetilde{s-2},\lambda,\omega)^{\mathrm{ss}}\oplus \pi(p-1-s,\lambda^{-1},\omega^s)^{\mathrm{ss}}}$, where $\lambda = \overline{a/p} \cdot \frac{s}{t}$, and $\smash{\widetilde n}$ is the integer in $\{1,\ldots,p-1\}$ which is congruent to $n$ modulo~$p-1$; or $\overline \Theta_{r+2,a} \cong \pi(s-2,0,\omega)$.
\end{enumerate}
If, on the other hand, $p \mid t$, then $\smash{\overline \Theta_{r+2,a}}$ is one of the following: 
\begin{enumerate}[\tttttt] 
	\item either $\overline \Theta_{r+2,a} \cong  \pi(s,0,1)$; or $\overline \Theta_{r+2,a} \cong \pi(s-2,0,\omega)$.
\end{enumerate} 
\end{lemma}

\begin{proof}
First, we will consider the case when $r \ne 2p$. A consequence of lemma~\ref{l4.1.1}, with $A = I(\sigma_r)$, with $B = I(\sigma^{p-1-s}(s))$, with $\beta$ being induced by $\Psi$, with $\smash{B' = \left(T-\overline{p/a} \cdot \frac{t}{s}\right)(I(\sigma_{p-1-s}(s)))}$, with $C = \overline \Theta_{r+2,a}$, and with $\gamma$ being the surjection $\sigma_{r} \twoheadrightarrow \overline{\Theta}_{r+2,a}$, is the fact that there is some $C_1 \subseteq \overline \Theta_{r+2,a}$ such that $\overline \Theta_{r+2,a}/C_1$ is isomorphic to a quotient of $I(\sigma_{p-1-s}(s))/B'$ and $C_1$ is isomorphic to $I(\Ker\Psi)/(I(\Ker\Psi)\cap X(r+2,a))$. Here the choice of $B'$ can be made due to theorem~\ref{t1}. Firstly, suppose that $p \nmid t$. Note that $\smash{I(\sigma_{p-1-s}(s))\big/\!\left(T-\overline{p/a} \cdot \frac{t}{s}\right) \cong \pi(p-1-s,\lambda^{-1},\omega^s)}$, where $\lambda = \overline{a/p} \cdot \frac{s}{t}$. Hence, $\smash{\overline \Theta_{r+2,a}}$ is a quotient of a module which has a series with factors
	\[\textstyle \pi(p-1-s,\lambda^{-1},\omega^s),\, I\left(\langle \overline \theta\sigma_{r - (p+1)}, y^r\rangle_{\mathrm{GL}_2(\mathbb F_p)}\right)\big/\left(I\left(\langle \overline \theta\sigma_{r - (p+1)}, y^r\rangle_{\mathrm{GL}_2(\mathbb F_p)}\right) \cap X(r+2,a)\right).\]
Since $X(r+2,a)$  contains  $I(Y_r)$, due to lemma~\ref{l4.1}, this is a quotient of a module which has a series with factors
	\[\textstyle \pi(p-1-s,\lambda^{-1},\omega^s),\, I\left(\langle \overline \theta\sigma_{r - (p+1)}, y^r\rangle_{\mathrm{GL}_2(\mathbb F_p)}\big/\langle \overline \theta^2 \sigma_{r - 2(p+1)}, y^r\rangle_{\mathrm{GL}_2(\mathbb F_p)}\right).\]
It can be shown, by several applications of the isomorphism theorems, that
	\[\smash{\textstyle \langle \overline \theta\sigma_{r - (p+1)}, y^r\rangle_{\mathrm{GL}_2(\mathbb F_p)}\big/\langle \overline \theta^2 \sigma_{r - 2(p+1)} , y^r\rangle_{\mathrm{GL}_2(\mathbb F_p)}}\]
is a quotient of $M = \textstyle \overline \theta\sigma_{r - (p+1)} /\,\overline \theta^2 \sigma_{r - 2(p+1)}$. Due to parts~(a) and~(c) of lemma~3.2 in~\cite{b1}, there is a series $M = M_2 \supseteq M_1 \supseteq M_0 = \{0\}$, such that $M_2/M_1$ is isomorphic to $\sigma_{(2-s \bmod p-1)} (s-1)$, and $M_1$ is isomorphic to $\sigma_{\widetilde{s-2}} (1)$, where $\smash{\widetilde n}$ is the integer in $\{1,\ldots,p-1\}$ which is congruent to $n$ modulo~$p-1$. Hence $\smash{\overline \Theta_{r+2,a}}$ is a quotient of a module which has a series with factors
	\[\textstyle \pi(p-1-s,\lambda^{-1},\omega^s),\, I\left(\sigma_{(2-s \bmod p-1)}(s-1)\right),\, I\left(\sigma_{\widetilde{s-2}}(1)\right).\]
If $s \not\in\{1,2,3\}$, the only pair of integers from the set $\smash{\{p-1-s,\widetilde{s-2},(2-s\bmod p-1)\}}$ which adds up to $(-2\bmod{p-1})$ is the pair $\smash{(p-1-s,\widetilde{s-2})}$, since $p \geqslant 5$. If $s=2$ and $p \geqslant 5$, then there are also the pairs $(p-3,0)$ and $(p-3,p-1)$, which can't give a reducible representation since $\omega\ne \omega^{p-1}$. If $p=3$ and $s=2$, then  it can't be the case that $\overline \Theta_{r+2,a}^{\mathrm{ss}} \cong  \pi(\nu,\lambda,\omega)^{\mathrm{ss}}\oplus \pi(0,\lambda^{-1},\omega)^{\mathrm{ss}}$ with $\nu \in\{0,2\}$, since $\omega \ne \omega^{\nu+2}$. Consequently, due to the classification in theorem~\ref{t0.1}, either $\smash{\overline \Theta_{r+2,a}^{\mathrm{ss}} \cong  \pi(\widetilde{s-2},\lambda,\omega)^{\mathrm{ss}}\oplus \pi(p-1-s,\lambda^{-1},\omega^s)^{\mathrm{ss}}}$, or $\smash{\overline \Theta_{r+2,a} \cong \pi(\widetilde{s-2},0,\omega)}$.
 
Now suppose that $p \mid t$. Then, similarly, $\smash{\overline \Theta_{r+2,a}}$ is a quotient of a module which has a series with factors
	\[\textstyle \pi(s,0,1), \,  I\left(\sigma_{(2-s \bmod p-1)}(s-1)\right), \, I\left(\sigma_{\widetilde{s-2}}(1)\right).\]
Then $\smash{\overline \Theta_{r+2,a}}$ must be irreducible, since if $p \geqslant 5$ and $s \not\in\{1,3\}$, then there is no pair of integers from the set $\smash{\{\widetilde{s-2},(2-s\bmod p-1)\}}$ which adds up to $(-2\bmod{p-1})$, and if $p = 3$ and $s \not\in\{1,3\}$, then it can't be the case that $\overline \Theta_{r+2,a}^{\mathrm{ss}} \cong  \pi(\nu,\lambda,\omega)^{\mathrm{ss}}\oplus \pi(0,\lambda^{-1},\omega)^{\mathrm{ss}}$ with $\nu \in\{0,2\}$, since $\omega \ne \omega^{\nu+2}$. Consequently, due to the classification in theorem~\ref{t0.1}, either $\overline \Theta_{r+2,a} \cong  \pi(s,0,1)$, or $\smash{\overline \Theta_{r+2,a} \cong \pi(\widetilde{s-2},0,\omega)}$.

Finally, if $r = 2p$, then theorem~\ref{t1} and the first part of this proof still apply, the only difference being that the second to last factor in the series for $\overline{\Theta}_{r+2,a}$ does not occur, so that $\smash{\overline \Theta_{r+2,a}}$ is a quotient of a module which has a series with factors
	\[\textstyle \pi(p-1-s,\lambda^{-1},\omega^s),\, I\left(\sigma_{\widetilde{s-2}}(1)\right).\]
This completes the proof of lemma~\ref{l4.2}.
\end{proof}

\begin{proof}[Proof of theorem~\ref{t0.0}] Follows from theorem~\ref{t0.1}, lemma~\ref{l4.2}, and previously known results in the case when $r \leqslant 2p-2$.
\end{proof}

 \section{Conjecture~\ref{c0000}} \label{Sect4}

\subsection{The case when $2 > v(a) > 1$}

In this subsection, we will complete the proof of theorem~\ref{t0.0''0}.

\begin{proof}[Proof of theorem~\ref{t0.0''0}]
\begin{enumerate}[\tttttt] 
\item \emph{The case when $s\ne 2$.}
The condition that $s \ne 2$ implies that $p \geqslant 5$ and $s \in \{4,\ldots,p-1\}$. Then $\smash{\overline \Theta_{r+2,a}}$ has a series whose factors are quotients of submodules of
	\[\smash{\textstyle I\left(\sigma_{p-1-s}(s)\right),\, I\left(\sigma_{p+1-s}(s-1)\right), \, I\left(\sigma_{s-2}(1)\right)}.\]
Due to the conditions $p \geqslant 5$ and $s \in \{4,\ldots,p-1\}$, the only pair of factors which can give a reducible $\smash{\overline \Theta_{r+2,a}}$ is
	$\smash{\textstyle \left(I\left(\sigma_{p-1-s}(s)\right),\, I\left(\sigma_{s-2}(1)\right)\right)}$.
Let $\alpha = 0$, and define $M_r$ by the equation $\textstyle M_r = \smash{\sum_{r>j(p-1)>0} {r \choose j(p-1)}}$. Then, in view of lemma~\ref{lemma7}, we have $\mathfrak m_0(\varnothing) = M_r$, so 
	\[\textstyle v(\mathfrak m_0(\varnothing)) = v(\frac{t}{s}p (1 + O(p))) = v(t).\]
Moreover, $v(\mathfrak m_w(\varnothing))$ can be expressed as a linear combination of $M_{r},M_{r-1},\ldots,M_{r-w}$. In particular, since $s > 3$, it follows that $v(M_{r-i}) = v(t)$, for all $i \in \{0,1,2,3\}$, which implies that $v(\mathfrak m_0(\varnothing))  \geqslant v(\mathfrak m_w(\varnothing))$, for all $w \in \{1,2,3\}$. Consequently, theorem~\ref{theorem26} can be applied to show that either $[1,X^{p+1-s}]$ is in the image under $\Psi_1$ of the kernel of reduction, or $T[1,X^{p+1-s}]$ is in the image under $\Psi_1$ of the kernel of reduction. In any case, $\smash{\overline \Theta_{r+2,a}}$ has a series whose factors are quotients of  submodules of
	\[\textstyle I\left(\sigma_{p+1-s}(s-1)\right),\,I\left(\sigma_{s-2}(1)\right).\]
These modules cannot pair up to give a reducible $\smash{\overline \Theta_{r+2,a}}$, which completes the proof in the case~$s \ne 2$.
\item \emph{The case when $s = 2$ and $t \not\equiv \{1,2\}\pmod p$.}
The matrix $\mathfrak M^{(r,1,1)}$ contains
	\[\textstyle \mathfrak M' = \begin{Xmatrix} 1 & 1 & \sum_{r-3 > i(p-1) > 0} {r \choose i(p-1)+1} \\ 1-r & 1 & \sum_{r-3 > i(p-1) > 0} {r \choose i(p-1)+1} i \end{Xmatrix} = \begin{Xmatrix} 1 & 1 & 2(1-r) \\ 1-r & 1 & (1-r)(2-r) \end{Xmatrix}\]
as a submatrix, which has full rank over $\mathbb F_p$ (and therefore contains $(0,1)^T$ in its range) for $r \not\equiv 0 \pmod p$. Moreover, if $v_1,v_2,v_3$ denote the columns of $\mathfrak M'$, then 
\[\textstyle (r-1)v_1+(r-1)v_2+v_3 = 0,\]
so there is an element $\nu \in \overline{\theta}\sigma_{r-(p+1)}$ in the subspace $N \subseteq \overline{\theta}\sigma_{r-(p+1)}$ which corresponds to $\sigma_{p-1}$, such that the coefficient of $\overline{\theta} x^{r-(p-1)}$ in $\nu$ is $r-1$. Consequently, if $r \not\equiv 1 \pmod p$, then this element does not belong to $\smash{\overline{\theta}^2 \sigma_{r-2(p+1)}}$. Therefore, unless $\smash{p \mid r(r-1)}$, we have that $\smash{\overline \Theta_{r+2,a}}$ has a series whose factors are quotients of  submodules of
	\[\textstyle I\left(\sigma_{p+1-s}(s-1)\right),\,I\left(\sigma_{s-2}(1)\right),\]
which cannot pair up to give a reducible $\smash{\overline \Theta_{r+2,a}}$. This completes the proof in the case when $s=2$ and $t \not\equiv \{1,2\}\pmod p$.
\item \emph{The case when $s = 2$ and $t \equiv \{1,2\}\pmod p$.}
Calling the method \verb|find_m_w(1,0,1)| in the second program in the appendix %
 shows that in both cases the relevant $\mathfrak m_w(\varnothing)$ are either zero or polynomials of the type $Cr^\alpha (r-1)^\beta$. More specifically,
\begin{align*}
\textstyle \mathfrak m_2(\varnothing) & =  \textstyle \frac{1}{2} r (r-1), \\
\textstyle \mathfrak m_3(\varnothing) & = \textstyle -\frac12 r^2 (r-1).
\end{align*}
In particular, if $t \equiv \{1,2\}\pmod p$ then all of the relevant $\mathfrak m_w(\varnothing)$ vanish. So theorem~\ref{theorem26} applies, and shows that the factor which is a quotient of $\smash{I\left(\sigma_{p-3}(2)\right)}$ is in fact a quotient of $\smash{\pi(p-3,0,\omega^2)}$, which implies that $\smash{\overline \Theta_{r+2,a}}$ has a series whose factors are quotients of  submodules of
	$\smash{\textstyle I\left(\sigma_{p+1-s}(s-1)\right),\,I\left(\sigma_{s-2}(1)\right)}$,
which cannot pair up to give a reducible $\smash{\overline \Theta_{r+2,a}}$. Consequently, $\smash{\overline \Theta_{r+2,a}}$ must be irreducible. This completes the proof in the case when $s=2$ and $t \equiv \{1,2\}\pmod p$.
\end{enumerate}
\end{proof}

\subsection{The case when $3 > v(a) > 2$}

In this subsection, we will complete the proof of theorem~\ref{t0.0''}.

\begin{proof}[Proof of theorem~\ref{t0.0''}]
The case when $p=3$ is vacuous, so in the remainder of this proof we are going to assume that $p > 3$.
\begin{enumerate}[\tttttt] 
\item \emph{The case when $s \not \in \{2,4\}$ and $t \not \equiv\{0,1\} \pmod{p}$.}
The condition that $s \not \in \{2,4\}$ implies that $p \geqslant 7$ and $s \in \{6,\ldots,p-1\}$, and theorem~\ref{theorem26} applies. Then $\smash{\overline \Theta_{r+2,a}}$ has a series whose factors are quotients of  submodules of
	\[\textstyle I\left(\sigma_{p-1-s}(s)\right),\, I\left(\sigma_{p+1-s}(s-1)\right), \, I\left(\sigma_{s-2}(1)\right), \,I\left(\sigma_{p+3-s}(s-2)\right),\,I\left(\sigma_{s-4}(2)\right).\]
Due to the conditions $p \geqslant 7$ and $s \in \{6,\ldots,p-1\}$, the only pairs of factors which can give a reducible $\smash{\overline \Theta_{r+2,a}}$ are
	\[\textstyle \left(I\left(\sigma_{p-1-s}(s)\right),\, I\left(\sigma_{s-2}(1)\right)\right),\,\, \left(I\left(\sigma_{p+1-s}(s-1)\right),\,I\left(\sigma_{s-4}(2)\right)\right).\]
Now, suppose that $t \not \in \{0,1\}$. For $\alpha = 1$, after choosing the constant $C_1 = -\frac{p}{s-1} + O(p^2)$ in theorem~\ref{theorem26}, where $O(p^2)$ is such that $\mathfrak m_0(C_1) = 0$, we have $v(\mathfrak m_w(C_1)) \geqslant 1$ for $w \in \{1,2,3,4\}$, and 
	\[\textstyle \mathfrak m_1(C_1) = \sum_{i > 0} \sum_{l=0}^\alpha C_l {r - 1 + l \choose i(p-1) + l} i = \frac{pt(r-1)}{(s-1)(s-2)}-\frac{pt}{s-1} = \frac{pt(1-t)}{(s-1)(s-2)},\]
so $v(\mathfrak m_1(C_1)) = 1$. Consequently, theorem~\ref{theorem26} can be applied to show that $[1,X^{p+1-s}]$ is in the image under $\Psi_1$ of the kernel of reduction. Since $I_{s-2}(1)$ is not semi-simple, due to lemma~\ref{lemma6}, $[1,X^{p+1-s}]$ generates the whole of $I_{s-2}(1)$, which implies that $\smash{\overline \Theta_{r+2,a}}$ has a series whose factors are quotients of  submodules of
	\[\textstyle I\left(\sigma_{p-1-s}(s)\right), \,I\left(\sigma_{p+3-s}(s-2)\right),\,I\left(\sigma_{s-4}(2)\right).\]
These modules cannot pair up to give a reducible $\smash{\overline \Theta_{r+2,a}}$, which completes the proof in the case~$t \not \equiv\{0,1\} \pmod{p}$.
\item \emph{The case when $s\not\in\{2,4\}$ and $t\equiv 1\pmod p$.}
Due to the remark succeeding the proof of theorem~\ref{theorem26}, we can replace the binomials ${i \choose w}$ in the statement of the theorem by ${i(p-1) \choose w}$. Then we have $v(\mathfrak m_1(C_1)) \geqslant 2$. Moreover, $v(C_1) \geqslant 1$, and 
	\begin{align*}
	& \textstyle v\left( (-1)^{w+1} + \sum_{i > 0} {r \choose i(p-1) + 1} {i(p-1) \choose w}\right) 
	\\ & \textstyle \qquad = v\left((-1)^{w+1} + O(p) +  \sum_{i > 0} {r \choose i(p-1) + 1} \sum_v (-1)^{w-v}{i(p-1) + 1 \choose v}\right)
	\\ & \textstyle \qquad= v\left( (-1)^{w+1} + O(p) +\sum_v  \sum_{i > 0} {r-v \choose i(p-1) + 1-v} (-1)^{w-v} {r \choose v}\right)
	\\ & \textstyle \qquad= v\left(O(p)+\sum_{v>0} \sum_{i > 0} {r-v \choose i(p-1) + 1-v} (-1)^{w-v} {r \choose v}\right) \geqslant 1.
	\end{align*}
The last part follows since $\smash{ v\left(\sum_{i > 0} {r-v \choose i(p-1) + 1-v}\right) \geqslant 1}$, whenever $v > 0$, by lemma~\ref{lemma7wefwfwerfwf}. So we find, by lemmas~\ref{lemma7} and~\ref{lemma742}, that
	\begin{align*}
	\textstyle \mathfrak m_w(C_1) & \textstyle = \sum_{i > 0} \sum_{l=0}^1 C_l {r - 1 + l \choose i(p-1) + l} {i(p-1) \choose w}
	\\ & \textstyle =O(p^2)-(-1)^{w} \frac{p}{s-1} + \sum_{i > 0} {r - 1 \choose i(p-1)} {i(p-1) \choose w}
	\\ & \textstyle =O(p^2)-(-1)^{w} \frac{p}{s-1} + {r-1 \choose w} \sum_{i > 0} {r - 1 - w \choose i(p-1)-w}  
	\\ & \textstyle =O(p^2)-(-1)^{w} \frac{p}{s-1} + {r-1 \choose w} \sum_{i > 0} \sum_v (-1)^{v} {w \choose v} {r - 1 - v \choose i(p-1)}  
	\\ & \textstyle 
	=O(p^2) -(-1)^{w} \frac{p}{s-1}+ p {r-1 \choose w} \sum_v (-1)^{v} {w \choose v} \frac{1}{s-1-v}  
	\\ & \textstyle 
	=O(p^2) -(-1)^{w} \frac{p}{s-1}+ p {r-1 \choose w} (-1)^{w} \frac{w!}{(s-1)_{w+1}} 
	\\ & \textstyle 
	=O(p^2) -(-1)^{w} \frac{p}{s-1}+ (-1)^{w}p \frac{(r-1)_w}{(s-1)_{w+1}} 
	\\ & \textstyle 
	= O(p^2) -(-1)^{w} \frac{p}{s-1}+ (-1)^{w}p \frac{(s-2)_w}{(s-1)_{w+1}} 
	\\ & \textstyle 
	= O(p^2) -(-1)^{w} \frac{p}{s-1}+ (-1)^{w} \frac{p}{s-1} = O(p^2),
	\end{align*}
for  $w \in \{1,2,3,4\}$. Consequently, theorem~\ref{theorem26} can be applied to show that $[1,X^{p+1-s}]$ is in the image under $\Psi_1$ of the kernel of reduction.
Similarly, when $\alpha = 0$, we have $v(\mathfrak m_w(\varnothing)) \geqslant 1$ for $w \in \{0,1,2,3,4,5\}$, with equality for $w = 0$ since $\mathfrak m_0(\varnothing) = \frac tsp + O(p^2)$. So theorem~\ref{theorem26} can be applied to show that $[1,X^{p-1-s}]$ is in the image under $\Psi_0$ of the kernel of reduction. This implies that $\smash{\overline \Theta_{r+2,a}}$ has a series whose factors are quotients of  submodules of
	\[\textstyle I\left(\sigma_{s-2}(1)\right), \,I\left(\sigma_{p+3-s}(s-2)\right),\,I\left(\sigma_{s-4}(2)\right).\]
These modules cannot pair up to give a reducible $\smash{\overline \Theta_{r+2,a}}$, which completes the proof in the case~$t\equiv 1\pmod p$.
\item \emph{The case when $s\not\in\{2,4\}$ and $t\equiv 0\pmod p$.}
A similar calculation to the one in the previous paragraph shows that, for $\alpha = 1$ and $w \in \{0,1,2,3,4\}$, we have $v(\mathfrak m_w(C_1)) \geqslant 2$, so theorem~\ref{theorem26} can be applied to show that $T[1,X^{p+1-s}]$ is in the image under $\Psi_1$ of the kernel of reduction. For $\alpha = 0$ and $w \in \{0,1,2,3,4,5\}$, a similar calculation shows that $v(\mathfrak m_w(\varnothing)) \geqslant 2$. Suppose we define $M_r$ by the equation $\textstyle M_r = \smash{\sum_{r>j(p-1)>0} {r \choose j(p-1)}}$. Then, in view of lemma~\ref{lemma7}, we have $\mathfrak m_0(\varnothing) = M_r$, so 
	\[\textstyle v(\mathfrak m_0(\varnothing)) = v(\frac{t}{s}p (1 + O(p))) = v(t).\]
Moreover, $v(\mathfrak m_w(\varnothing))$ can be expressed as a linear combination of $M_{r},M_{r-1},\ldots,M_{r-w}$. In particular, since $s > 5$, it follows that $v(M_{r-i}) = v(t)$, for all $i \in \{0,1,2,3,4,5\}$, which implies that $v(\mathfrak m_0(\varnothing))  \geqslant v(\mathfrak m_w(\varnothing))$, for all $w \in \{1,2,3,4,5\}$. Therefore, if $v(\mathfrak m_0(\varnothing)) \geqslant 3$, then $v(\mathfrak m_w(\varnothing)) \geqslant 3$ for $w \in \{1,2,3,4,5\}$, and then the proof of the case when $t\equiv 0\pmod p$ can be completed in the same fashion as the proof of the case $t\equiv 1\pmod p$.
\item \emph{The case when $s=2$ and $r \not\equiv \{0,1\} \pmod{p}$.}
In this case, 
	\[\textstyle \smash{\overline{(T-a)(p^{-2}\phi^\ast_{r,2,3,2,{\begin{Lmatrix} p & 0 \\ 0 & 1 \end{Lmatrix}}})}} = [1,x^2y^{r-2}]\]
gets mapped to $[1,Y^{p-3}]$ by $\Psi_0$. Consequently, theorem~\ref{theorem26} can be applied to show that $[1,Y^{p-3}]$ is in the image under $\Psi_0$ of the kernel of reduction. Moreover, the matrix $\mathfrak M^{(r,2,1)}$ contains
	\[\textstyle \mathfrak M' = \begin{Xmatrix} 1 & 1 & \sum_{r-3 > i(p-1) > 0} {r \choose i(p-1)+1} \\ 1-r & 1 & \sum_{r-3 > i(p-1) > 0} {r \choose i(p-1)+1} i \end{Xmatrix} = \begin{Xmatrix} 1 & 1 & 2(1-r) \\ 1-r & 1 & (1-r)(2-r) \end{Xmatrix}\]
as a submatrix, which has full rank over $\mathbb F_p$ (and therefore contains $(0,1)^T$ in its range) for $r \not\equiv 0 \pmod p$. Moreover, if $v_1,v_2,v_3$ denote the columns of $\mathfrak M'$, then 
	\[\textstyle (r-1)v_1+(r-1)v_2+v_3 = 0.\]
Therefore, there is an element $\nu \in \overline{\theta}\sigma_{r-(p+1)}$ which belongs to the subspace $N \subseteq \overline{\theta}\sigma_{r-(p+1)}$ corresponding to $\sigma_{p-1}$, such that the coefficient of $\overline{\theta} x^{r-(p-1)}$ in $\nu$ is $r-1$. Consequently, if $r \not\equiv 1 \pmod p$, then this element does not belong to $\smash{\overline{\theta}^2 \sigma_{r-2(p+1)}}$. This means that, unless $\smash{p \mid r(r-1)}$, we have that $\smash{\overline \Theta_{r+2,a}}$ has a series whose factors are quotients of  submodules of
	\[\textstyle I\left(\sigma_{2}\right),\,I\left(\sigma_{p-3}(2)\right),\]
which cannot pair up to give a reducible $\smash{\overline \Theta_{r+2,a}}$. This completes the proof in the case when $s=2$.

Calculating the determinants of the relevant square submatrices of $\mathfrak M^{(r,m,\alpha)}$ is completely automated by the first program written in \texttt{Sage} listed in the appendix. In particular, calling the method \verb|print_the_roots_for_all_matrices(m,True)| with $\mathtt{m}=x$ lists the possible congruence classes modulo $p(p-1)$ outside of which $\smash{\overline \Theta_{r+2,a}}$ is irreducible, when $m=x$ and $s \in \{2,\ldots,2x\}$.
\item \emph{The case when $s=2$ and $r \equiv 0 \pmod{p}$.} By calling the methods \verb|exceptional_cases(2,1,1)| and \verb|gcd_for_the_matrix(2,1,1)|, we can show that $\smash{\overline \Theta_{r+2,a}}$ has a series whose factors are quotients of submodules of
	\[\textstyle I\left(\sigma_{0}(1)\right),\,I\left(\sigma_{2}\right),\,I\left(\sigma_{p-3}(2)\right).\]
For $\alpha = 1$, after choosing the constant $C_1 = -\frac{p+1}{2} + O(p^2)$ in theorem~\ref{theorem26}, where $O(p^2)$ is such that $\mathfrak m_0(C_1) = 0$, we have $v(\mathfrak m_w(C_1)) \geqslant 1$ for $w \in \{1,2,3,4\}$. There is the matrix
		\begin{align*}
		\textstyle \mathfrak M'' & \textstyle = \begin{Xmatrix} \sum_{r-1>i(p-1)+1>1} {r-1 \choose i(p-1)} & \sum_{r-1>i(p-1)+1>1} {r \choose i(p-1)+1}  \\ \sum_{r-1>i(p-1)+1>1} {r-1 \choose i(p-1)} i & \sum_{r-1>i(p-1)+1>1} {r \choose i(p-1)+1} i \end{Xmatrix}
\\		 & \textstyle \equiv_{p^2} \begin{Xmatrix} 1-r+2p & 2-2r+2p \\ 3p(1-r)-\frac{(r-1)^2}{p-1} & 2p(1-r)-\frac{(r-1)(r-2)}{p-1} \end{Xmatrix}.
	\end{align*}
Then		
	\begin{align*}
		\textstyle \mathfrak M'' \begin{Xmatrix} 1 \\ C_1 \end{Xmatrix} & \textstyle = \mathfrak M'' \begin{Xmatrix} 1 \\ -\frac{p+1}{2} \end{Xmatrix} \equiv_{p^2} \begin{Xmatrix} 0 \\ p - \frac{r}{2} \end{Xmatrix}.
	\end{align*}
We consider two cases:
\begin{enumerate}[1.]
\item $r \not\equiv 2p \pmod{p^2(p-1)}$. Then theorem~\ref{theorem26} applies, and we have that $\smash{\overline \Theta_{r+2,a}}$ has a series whose factors are quotients of submodules of
	\[\textstyle I\left(\sigma_{2}\right),\,I\left(\sigma_{p-3}(2)\right).\]
These modules cannot pair up to give a reducible $\smash{\overline \Theta_{r+2,a}}$.
\item $r \equiv 2p \pmod{p^2(p-1)}$. In particular, since $r > 2p$, we have $r \geqslant p^2(p-1)+2p$. Then we have $v(\mathfrak m_w(C_1)) \geqslant 2$, where $\mathfrak m_w(C_1)$ is as in the statement of theorem~\ref{theorem26}. This is so since the expressions in $\mathfrak m_w(C_1)$ are equivalent modulo $p^2$ for any two weights which are congruent modulo $p^2(p-1)$, and the desired statement holds true when $r = 2p$. Moreover, ${r-1 \choose 0} - \frac{p+1}{2} {r \choose 1} \equiv 1 - O(p) \equiv_p 1$, so theorem~\ref{theorem26} applies, and we can conclude that $\smash{\overline \Theta_{r+2,a}}$ has a series whose factors are quotients of submodules of
	\[\textstyle \pi(0,0,\omega),\,I\left(\sigma_{2}\right),\,I\left(\sigma_{p-3}(2)\right).\]
These modules cannot pair up to give a reducible $\smash{\overline \Theta_{r+2,a}}$.
\end{enumerate}
\item \emph{The case when $s=2$ and $r \equiv 1 \pmod{p}$.} By calling the methods \verb|exceptional_cases(2,1,1)| and \verb|gcd_for_the_matrix(2,1,1)|, we can show that $\smash{\overline \Theta_{r+2,a}}$ has a series whose factors are quotients of submodules of
	\[\textstyle I\left(\sigma_{p-1}(1)\right),\,I\left(\sigma_{2}\right),\,I\left(\sigma_{p-3}(2)\right).\]
Moreover, the output of \verb|find_m_w(2,1,1)| shows that there are polynomials
	\begin{align*}
		f_1 & \textstyle = \sum_{r - 1 > i(p-1) + 1 > 1} \Upsilon_i^{(1)} x^{i(p-1)+1}y^{r-i(p-1)-1},
\\		f_2 & \textstyle = \sum_{r - 1 > i(p-1) + 1 > 1} \Upsilon_i^{(2)} x^{i(p-1)+1}y^{r-i(p-1)-1},
	\end{align*}
with integral coefficients, such that $\left[1,f_1\right]$ and  $\left[1,f_2\right]$ are in the image of $T-a$, such that $\smash{\sum_{r - 1 > i(p-1) + 1 > 1} \Upsilon_i^{(j)} = 0}$, and such that $\smash{v(S_w^{(j)}) \geqslant 1}$, where 
	\[\textstyle S_w^{(j)} = \sum_{r - 1 > i(p-1) + 1 > 1} \Upsilon_i^{(j)} {i \choose w},\]
for all $w \in \{0,1,2,3,4\}$ and all $j \in \{1,2\}$. Moreover,
	\[\textstyle S_1^{(1)} \equiv S_1^{(2)} \equiv r - (p+1) \pmod{p^2}.\]
We consider two cases:
\begin{enumerate}[1.]
\item $r \not\equiv p+1 \pmod{p^2(p-1)}$. Then due to the proof of theorem~\ref{theorem26} there is a polynomial %
	\begin{align*}
		f' \textstyle = \sum_{r - 1 > i(p-1) + 1 > 1} \Upsilon_i' x^{i(p-1)+1}y^{r-i(p-1)-1}
	\end{align*}%
with integral coefficients, such that $\left[1,f'\right] + \left[1,\smash{\overline{\theta}^2 h'}\right] + O(p^\delta)$ is in the image of $T-a$, for some $h'$ and some $\delta > 0$, and such that
	\begin{align*}
		\textstyle \sum_{r - 1 > i(p-1) + 1 > 1} \Upsilon_i' & \textstyle = 1,
\\		\textstyle \sum_{r - 1 > i(p-1) + 1 > 1} i\Upsilon_i'  & \textstyle = -1.
	\end{align*}
Then 
	\begin{align*}
		\textstyle f'+xy^{r-1}-2x^{r-1}y & \textstyle \in  \smash{ \ker\Psi_1 \subseteq \overline{\theta} \sigma_{r-(p+1)}},
\\		\textstyle f'+xy^{r-1}-2x^{r-1}y  & \textstyle  \not\in \smash{\overline{\theta}^2 \sigma_{r-2(p+1)}},
	\end{align*}
and hence $\smash{\overline \Theta_{r+2,a}}$ has a series whose factors are quotients of submodules of
	\[\textstyle I\left(\sigma_{2}\right),\,I\left(\sigma_{p-3}(2)\right).\] 
These modules cannot pair up to give a reducible $\smash{\overline \Theta_{r+2,a}}$.
\item $r \equiv p+1 \pmod{p^2(p-1)}$. In particular, since $r > p+1$, we have $r > p^2(p-1)+p$. Then, for any $C_1$ and any $w \in \{0,1,2,3,4\}$, we have $v(\mathfrak m_w(C_1)) \geqslant 2$, where $\mathfrak m_w(C_1)$ is as in the statement of theorem~\ref{theorem26}. This is so since the expressions in $\mathfrak m_w(C_1)$ are equivalent modulo $p^2$ for any two weights which are congruent modulo $p^2(p-1)$, and the desired statement holds true when $r = p+1$. Therefore, we can arbitrarily choose $C_1$ in a way that theorem~\ref{theorem26} applies, and hence we can deduce that $\smash{\overline \Theta_{r+2,a}}$ has a series whose factors are quotients of submodules of
	\[\textstyle \pi(0,0,\omega),\,I\left(\sigma_{2}\right),\,I\left(\sigma_{p-3}(2)\right).\]
These modules cannot pair up to give a reducible $\smash{\overline \Theta_{r+2,a}}$.
\end{enumerate}
\item \emph{The case when $s=4$  and $r \not\equiv \{0,1,2,3,4\} \pmod{p}$.} Can be done in similar fashion as the case when $s=2$, or by calling the method \verb|print_the_roots_for_all_matrices(2,True)| in the first program in the appendix.
\item \emph{The case when $s=4$  and $r \equiv 0 \pmod{p}$.} In this case, only \verb|find_m_w(2,0,2)| returns that the relevant $\mathfrak m_w(\varnothing)$ are divisible by $r$. The output of the method \[\verb|gcd_for_the_matrix(2,2,2)|\] in the first program in the appendix shows that the factor in $\smash{\overline \Theta_{r+2,a}}$ corresponding to $\smash{I\left(\sigma_{0}(2)\right)}$ is trivial, since $\frac{1}{2}(r-1)(r-2)(r-3)(r-4) \ne 0$. Printing \[\verb|A.transpose().kernel().basis_matrix()|\] in the method \verb|exceptional_cases(2,2,2)| in the first program in the appendix shows that the factor in $\smash{\overline \Theta_{r+2,a}}$ corresponding to $\smash{I\left(\sigma_{p-1}(2)\right)}$ is trivial, since $\frac{1}{2}(r-2)(r-3) \ne 0$. These three facts together show that $\smash{\overline \Theta_{r+2,a}}$ has a series whose factors are quotients of  submodules of
	\[\textstyle \pi(p-5,0,\omega^4),\, I\left(\sigma_{p-3}(3)\right), \, I\left(\sigma_{2}(1)\right).\]
These modules cannot pair up to give a reducible $\smash{\overline \Theta_{r+2,a}}$, which completes this case of the proof.
\item \emph{The case when $s=4$  and $r \equiv \{1,2%
   \} \pmod{p}$.} Since $2 > \alpha$, and since
	\[\textstyle \smash{\sum_{i>0} {r - 1 \choose w} {r - 1 - v \choose i(p-1)}} = O(p)\]
whenever $w > (r-1 \bmod {p})$, the same calculations for $\mathfrak m_w(C_1)$ as in the case when $s\not\in\{2,4\}$ and $t\equiv 1\pmod p$ apply, showing that $\mathfrak v_1(C_1) = 1$. Similarly as in the proof of theorem~\ref{t0.0''0}, calling the method \verb|find_m_w(2,alpha,2)| with $\mathtt{alpha}=0,1$ shows that in both cases the relevant $\mathfrak m_w(C_1,\ldots,C_\alpha)$ are polynomials which are divisible by $(r-1)(r-2)(r-3)$, which similarly completes this case of the proof.
\item \emph{The case when $s=4$  and $r \equiv 3 \pmod{p}$.} From the output of the methods 
	\[\verb|find_m_w(2,|\alpha\verb|,2)|, \, \verb|exceptional_cases(2,1,2)|, \, \verb|gcd_for_the_matrix(2,1,2)|,\]
with $\alpha = 0,1,2$, we get that $\smash{\overline \Theta_{r+2,a}}$ has a series whose factors are quotients of submodules of
	\[\textstyle I\left(\sigma_{p-3}(1)\right),\,I\left(\sigma_{p-1}(2)\right),\,\pi(0,0,\omega^2).\]
Note that $r > p+3$ implies that $r \geqslant p^2+3$, so we can follow the proof of theorem~\ref{theorem26} to find an element in the kernel of reduction which gets mapped to $T[1,1]$ by $\Psi_2$, thus obtaining the $\pi(0,0,\omega^2)$ term. 
We consider two cases:
\begin{enumerate}[1.]
\item $r \not\equiv p+3 \pmod{p^2(p-1)}$. The output of the method \verb|find_m_w(2,2,2)| shows that there are polynomials
	\begin{align*}
		f_1 & \textstyle = \sum_{r - 2 > i(p-1) + 2 > 2} \Upsilon_i^{(1)} x^{i(p-1)+2}y^{r-i(p-1)-2},
\\		f_2 & \textstyle = \sum_{r - 2 > i(p-1) + 2 > 2} \Upsilon_i^{(2)} x^{i(p-1)+2}y^{r-i(p-1)-2},
	\end{align*}
with integral coefficients, such that $\smash{\left[1,f_1\right] + O(p^{1+\delta})}$ and $\smash{\left[1,f_2\right] + O(p^{1+\delta})}$ are in the image of $T-a$,
 for %
  some $\delta > 0$, such that $\smash{\sum_{r - 2 > i(p-1) + 2 > 2} \Upsilon_i^{(j)} = 0}$, and such that $\smash{v(S_w^{(j)}) \geqslant 1}$, where 
	\[\textstyle S_w^{(j)} = \sum_{r - 2 > i(p-1) + 2 > 2} \Upsilon_i^{(j)} {i \choose w},\]
for all $w \in \{0,1,2,3\}$ and all $j \in \{1,2\}$. These two polynomials correspond to the last two columns of the output. In particular, $\smash{\Upsilon_i^{(1)} = {r-1 \choose i(p-1)+1}}$ and $\smash{\Upsilon_i^{(2)} = {r \choose i(p-1)+2}}$. Let $\varkappa = \frac{r - (p+3)}{p}$, so that $v(\varkappa)=0$ by assumption. %
 We can calculate, using lemma~\ref{lemma7}, that
\begin{align*}
		\textstyle \begin{Xmatrix} S_0^{(1)} & S_0^{(2)} \\ S_1^{(1)} & S_1^{(2)} \\ S_2^{(1)} & S_2^{(2)} \end{Xmatrix} & \textstyle \equiv_{p^2} \begin{Xmatrix} -\frac{5}{2}\varkappa p & -5\varkappa p \\-\frac 32 \varkappa p & -\frac{5}{2}\varkappa p \\ \frac12 \varkappa p & \varkappa p \end{Xmatrix}.
	\end{align*}
Consequently, from the proof of theorem~\ref{theorem26} we can deduce that there are polynomials
	\begin{align*}
		f_1' & \textstyle = \sum_{r - 2 > i(p-1) + 2 > 2} \Upsilon_i^{(1)'} x^{i(p-1)+2}y^{r-i(p-1)-2}, \\
		f_2' & \textstyle = \sum_{r - 2 > i(p-1) + 2 > 2} \Upsilon_i^{(2)'} x^{i(p-1)+2}y^{r-i(p-1)-2},
	\end{align*}
with integral coefficients, such that
	\begin{align*}
		\left[1,f_1'\right] + \left[1,\smash{\overline{\theta}^3 h_1'}\right] + O(p^{\delta}) & \textstyle \in \Img (T-a),
\\		\left[1,f_2'\right]+ \left[1,\smash{\overline{\theta}^3 h_2'}\right]+ O(p^{\delta}) & \textstyle \in \Img (T-a),
	\end{align*}
for some $h_1',h_2'$ and some $\delta > 0$, and such that
	\begin{align*}
		\textstyle \begin{Xmatrix} S_0^{(1)'} & S_0^{(2)'} \\ S_1^{(1)'} & S_1^{(2)'} \\ S_2^{(1)'} & S_2^{(2)'} \end{Xmatrix} & \textstyle \equiv_{p} \begin{Xmatrix} -5 & -10 \\ -3 & -5 \\ 1 & 2 \end{Xmatrix},
	\end{align*}
where 
	\[\textstyle S_w^{(j)'} = \sum_{r - 2 > i(p-1) + 2 > 2} \Upsilon_i^{(j)'} {i \choose w},\]
for all $w \in \{0,1,2,3\}$ and all $j \in \{1,2\}$. Note that
\begin{align*}
		\textstyle \begin{Xmatrix} 1 \\ -1 \\ -2 \\ 1 \end{Xmatrix} \in \Ker \begin{Xmatrix} 1 & 1 & -5 & -10 \\ 0 & 1 & -3 & -5 \\ -1 & -1 & 1 & 2 \end{Xmatrix},
\end{align*}
where the matrix is seen as a map of a vector space over $\mathbb F_p$, and hence %
	\begin{align*}
	    \textstyle x^2y^{r-2} - x^{r-2}y^2 -2 f_1' + f_2' & \in  \smash{ \ker\Psi_1 \subseteq \overline{\theta}^2 \sigma_{r-2(p+1)}},\\
	    \textstyle x^2y^{r-2} - x^{r-2}y^2 -2 f_1' + f_2' & \not\in  \smash{  \overline{\theta}^3 \sigma_{r-3(p+1)}},
	 \end{align*}
(where $\smash{  \overline{\theta}^3 \sigma_{r-3(p+1)}}$ is the trivial subspace if $r < 3(p+1)$). %
 Consequently, $\smash{\overline \Theta_{r+2,a}}$ has a series whose factors are quotients of submodules of
	\[\textstyle I\left(\sigma_{p-3}(1)\right),\,\pi(0,0,\omega^2).\]
These modules cannot pair up to give a reducible $\smash{\overline \Theta_{r+2,a}}$.
\item $r \equiv p+3 \pmod{p^2(p-1)}$. In particular, since $r > p+3$, we have $r > p^2(p-1)+p$. For $\alpha = 1$, after choosing the constant $C_1 = -\frac{p}{3} + O(p^2)$ in theorem~\ref{theorem26}, where $O(p^2)$ is such that $\mathfrak m_0(C_1) = 0$, we have $v(\mathfrak m_w(C_1)) \geqslant 1$ for $w \in \{1,2,3,4\}$. There is the matrix
		\begin{align*}
		\textstyle \mathfrak M'' & \textstyle = \begin{Xmatrix} \sum_{i>0} {r-1 \choose i(p-1)} & \sum_{i>0} {r \choose i(p-1)+1}  \\ \sum_{i>0} {r-1 \choose i(p-1)} i & \sum_{i>0} {r \choose i(p-1)+1} i \end{Xmatrix}
\\		 & \textstyle \equiv_{p^2} \begin{Xmatrix} \frac p3 & \frac{11p}{6} + \frac{r-4}{p-1} \\ \frac{p(r-1)}{6} & \frac{p(11-2r)}{6} - \frac{r-4}{(p-1)^2} \end{Xmatrix} \textstyle \equiv_{p^2} \begin{Xmatrix} \frac p3 & \frac{11p}{6} + \frac{r-4}{p-1} \\ \frac{p}{3} & \frac{5p}{6} + \frac{4-r}{(p-1)^2} \end{Xmatrix}.
	\end{align*}
Then		
	\begin{align*}
		\textstyle \mathfrak M'' \begin{Xmatrix} 1 \\ C_1 \end{Xmatrix} & \textstyle = \mathfrak M'' \begin{Xmatrix} 1 \\ -\frac{p}{3} \end{Xmatrix} \equiv_{p^2} \begin{Xmatrix} 0 \\ 0 \end{Xmatrix}.
	\end{align*}
Hence $v(\mathfrak m_1(C_1)) \geqslant 2$. Moreover, $v(\mathfrak m_w(C_1)) \geqslant 2$, for all $w \in \{2,3,4\}$, since that statement holds true for $r = p+3$. Consequently, theorem~\ref{theorem26} applies, and we get that $\smash{\overline \Theta_{r+2,a}}$ has a series whose factors are quotients of submodules of
	\[\textstyle \pi(p-3,0,\omega),\,I\left(\sigma_{p-1}(2)\right),\,\pi(0,0,\omega^2).\]
These modules cannot pair up to give a reducible $\smash{\overline \Theta_{r+2,a}}$.
\end{enumerate}
\item \emph{The case when $s=4$  and $r \equiv 4 \pmod{p}$.} In this case, similarly as in the case when $s=4$ and $r \equiv 0 \pmod{p}$, printing \[\verb|A.transpose().kernel().basis_matrix()|\] in the method \verb|exceptional_cases(2,2,2)| in the first program in the appendix shows that the factor in $\smash{\overline \Theta_{r+2,a}}$ corresponding to $\smash{I\left(\sigma_{p-1}(2)\right)}$ is trivial, since $\frac{1}{2}(r-2)(r-3) \ne 0$. Even though the output of the method \[\verb|gcd_for_the_matrix(2,2,2)|\] in the first program in the appendix shows that the corresponding matrix does not have full rank, by printing that matrix we can see that its second row is equal to zero when $r \equiv 4 \pmod{p}$ and that $(0,0,1)^T$ is in its range. Consequently, the factor in $\smash{\overline \Theta_{r+2,a}}$ corresponding to $\smash{I\left(\sigma_{0}(2)\right)}$ must be trivial. The element $h = x^2y^{r-2}-y^2x^{r-2}$ is in the kernel of reduction, and
	\[\textstyle \smash{h = \overline{\theta} (a_1xy^{r-p-2} + \cdots + a_{\lfloor r/(p-1)\rfloor}x^{r-p-2}y)},\]
where 
	\[\smash{\textstyle \smash{a_1 + \cdots + a_{\lfloor r/(p-1)\rfloor} \equiv_p 2 \not\equiv_p 0}}.\]
Consequently, $\smash{h \in \overline{\theta} \sigma_{r-(p+1)}}$ and $\smash{h \not\in \overline{\theta}^2 \sigma_{r-2(p+1)}}$, so there is a nontrivial element of $I_{r-2}(1)$ which is in the kernel of reduction. Since $I_{r-2}(1) \cong I_2(1)$ is not semi-simple, it follows from lemma~\ref{l4.1LL} that the factor of $\smash{\overline \Theta_{r+2,a}}$ which is a quotient of $\smash{I\left(\sigma_{2}(1)\right)}$ must be trivial. These facts together show that $\smash{\overline \Theta_{r+2,a}}$ has a series whose factors are quotients of submodules of
	\[\textstyle I\left(\sigma_{p-5}(4)\right),\, I\left(\sigma_{p-3}(3)\right).\]
These modules cannot pair up to give a reducible $\smash{\overline \Theta_{r+2,a}}$, which completes this case of the proof.
\item \emph{The small leftover cases.} Since theorem~\ref{theorem26} holds true under the assumption that $r > m(p+1)$, we are left with the cases when 
	\[\textstyle (p,r) \in \{(3,6),(3,8),(5,12)\}.\]
The case when $(p,r) = (3,6)$ is trivial. For the case when $(p,r) = (3,8)$, note that
	\[\textstyle 3 xy^7 + 2{7 \choose 2} x^3y^5 + 2{7 \choose 4} x^5y^3 + 2{7 \choose 1} x^7y + E' = 3xy^7 + 42x^3y^5 + 70x^5y^3 + 14 x^7 y + E',\]
is in the image of $T - a$, where $v(E') \geqslant v(a) - 1 > 1$, and therefore $x^5y^3 - x^7y = \smash{\overline{\theta} x^4}$ is in the kernel of reduction. Similarly, $xy^7 - x^3y^5$ is in the kernel of reduction, and we know that $xy^7,x^7y$ are in the kernel of reduction, so $x^3y^5 - x^5y^3 = \smash{\overline{\theta} x^2y^2}$ is in the kernel of reduction as well. In particular, $\smash{\overline \Theta_{10,a}}$ must be $\pi(0,0,1)$. For the case when $(p,r) = (5,12)$, note that
	\[\textstyle 5xy^{11} + 4 {11 \choose 4} x^5y^7 + 4 {11 \choose 8} x^9y^3 + E'' = 5xy^{11} + 1320 x^5y^7 + 660 x^9y^3 + E''\]
is in the image of $T - a$, where $v(E'') \geqslant v(a) - 1 > 1$, and therefore
	\[\textstyle xy^{11} + 264x^5y^7 + 132x^9y^3 + E'''\]
is in the image of $T - a$, where $v(E''') \geqslant v(a) - 2 > 0$. Since we know that $xy^{11}$ is in the kernel of reduction, it follows that
	$-xy^{11} - x^5y^7 + 2x^9y^3 = \smash{\overline{\theta} (4y^6+3x^4y^2)}$ is in the kernel of reduction. Similarly,
\[\textstyle 5 y^{12} + 4 {11 \choose 3} x^4y^{8} + 4{11 \choose 7} x^8y^4 + 4{11 \choose 11}x^{12} + E''''\]
is in the image of $T - a$, where $v(E'''') \geqslant v(a) - 1 > 1$. Moreover, $x^{12} + O(p^2)$ is in the image of $T - a$. Consequently, $y^{12} + 2 x^4y^{8} - x^8y^4$ is in the kernel of reduction. In particular, $\smash{\overline \Theta_{14,a}}$ must be $\pi(2,0,\omega)$.
\end{enumerate}
\end{proof}

\subsection{The case when $\NUMBER > v(a) > 3$}\label{extrasub}

In this subsection, we will complete the proof of theorem~\ref{t0.0''__}, and we will also prove the first and the third part of theorem~\ref{iygutfrd}.

\begin{proof}[Proof of theorem~\ref{t0.0''__}]
The case when $p=3$ is vacuous, so in the remainder of this proof we are going to assume that $p > 3$.
\begin{enumerate}[\tttttt] 
\item \emph{The case when $s \not \in \{2,4,6\}$ and consequently $t \not \equiv\{0,1,2\} \pmod{p}$.}
The proof of this case is similar to the first case in the proof of theorem~\ref{t0.0''}. The program written in \texttt{Sage} listed in the appendix completely automates this proof. More specifically, for each $0 \leqslant \alpha < m$, it finds constants $C_1,\ldots,C_{\alpha}$ such that $v(C_i) \geqslant 1$, for all $1 \leqslant i \leqslant \alpha$, and $\mathfrak m_{w} (C_1,\ldots,C_\alpha) = 0$, for all $0 \leqslant w < \alpha$. Then, it calculates that $\mathfrak m_{\alpha} (C_1,\ldots,C_\alpha) \not\equiv 0 \pmod{p^2}$, which holds true due to the fact that $t \not \equiv\{0,1,2\} \pmod{p}$. Since $\mathfrak m_{w} (C_1,\ldots,C_\alpha)$ is a linear combination of sums of the type $\textstyle \smash{\sum_{j(p-1)>0} {r-\alpha-v \choose j(p-1)}}$, for $0\leqslant v\leqslant w$, and each of these sums has positive valuation since $s - \alpha - w \geqslant s - \alpha - (2m+1-\alpha) > 0$, it follows that $v(\mathfrak m_{w} (C_1,\ldots,C_\alpha)) \geqslant 1$, for all $\alpha < w \leqslant 2m+1-\alpha$. Thus, theorem~\ref{theorem26} applies, and it implies that $\smash{\overline \Theta_{r+2,a}}$ must in fact be isomorphic to $\pi(s-2m,0,\omega^m)$, which is irreducible.
\item \emph{The case when $s\in\{2,4,6\}$.}
Can be done in similar fashion as the case when $s=2$ in the proof of theorem~\ref{t0.0''}, or by calling the method \verb|print_the_roots_for_all_matrices(3,True)| in the first program in the appendix, and the remaining cases being dealt with as in the proof of theorem~\ref{t0.0''}.
\end{enumerate}
\end{proof}

The above proof follows the same outline as the proof of theorem~\ref{t0.0''}, and it can be completely automated. The method \verb|verify_conjecture_eight(m)|, with $\mathtt{m}=x$, combines the two cases in the proof and simply prints whether conjecture~\ref{c1111} is true for $m=x$.

\begin{proof}[Proof of the first part of theorem~\ref{iygutfrd}]
Can be done in similar fashion as the proof of theorem~\ref{t0.0''}, or by calling the method \verb|verify_conjecture_eight(m)| in the first program in the appendix, for each $4\leqslant m<\NUMBER$.
\end{proof}

\begin{proof}[Proof of the third part of theorem~\ref{iygutfrd}]
Follows from theorem~\ref{t0.0''}.
\end{proof}

\subsection{The general case}\label{extrasub'}

In this subsection, we will complete the proof of the second part of theorem~\ref{iygutfrd}.

\begin{proof}[Proof of the second part of theorem~\ref{iygutfrd}]
The proof of this case follows the same outline as the first case in the proof of theorem~\ref{t0.0''}. Let $m$ be the integer such that $m+1>v(a)>m$. In this case, $s \not \in \{2,\ldots,2m\}$, which implies that $p > 2m+1$. Suppose that, for each $\alpha \in \{0,\ldots,m-1\}$, we find terms $C_1,\ldots,C_\alpha$ such that $v(\mathfrak m_{\alpha} (C_1,\ldots,C_\alpha)) = 1$. It can be proven as in the proof of theorem~\ref{t0.0''__} that $v(\mathfrak m_{w} (C_1,\ldots,C_\alpha)) \geqslant 1$, for all $\alpha < w \leqslant 2m+1-\alpha$. Therefore, supposing we find suitable terms $C_1,\ldots,C_\alpha$, theorem~\ref{theorem26} will apply, and it will imply that $\smash{\overline \Theta_{r+2,a}}$ must in fact be isomorphic to $\pi(s-2m,0,\omega^m)$. Consequently, the proof of the second part of theorem~\ref{iygutfrd} will be completed once we show the existence of suitable terms $C_1,\ldots,C_\alpha$. In fact, for each $1 \leqslant j \leqslant \alpha$, we will explicitly define $C_j$ by $C_j = pc_j$, where
	\[\textstyle  c_j = j! {\alpha \choose j} \frac{(-1)^j (s+1-r)}{(s-2\alpha+1+j)\cdots(s-2\alpha+1)} + j! {\alpha+1 \choose j+1} \frac{(-1)^j (r - \alpha)}{(s-2\alpha+1+j)\cdots(s-2\alpha+1)} + O(p).\]
In order to show that $C_1,\ldots,C_\alpha$ satisfy the desired properties, we must show that
	\[\textstyle A (1,C_1,\ldots,C_\alpha)^T = (0,\ldots,0,pu)^T,\]
where $u \in \mathbb Z_p$ is a unit, and $A$ is a $(\alpha+1)\times(\alpha+1)$ matrix defined by
	\begin{align*}
	\textstyle A_{w,0} & \textstyle = p\frac{(-1)^w(s-r)(r-\alpha)_w}{(s-\alpha)_{w+1}} + O(p^2),
	\\ \textstyle A_{w,j} & \textstyle = \sum_v(-1)^{w-v} {j+w-v-1\choose w-v} {r-\alpha+j\choose v} \left( {s-\alpha+j-v \choose j-v} - {r-\alpha+j-v \choose j-v} \right) + O(p),
	\end{align*}
where we index the entries by $w \in \{0,\ldots,\alpha\}$ and $j \in \{1,\ldots,\alpha\}$. Equivalently, we want to show that
	\[\textstyle B (1,c_1,\ldots,c_\alpha)^T = (0,\ldots,0,u)^T,\]
where $u \in \mathbb Z_p$ is a unit, and $B$ is a $(\alpha+1)\times(\alpha+1)$ matrix defined by
	\begin{align*}
	\textstyle B_{w,0} & \textstyle = \frac{(-1)^w(s-r)(r-\alpha)_w}{(s-\alpha)_{w+1}} + O(p),
	\\ \textstyle B_{w,j} & \textstyle = \sum_v(-1)^{w-v} {j+w-v-1\choose w-v} {r-\alpha+j\choose v} \left( {s-\alpha+j-v \choose j-v} - {r-\alpha+j-v \choose j-v} \right) + O(p).
	\end{align*}
We will slightly abuse notation and view $c_1,\ldots,c_\alpha,B,u$ as being reduced modulo~$p$, thus dropping the $O(p)$ terms. We will prove that $\smash{\textstyle B (1,c_1,\ldots,c_\alpha)^T = (0,\ldots,0,u)^T}$, where $\smash{u = \frac{(s-r)_{\alpha+1}}{(s-\alpha)_{\alpha+1}}}$. First, note that
	\begin{align*}
	\textstyle \sum_{w \geqslant 0} X^w \sum_v (-1)^{w-v} {j+w-v-1\choose w-v}  {j \choose v}  & \textstyle = \sum_v (-X)^{v} (1-X)^{-j} {j \choose v} = (1-X)^j (1-X)^{-j} = 1,
	\end{align*}
so $\smash{\sum_v (-1)^{w-v} {j+w-v-1\choose w-v}  {j \choose v} = \delta_{w = 0}}$. Therefore,
	\begin{align*}
	\textstyle \sum_v(-1)^{w-v} {j+w-v-1\choose w-v} {r-\alpha+j\choose v} {r-\alpha+j-v \choose j-v}  & \textstyle = {r-\alpha+j\choose j} \sum_v(-1)^{w-v} {j+w-v-1\choose w-v}  {j \choose v} 
	\\  & \textstyle = \delta_{w = 0} {r-\alpha+j\choose j}.
	\end{align*}
Consequently, the matrix $B$ is given by
	\begin{align*}
	\textstyle B_{w,0} & \textstyle = \frac{(-1)^w(s-r)(r-\alpha)_w}{(s-\alpha)_{w+1}} + O(p),
	\\ \textstyle B_{w,j} & \textstyle = - \delta_{w = 0} {r-\alpha+j \choose j} + \sum_v(-1)^{w-v} {j+w-v-1\choose w-v} {r-\alpha+j\choose v} {s-\alpha+j-v \choose j-v} + O(p).
	\end{align*}
We are going to consider two cases, when $w = 0$ and when $w > 0$.
\begin{enumerate}[\tttttt] 
\item \emph{The case when $w = 0$.}
In this case, we want to show that
	\begin{align*}
	\textstyle \sum_{j=1}^\alpha c_j \left({s - \alpha + j \choose j} - {r - \alpha + j \choose j}\right) = \delta_{\alpha \ne 0} \frac{r-s}{s-\alpha},
	\end{align*}
reduced modulo~$p$. If $\alpha = 0$, then the left side is the empty sum, so the identity is vacuously true. Suppose that $\alpha > 0$. We are going to view the terms in the identity as rational functions in the formal variables $r$ and $s$, over the field $\mathbb Q$. Then we can replace $r$ and $s$ by $r+\alpha$ and $s+\alpha$, which leaves us with the task of showing the identity
	\begin{align*}
	\textstyle \Xi(r,s) \defeq \sum_{j} \frac{(-1)^j j!}{j+1} {\alpha \choose j} \frac{ (j+1)(s+1)+(\alpha-j)r}{(s-\alpha+1+j)\cdots(s-\alpha+1)} \left({s + j \choose j} - {r + j \choose j}\right) = \frac{r-s}{s}.
	\end{align*}
Note that
	\begin{align*}
	\textstyle \Xi(r+1,s)-\Xi(r,s) & \textstyle = \sum_{j} \frac{(-1)^j j!}{j+1} {\alpha \choose j} \frac{\alpha-j}{(s-\alpha+1+j)\cdots(s-\alpha+1)} {s + j \choose j}
		\\ & \textstyle \qquad - \sum_{j} (-1)^j j! {\alpha \choose j} \frac{s+1-\alpha+j}{(s-\alpha+1+j)\cdots(s-\alpha+1)} {r + j \choose j-1}
		\\ & \textstyle \qquad - \sum_{j} (-1)^j j! {\alpha \choose j} \frac{\alpha-j}{(s-\alpha+1+j)\cdots(s-\alpha+1)} {r + j + 1 \choose j}
	\\ & \textstyle = \sum_{j} \frac{(-1)^j j!}{j+1} {\alpha \choose j} \frac{\alpha-j}{(s-\alpha+1+j)\cdots(s-\alpha+1)} {s + j \choose j}
		\\ & \textstyle \qquad - \sum_{j} (-1)^j j! {\alpha \choose j} \frac{1}{(s-\alpha+j)\cdots(s-\alpha+1)} {r + j \choose j-1}
		\\ & \textstyle \qquad - \sum_{j} (-1)^{j-1} (j-1)! {\alpha \choose j-1} \frac{\alpha-j+1}{(s-\alpha+j)\cdots(s-\alpha+1)} {r + j \choose j-1}
	\\ & \textstyle = \sum_{j} \frac{(-1)^j j!}{j+1} {\alpha \choose j} \frac{\alpha-j}{(s-\alpha+1+j)\cdots(s-\alpha+1)} {s + j \choose j}
		\\ & \textstyle \qquad + \sum_{j}   \frac{(-1)^j (j-1)!}{(s-\alpha+j)\cdots(s-\alpha+1)} {r + j \choose j-1} \left( (\alpha-j+1) {\alpha \choose j-1} - j{\alpha \choose j}\right)
		\\ & \textstyle = \sum_{j\geqslant 0} {\alpha \choose j+1}  {s + j \choose j} \frac{(-1)^j j!}{(s-\alpha+1+j)\cdots(s-\alpha+1)}= g(s),
	\end{align*}
since $\smash{(\alpha-j+1) {\alpha \choose j-1} = j{\alpha \choose j}}$. In particular, since $\Xi(r,s)$ is a polynomial in $r$, it follows that $\Xi(r,s)$ must be linear in $r$. By combining this fact with the fact that $\Xi(s,s) = 0$, we get that $\Xi(r,s) = (r-s) g(s)$. This leaves us with the task of showing the identity 
	\begin{align*}
	\textstyle  sg(s) & \textstyle =  \sum_{j \geqslant 0}(-1)^j {\alpha \choose j+1} \frac{ (s+j)_{j+1}}{(s-\alpha+1+j)_{j+1}} = 1,
	\end{align*}
which is equivalent to
	\begin{align*}
	\textstyle h(\alpha,s) = 1 - sg(s) & \textstyle = - \sum_{j} (-1)^j {\alpha \choose j+1} \frac{ (s+j)_{j+1}}{(s-\alpha+1+j)_{j+1}} = \sum_{j} (-1)^j {\alpha \choose j} \frac{ (s-1+j)_{j}}{(s-\alpha+j)_{j}} = 0.
	\end{align*}
Note that $\smash{\textstyle h(\alpha,s) = \null_2F_1(-a,s; -a+s+1; 1)}$ can be written explicitly in closed form by using Euler's formula
	\begin{align*}
	\textstyle \null_2F_1(a,b; c;z) = \frac{\Gamma(c)}{\Gamma(b)\Gamma(c-b)} \int_0^1 \frac{t^{b-1} (1-t)^{c-b-1}}{(1-tz)^a} \,\mathrm dt,
	\end{align*}
which is given in~\cite{b13}. Specifically, $\smash{h(\alpha,s) = \null_2F_1(-a,s; -a+s+1; 1) = \frac{\Gamma(-a+s+1)}{\Gamma(s+1)\Gamma(1-a)} = 0}$, as $\Gamma(1-a) = \infty$, and $\Gamma(-a+s+1)$ and $\Gamma(s+1)$ are finite.
\item \emph{The case when $w \ne 0$.}
In this case, we want to show that
	\begin{align*}
	\textstyle \sum_{j=1}^\alpha c_j \sum_v(-1)^{w-v} {j+w-v-1\choose w-v} {r-\alpha+j\choose v} {s-\alpha+j-v \choose j-v} = \frac{(-1)^w(r-s)(r-\alpha)_w}{(s-\alpha)_{w+1}} + \delta_{w = \alpha} \frac{(s-r)_{\alpha+1}}{(s-\alpha)_{\alpha+1}},
	\end{align*}
for all $0 < w \leqslant \alpha$. 
 Suppose that $\alpha > 0$. Again, we are going to view the terms in the identity as rational functions in the formal variables $r$ and $s$, over the field $\mathbb Q$. Then we can replace $r$ and $s$ by $r+\alpha$ and $s+\alpha$, which leaves us with the task of showing the identity
	\begin{align*}
	\textstyle \Xi_w(r,s) \defeq \sum_{j=1}^\alpha d_j \sum_v(-1)^{w-v} {j+w-v-1\choose w-v} {r+j\choose v} {s+j-v \choose j-v} = \frac{(-1)^w(r-s)(r)_w}{(s)_{w+1}} + \delta_{w = \alpha} \frac{(s-r)_{\alpha+1}}{(s)_{\alpha+1}},
	\end{align*}
where
	\begin{align*}
	\textstyle d_j = \frac{(-1)^j j!}{j+1} {\alpha \choose j} \frac{(j+1)(s+1)+(\alpha-j)r}{(s-\alpha+1+j)\cdots(s-\alpha+1)}.
	\end{align*}
This identity can be verified by a computer for all $\alpha \leqslant \NUMBERRRRRR$, which completes the proof of the second part of theorem~\ref{iygutfrd}.
\end{enumerate}
\end{proof}

\SECTION \PART{The boundary of weight space}

\appendix

\section{Appendix}

\subsection{Calculations with \texttt{Sage}}

\subsubsection{Verifying conjecture~\ref{c1111}}

The following program can be used to find the determinants of the relevant square submatrices of $\smash{\mathfrak M^{(r,m,\alpha)}}$, find their greatest common divisor as a polynomial in $\smash{r}$, and list a basis for the kernel of the matrix $\smash{\mathfrak M^{(r,m,\alpha)}}$ seen as a linear operator over $\smash{\mathbb F_p}$. In particular, calling the method \verb|print_the_roots_for_all_matrices(m,True)| with $\smash{\mathtt{m}=x}$ lists the possible congruence classes modulo $\smash{p}$ outside of which $\smash{\overline \Theta_{r+2,a}}$ is irreducible, for $\smash{m=x}$ and for each possible $\smash{s \in \{2,\ldots,2x\}}$. The method \verb|verify_conjecture_eight(m)|, with $\smash{\mathtt{m}=x}$, combines this with theorem~\ref{theorem26}, as explained in subsection~\ref{extrasub}, and simply prints whether the first two parts of conjecture~\ref{c1111} are true for $\smash{m=x}$.

\newcommand{\REF}{}

\

\textcolor{OrangeRed}{\texttt{conjecture\_eight.sage}}\\[-8pt]
{\footnotesize
\begin{Verbatim}[commandchars=\\\{\},numbers=left,firstnumber=1,stepnumber=1,codes={\catcode`\$=3\catcode`\^=7\catcode`\_=8}]
\PY{k+kn}{from} \PY{n+nn}{operator} \PY{k+kn}{import} \PY{n}{mul}
\PY{p}{[}\PY{n}{r}\PY{p}{,}\PY{n}{s}\PY{p}{,}\PY{n}{t}\PY{p}{,}\PY{n}{p}\PY{p}{,}\PY{n}{v}\PY{p}{]} \PY{o}{=} \PY{n}{var}\PY{p}{(}\PY{l+s}{\PYZsq{}}\PY{l+s}{r s t p v}\PY{l+s}{\PYZsq{}}\PY{p}{)}
\PY{n}{Ring} \PY{o}{=} \PY{n}{FractionField}\PY{p}{(}\PY{n}{PolynomialRing}\PY{p}{(}\PY{n}{QQ}\PY{p}{,}\PY{p}{[}\PY{l+s}{\PYZsq{}}\PY{l+s}{r}\PY{l+s}{\PYZsq{}}\PY{p}{,}\PY{l+s}{\PYZsq{}}\PY{l+s}{s}\PY{l+s}{\PYZsq{}}\PY{p}{,}\PY{l+s}{\PYZsq{}}\PY{l+s}{t}\PY{l+s}{\PYZsq{}}\PY{p}{,}\PY{l+s}{\PYZsq{}}\PY{l+s}{p}\PY{l+s}{\PYZsq{}}\PY{p}{]}\PY{p}{)}\PY{p}{)}

\PY{k}{def} \PY{n+nf}{factorise}\PY{p}{(}\PY{n}{g}\PY{p}{)}\PY{p}{:}
    \PY{k}{return} \PY{n}{g}\PY{o}{.}\PY{n}{factor}\PY{p}{(}\PY{p}{)} \PY{k}{if} \PY{n}{g}\PY{o}{!=}\PY{l+m+mi}{0} \PY{k}{else} \PY{n}{g}

\PY{c}{\PYZsh{} Returns the roots of $g$ as a sorted list.}
\PY{k}{def} \PY{n+nf}{get\PYZus{}roots}\PY{p}{(}\PY{n}{g}\PY{p}{)}\PY{p}{:}
    \PY{k}{if} \PY{n}{g}\PY{o}{!=}\PY{l+m+mi}{0}\PY{p}{:}
        \PY{n}{list\PYZus{}of\PYZus{}roots} \PY{o}{=} \PY{n}{g}\PY{o}{.}\PY{n}{roots}\PY{p}{(}\PY{p}{)}
        \PY{k}{for} \PY{n}{i} \PY{o+ow}{in} \PY{n+nb}{range}\PY{p}{(}\PY{l+m+mi}{0}\PY{p}{,}\PY{n+nb}{len}\PY{p}{(}\PY{n}{list\PYZus{}of\PYZus{}roots}\PY{p}{)}\PY{p}{)}\PY{p}{:}
            \PY{n}{list\PYZus{}of\PYZus{}roots}\PY{p}{[}\PY{n}{i}\PY{p}{]} \PY{o}{=} \PY{n}{list\PYZus{}of\PYZus{}roots}\PY{p}{[}\PY{n}{i}\PY{p}{]}\PY{p}{[}\PY{l+m+mi}{0}\PY{p}{]}
        \PY{n}{list\PYZus{}of\PYZus{}roots}\PY{o}{.}\PY{n}{sort}\PY{p}{(}\PY{p}{)}
        \PY{k}{return} \PY{n}{list\PYZus{}of\PYZus{}roots}
    \PY{k}{return} \PY{p}{[}\PY{p}{]}

\PY{c}{\PYZsh{} Returns $\eta(X,Y)$, from lemma \ref{lemma7wefwfwerfwf}.}
\PY{k}{def} \PY{n+nf}{eta}\PY{p}{(}\PY{n}{X}\PY{p}{,}\PY{n}{Y}\PY{p}{)}\PY{p}{:}
    \PY{k}{if} \PY{n}{X}\PY{o}{\PYZgt{}}\PY{o}{=}\PY{l+m+mi}{1} \PY{o+ow}{and} \PY{n}{Y}\PY{o}{\PYZgt{}}\PY{o}{=}\PY{l+m+mi}{0}\PY{p}{:}
        \PY{k}{return} \PY{n}{binomial}\PY{p}{(}\PY{n}{X}\PY{p}{,}\PY{n}{Y}\PY{p}{)}
    \PY{k}{if} \PY{n}{X}\PY{o}{\PYZlt{}}\PY{l+m+mi}{1} \PY{o+ow}{and} \PY{n}{Y}\PY{o}{\PYZgt{}}\PY{o}{=}\PY{l+m+mi}{0}\PY{p}{:}
        \PY{k}{return} \PY{l+m+mi}{2}\PY{o}{*}\PY{n}{binomial}\PY{p}{(}\PY{n}{X}\PY{o}{\PYZhy{}}\PY{l+m+mi}{1}\PY{p}{,}\PY{n}{Y}\PY{p}{)} \PY{k}{if} \PY{n}{X}\PY{o}{==}\PY{l+m+mi}{0} \PY{o+ow}{and} \PY{n}{Y}\PY{o}{==}\PY{l+m+mi}{0} \PY{k}{else} \PY{n}{binomial}\PY{p}{(}\PY{n}{X}\PY{o}{\PYZhy{}}\PY{l+m+mi}{1}\PY{p}{,}\PY{n}{Y}\PY{p}{)}
    \PY{k}{if} \PY{n}{X}\PY{o}{\PYZlt{}}\PY{l+m+mi}{1} \PY{o+ow}{and} \PY{n}{Y}\PY{o}{\PYZlt{}}\PY{l+m+mi}{0}\PY{p}{:}
        \PY{k}{return} \PY{n}{binomial}\PY{p}{(}\PY{n}{X}\PY{o}{\PYZhy{}}\PY{l+m+mi}{1}\PY{p}{,}\PY{n}{X}\PY{o}{\PYZhy{}}\PY{n}{Y}\PY{p}{)} \PY{k}{if} \PY{n}{X}\PY{o}{\PYZgt{}}\PY{o}{=}\PY{n}{Y} \PY{k}{else} \PY{l+m+mi}{0}
    \PY{k}{return} \PY{l+m+mi}{0}

\PY{c}{\PYZsh{} Returns the matrix $A$ with $r$ substituted by $e$.}
\PY{k}{def} \PY{n+nf}{substitute\PYZus{}in\PYZus{}matrix}\PY{p}{(}\PY{n}{A}\PY{p}{,}\PY{n}{e}\PY{p}{)}\PY{p}{:}
    \PY{n}{B} \PY{o}{=} \PY{n}{copy}\PY{p}{(}\PY{n}{A}\PY{p}{)}
    \PY{k}{for} \PY{n}{i} \PY{o+ow}{in} \PY{n+nb}{range}\PY{p}{(}\PY{l+m+mi}{0}\PY{p}{,}\PY{n}{B}\PY{o}{.}\PY{n}{nrows}\PY{p}{(}\PY{p}{)}\PY{p}{)}\PY{p}{:}
        \PY{k}{for} \PY{n}{j} \PY{o+ow}{in} \PY{n+nb}{range}\PY{p}{(}\PY{l+m+mi}{0}\PY{p}{,}\PY{n}{B}\PY{o}{.}\PY{n}{ncols}\PY{p}{(}\PY{p}{)}\PY{p}{)}\PY{p}{:}
            \PY{n}{B}\PY{p}{[}\PY{n}{i}\PY{p}{,}\PY{n}{j}\PY{p}{]}\PY{o}{=} \PY{n}{A}\PY{p}{[}\PY{n}{i}\PY{p}{,}\PY{n}{j}\PY{p}{]}\PY{o}{.}\PY{n}{subs}\PY{p}{(}\PY{n}{r}\PY{o}{=}\PY{n}{e}\PY{p}{)}
    \PY{k}{return} \PY{n}{B}

\PY{c}{\PYZsh{} Returns the matrix $\mathfrak{M}^{(r,m,\alpha)}$, from lemma \ref{l7Z}.}
\PY{k}{def} \PY{n+nf}{construct\PYZus{}matrix}\PY{p}{(}\PY{n}{m}\PY{p}{,}\PY{n}{alpha}\PY{p}{,}\PY{n}{L}\PY{p}{)}\PY{p}{:}
    \PY{c}{\PYZsh{} $M$ is the number of rows, $N+2$ is the number of columns.}
    \PY{n}{M} \PY{o}{=} \PY{n}{alpha}\PY{o}{+}\PY{l+m+mi}{1}\PY{p}{;} \PY{n}{N} \PY{o}{=} \PY{n}{m}\PY{o}{+}\PY{l+m+mi}{1}\PY{p}{;} \PY{n}{A} \PY{o}{=} \PY{n}{matrix}\PY{p}{(}\PY{n}{Ring}\PY{p}{,}\PY{n}{M}\PY{p}{,}\PY{n}{N}\PY{o}{+}\PY{l+m+mi}{2}\PY{p}{)}

    \PY{c}{\PYZsh{} Filling in the first two columns.}
    \PY{n}{A}\PY{p}{[}\PY{l+m+mi}{0}\PY{p}{,}\PY{l+m+mi}{0}\PY{p}{]} \PY{o}{=} \PY{l+m+mi}{1}\PY{p}{;} \PY{n}{A}\PY{p}{[}\PY{n}{M}\PY{o}{\PYZhy{}}\PY{l+m+mi}{1}\PY{p}{,}\PY{l+m+mi}{0}\PY{p}{]} \PY{o}{=} \PY{n}{A}\PY{p}{[}\PY{n}{M}\PY{o}{\PYZhy{}}\PY{l+m+mi}{1}\PY{p}{,}\PY{l+m+mi}{0}\PY{p}{]}\PY{o}{\PYZhy{}}\PY{p}{(}\PY{o}{\PYZhy{}}\PY{l+m+mi}{1}\PY{p}{)}\PY{o}{\PYZca{}}\PY{p}{(}\PY{n}{alpha}\PY{p}{)}
    \PY{k}{if} \PY{n}{L}\PY{o}{\PYZgt{}}\PY{o}{=}\PY{n}{alpha} \PY{o+ow}{and} \PY{n}{m}\PY{o}{+}\PY{n}{alpha}\PY{o}{\PYZgt{}}\PY{o}{=}\PY{l+m+mi}{2}\PY{o}{*}\PY{n}{L}\PY{p}{:}
        \PY{k}{for} \PY{n}{w} \PY{o+ow}{in} \PY{n+nb}{range}\PY{p}{(}\PY{l+m+mi}{0}\PY{p}{,}\PY{n}{M}\PY{p}{)}\PY{p}{:}
            \PY{n}{A}\PY{p}{[}\PY{n}{w}\PY{p}{,}\PY{l+m+mi}{1}\PY{p}{]} \PY{o}{=} \PY{n}{binomial}\PY{p}{(}\PY{l+m+mi}{2}\PY{o}{*}\PY{n}{L}\PY{o}{\PYZhy{}}\PY{n}{r}\PY{p}{,}\PY{n}{w}\PY{p}{)}
        \PY{n}{A}\PY{p}{[}\PY{n}{M}\PY{o}{\PYZhy{}}\PY{l+m+mi}{1}\PY{p}{,}\PY{l+m+mi}{1}\PY{p}{]} \PY{o}{=} \PY{n}{A}\PY{p}{[}\PY{n}{M}\PY{o}{\PYZhy{}}\PY{l+m+mi}{1}\PY{p}{,}\PY{l+m+mi}{1}\PY{p}{]}\PY{o}{\PYZhy{}}\PY{l+m+mi}{1}

    \PY{c}{\PYZsh{} Filling in the last $N$ columns.}
    \PY{k}{for} \PY{n}{l} \PY{o+ow}{in} \PY{n+nb}{range}\PY{p}{(}\PY{n}{alpha}\PY{o}{\PYZhy{}}\PY{n}{m}\PY{p}{,}\PY{n}{M}\PY{p}{)}\PY{p}{:}
        \PY{k}{for} \PY{n}{w} \PY{o+ow}{in} \PY{n+nb}{range}\PY{p}{(}\PY{l+m+mi}{0}\PY{p}{,}\PY{n}{M}\PY{p}{)}\PY{p}{:}
            \PY{n}{a} \PY{o}{=} \PY{o}{\PYZhy{}}\PY{n}{binomial}\PY{p}{(}\PY{n}{r}\PY{o}{\PYZhy{}}\PY{n}{alpha}\PY{o}{+}\PY{n}{l}\PY{p}{,}\PY{l+m+mi}{2}\PY{o}{*}\PY{n}{L}\PY{o}{\PYZhy{}}\PY{n}{alpha}\PY{p}{)}\PY{o}{*}\PY{n}{binomial}\PY{p}{(}\PY{l+m+mi}{2}\PY{o}{*}\PY{n}{L}\PY{o}{\PYZhy{}}\PY{n}{r}\PY{p}{,}\PY{n}{w}\PY{p}{)} \PY{k}{if} \PY{n}{m}\PY{o}{+}\PY{n}{alpha}\PY{o}{\PYZgt{}}\PY{o}{=}\PY{l+m+mi}{2}\PY{o}{*}\PY{n}{L} \PY{k}{else} \PY{l+m+mi}{0}
            \PY{n}{b} \PY{o}{=} \PY{o}{\PYZhy{}}\PY{n}{binomial}\PY{p}{(}\PY{n}{r}\PY{o}{\PYZhy{}}\PY{n}{alpha}\PY{o}{+}\PY{n}{l}\PY{p}{,}\PY{n}{l}\PY{p}{)}\PY{o}{*}\PY{n}{binomial}\PY{p}{(}\PY{l+m+mi}{0}\PY{p}{,}\PY{n}{w}\PY{p}{)} \PYZbs{}
                    \PY{o}{+}\PY{n+nb}{sum}\PY{p}{(}\PY{p}{[}\PY{p}{(}\PY{o}{\PYZhy{}}\PY{l+m+mi}{1}\PY{p}{)}\PY{o}{\PYZca{}}\PY{p}{(}\PY{n}{v}\PY{p}{)}\PY{o}{*}\PY{n}{binomial}\PY{p}{(}\PY{n}{l}\PY{o}{\PYZhy{}}\PY{n}{v}\PY{p}{,}\PY{n}{w}\PY{o}{\PYZhy{}}\PY{n}{v}\PY{p}{)}\PY{o}{*}\PY{n}{eta}\PY{p}{(}\PY{l+m+mi}{2}\PY{o}{*}\PY{n}{L}\PY{o}{\PYZhy{}}\PY{n}{alpha}\PY{o}{+}\PY{n}{l}\PY{o}{\PYZhy{}}\PY{n}{v}\PY{p}{,}\PY{n}{l}\PY{o}{\PYZhy{}}\PY{n}{v}\PY{p}{)} \PYZbs{}
                        \PY{o}{*}\PY{n}{binomial}\PY{p}{(}\PY{n}{r}\PY{o}{\PYZhy{}}\PY{n}{alpha}\PY{o}{+}\PY{n}{l}\PY{p}{,}\PY{n}{v}\PY{p}{)} \PY{k}{for} \PY{n}{v} \PY{o+ow}{in} \PY{n+nb}{range}\PY{p}{(}\PY{l+m+mi}{0}\PY{p}{,}\PY{n}{w}\PY{o}{+}\PY{l+m+mi}{1}\PY{p}{)}\PY{p}{]}\PY{p}{)}
            \PY{n}{A}\PY{p}{[}\PY{n}{w}\PY{p}{,}\PY{n}{l}\PY{o}{+}\PY{l+m+mi}{2}\PY{o}{\PYZhy{}}\PY{n}{alpha}\PY{o}{+}\PY{n}{m}\PY{p}{]} \PY{o}{=} \PY{n}{factorise}\PY{p}{(}\PY{n}{simplify}\PY{p}{(}\PY{n}{a}\PY{o}{+}\PY{n}{b}\PY{p}{)}\PY{p}{)}
    \PY{k}{return} \PY{n}{A}

\PY{c}{\PYZsh{} Returns the relevant matrix arising from theorem \ref{theorem26}, when $s>2m$.}
\PY{k}{def} \PY{n+nf}{construct\PYZus{}big\PYZus{}matrix}\PY{p}{(}\PY{n}{m}\PY{p}{,}\PY{n}{alpha}\PY{p}{)}\PY{p}{:}
    \PY{c}{\PYZsh{} $M$ is the number of rows and the number of columns.}
    \PY{n}{M} \PY{o}{=} \PY{n}{alpha}\PY{o}{+}\PY{l+m+mi}{1}\PY{p}{;} \PY{n}{A} \PY{o}{=} \PY{n}{matrix}\PY{p}{(}\PY{n}{Ring}\PY{p}{,}\PY{n}{M}\PY{p}{,}\PY{n}{M}\PY{p}{)}

    \PY{k}{for} \PY{n}{l} \PY{o+ow}{in} \PY{n+nb}{range}\PY{p}{(}\PY{l+m+mi}{0}\PY{p}{,}\PY{n}{M}\PY{p}{)}\PY{p}{:}
        \PY{k}{for} \PY{n}{w} \PY{o+ow}{in} \PY{n+nb}{range}\PY{p}{(}\PY{l+m+mi}{0}\PY{p}{,}\PY{n}{M}\PY{p}{)}\PY{p}{:}
            \PY{n}{a} \PY{o}{=} \PY{p}{(}\PY{n}{s}\PY{o}{\PYZhy{}}\PY{n}{r}\PY{p}{)}\PY{o}{*}\PY{p}{(}\PY{n}{p}\PY{o}{/}\PY{p}{(}\PY{n}{s}\PY{o}{\PYZhy{}}\PY{n}{alpha}\PY{p}{)}\PY{p}{)}\PY{o}{*}\PY{p}{(}\PY{n}{binomial}\PY{p}{(}\PY{n}{r}\PY{o}{\PYZhy{}}\PY{n}{alpha}\PY{p}{,}\PY{n}{w}\PY{p}{)}\PY{o}{/}\PY{n}{binomial}\PY{p}{(}\PY{n}{s}\PY{o}{\PYZhy{}}\PY{n}{alpha}\PY{o}{\PYZhy{}}\PY{l+m+mi}{1}\PY{p}{,}\PY{n}{w}\PY{p}{)}\PY{p}{)} \PY{k}{if} \PY{n}{l}\PY{o}{==}\PY{l+m+mi}{0} \PY{k}{else} \PY{l+m+mi}{0}
            \PY{n}{b} \PY{o}{=} \PY{o}{\PYZhy{}}\PY{n}{binomial}\PY{p}{(}\PY{n}{r}\PY{o}{\PYZhy{}}\PY{n}{alpha}\PY{o}{+}\PY{n}{l}\PY{p}{,}\PY{n}{l}\PY{p}{)}\PY{o}{*}\PY{n}{binomial}\PY{p}{(}\PY{l+m+mi}{0}\PY{p}{,}\PY{n}{w}\PY{p}{)} \PYZbs{}
                    \PY{o}{+}\PY{n+nb}{sum}\PY{p}{(}\PY{p}{[}\PY{p}{(}\PY{o}{\PYZhy{}}\PY{l+m+mi}{1}\PY{p}{)}\PY{o}{\PYZca{}}\PY{p}{(}\PY{n}{w}\PY{o}{\PYZhy{}}\PY{n}{v}\PY{p}{)}\PY{o}{*}\PY{n}{binomial}\PY{p}{(}\PY{n}{l}\PY{o}{+}\PY{n}{w}\PY{o}{\PYZhy{}}\PY{n}{v}\PY{o}{\PYZhy{}}\PY{l+m+mi}{1}\PY{p}{,}\PY{n}{w}\PY{o}{\PYZhy{}}\PY{n}{v}\PY{p}{)}\PY{o}{*}\PY{n}{binomial}\PY{p}{(}\PY{n}{s}\PY{o}{\PYZhy{}}\PY{n}{alpha}\PY{o}{+}\PY{n}{l}\PY{o}{\PYZhy{}}\PY{n}{v}\PY{p}{,}\PY{n}{l}\PY{o}{\PYZhy{}}\PY{n}{v}\PY{p}{)} \PYZbs{}
                        \PY{o}{*}\PY{n}{binomial}\PY{p}{(}\PY{n}{r}\PY{o}{\PYZhy{}}\PY{n}{alpha}\PY{o}{+}\PY{n}{l}\PY{p}{,}\PY{n}{v}\PY{p}{)} \PY{k}{for} \PY{n}{v} \PY{o+ow}{in} \PY{n+nb}{range}\PY{p}{(}\PY{l+m+mi}{0}\PY{p}{,}\PY{n}{w}\PY{o}{+}\PY{l+m+mi}{1}\PY{p}{)}\PY{p}{]}\PY{p}{)}
            \PY{n}{A}\PY{p}{[}\PY{n}{w}\PY{p}{,}\PY{n}{l}\PY{p}{]} \PY{o}{=} \PY{n}{factorise}\PY{p}{(}\PY{n}{simplify}\PY{p}{(}\PY{n}{a}\PY{o}{*}\PY{p}{(}\PY{o}{\PYZhy{}}\PY{l+m+mi}{1}\PY{p}{)}\PY{o}{\PYZca{}}\PY{n}{w}\PY{o}{+}\PY{n}{b}\PY{p}{)}\PY{p}{)}
    \PY{k}{return} \PY{n}{A}

\PY{c}{\PYZsh{} Returns a polynomial whose roots are exceptional congruence}
\PY{c}{\PYZsh{} classes of $r$ modulo $p$ outside of which the relevant factors of}
\PY{c}{\PYZsh{} $\smash{\overline{\theta}^\alpha\sigma_{r-\alpha(p+1)}/\overline{\theta}^{\alpha+1}\sigma_{r-(\alpha+1)(p+1)}}$ cancel, as can be shown by lemma \ref{lemma5}.}
\PY{k}{def} \PY{n+nf}{exceptional\PYZus{}cases}\PY{p}{(}\PY{n}{m}\PY{p}{,}\PY{n}{alpha}\PY{p}{,}\PY{n}{L}\PY{p}{)}\PY{p}{:}
    \PY{n}{A} \PY{o}{=} \PY{n}{construct\PYZus{}matrix}\PY{p}{(}\PY{n}{m}\PY{p}{,}\PY{n}{alpha}\PY{p}{,}\PY{n}{L}\PY{p}{)}\PY{p}{;} \PY{n}{product} \PY{o}{=} \PY{l+m+mi}{1}
    \PY{n}{product} \PY{o}{*}\PY{o}{=} \PY{n}{A}\PY{p}{[}\PY{l+m+mi}{0}\PY{p}{,}\PY{l+m+mi}{0}\PY{p}{]}\PY{o}{*}\PY{n}{A}\PY{o}{.}\PY{n}{transpose}\PY{p}{(}\PY{p}{)}\PY{o}{.}\PY{n}{kernel}\PY{p}{(}\PY{p}{)}\PY{o}{.}\PY{n}{basis\PYZus{}matrix}\PY{p}{(}\PY{p}{)}\PY{p}{[}\PY{l+m+mi}{0}\PY{p}{,}\PY{l+m+mi}{0}\PY{p}{]}
    \PY{k}{if} \PY{n}{A}\PY{o}{.}\PY{n}{transpose}\PY{p}{(}\PY{p}{)}\PY{o}{.}\PY{n}{kernel}\PY{p}{(}\PY{p}{)}\PY{o}{.}\PY{n}{basis\PYZus{}matrix}\PY{p}{(}\PY{p}{)}\PY{p}{[}\PY{l+m+mi}{0}\PY{p}{,}\PY{n}{m}\PY{o}{+}\PY{l+m+mi}{2}\PY{p}{]} \PY{o}{!=} \PY{l+m+mi}{0}\PY{p}{:}
        \PY{n}{product} \PY{o}{/}\PY{o}{=} \PY{n}{A}\PY{o}{.}\PY{n}{transpose}\PY{p}{(}\PY{p}{)}\PY{o}{.}\PY{n}{kernel}\PY{p}{(}\PY{p}{)}\PY{o}{.}\PY{n}{basis\PYZus{}matrix}\PY{p}{(}\PY{p}{)}\PY{p}{[}\PY{l+m+mi}{0}\PY{p}{,}\PY{n}{m}\PY{o}{+}\PY{l+m+mi}{2}\PY{p}{]}
    \PY{k}{if} \PY{l+m+mi}{2}\PY{o}{*}\PY{n}{L}\PY{o}{\PYZhy{}}\PY{l+m+mi}{1}\PY{o}{\PYZgt{}}\PY{o}{=}\PY{n}{alpha} \PY{o+ow}{and} \PY{n}{alpha}\PY{o}{\PYZgt{}}\PY{o}{=}\PY{n}{L}\PY{p}{:}
        \PY{n}{product} \PY{o}{*}\PY{o}{=} \PY{n}{gcd\PYZus{}for\PYZus{}the\PYZus{}matrix}\PY{p}{(}\PY{n}{m}\PY{p}{,}\PY{n}{alpha}\PY{p}{,}\PY{n}{L}\PY{p}{)}
    \PY{k}{return} \PY{n}{product}

\PY{c}{\PYZsh{} Returns the greatest common divisor of the determinants of the relevant}
\PY{c}{\PYZsh{} square submatrices of $\mathfrak{M}^{(r,m,\alpha)}$, seen as polynomials in $r$ over $\mathbb{Q}$.}
\PY{k}{def} \PY{n+nf}{gcd\PYZus{}for\PYZus{}the\PYZus{}matrix}\PY{p}{(}\PY{n}{m}\PY{p}{,}\PY{n}{alpha}\PY{p}{,}\PY{n}{L}\PY{p}{)}\PY{p}{:}
    \PY{c}{\PYZsh{} $M$ is the number of rows, $N+2$ is the number of columns.}
    \PY{n}{M} \PY{o}{=} \PY{n}{alpha}\PY{o}{+}\PY{l+m+mi}{1}\PY{p}{;} \PY{n}{N} \PY{o}{=} \PY{n}{m}\PY{o}{+}\PY{l+m+mi}{1}\PY{p}{;} \PY{n}{A} \PY{o}{=} \PY{n}{construct\PYZus{}matrix}\PY{p}{(}\PY{n}{m}\PY{p}{,}\PY{n}{alpha}\PY{p}{,}\PY{n}{L}\PY{p}{)}
    \PY{n}{B} \PY{o}{=} \PY{n}{A}\PY{o}{.}\PY{n}{matrix\PYZus{}from\PYZus{}rows\PYZus{}and\PYZus{}columns}\PY{p}{(}\PY{n+nb}{range}\PY{p}{(}\PY{l+m+mi}{0}\PY{p}{,}\PY{n}{M}\PY{p}{)}\PY{p}{,} \PY{n+nb}{range}\PY{p}{(}\PY{l+m+mi}{0}\PY{p}{,}\PY{l+m+mi}{2}\PY{p}{)}\PY{o}{+}\PY{n+nb}{range}\PY{p}{(}\PY{n}{N}\PY{o}{\PYZhy{}}\PY{n}{M}\PY{o}{+}\PY{l+m+mi}{2}\PY{p}{,}\PY{n}{N}\PY{o}{+}\PY{l+m+mi}{2}\PY{p}{)}\PY{p}{)}
    \PY{n}{f} \PY{o}{=} \PY{n+nb}{range}\PY{p}{(}\PY{l+m+mi}{0}\PY{p}{,}\PY{n}{binomial}\PY{p}{(}\PY{n}{M}\PY{o}{+}\PY{l+m+mi}{2}\PY{p}{,}\PY{n}{M}\PY{p}{)}\PY{p}{)}\PY{p}{;} \PY{n}{counter} \PY{o}{=} \PY{l+m+mi}{0}
    \PY{k}{for} \PY{n}{i} \PY{o+ow}{in} \PY{n+nb}{range}\PY{p}{(}\PY{l+m+mi}{0}\PY{p}{,}\PY{n}{M}\PY{o}{+}\PY{l+m+mi}{2}\PY{p}{)}\PY{p}{:}
        \PY{k}{for} \PY{n}{j} \PY{o+ow}{in} \PY{n+nb}{range} \PY{p}{(}\PY{n}{i}\PY{o}{+}\PY{l+m+mi}{1}\PY{p}{,}\PY{n}{M}\PY{o}{+}\PY{l+m+mi}{2}\PY{p}{)}\PY{p}{:}
            \PY{n}{entries} \PY{o}{=} \PY{n+nb}{range}\PY{p}{(}\PY{l+m+mi}{0}\PY{p}{,}\PY{n}{M}\PY{o}{+}\PY{l+m+mi}{2}\PY{p}{)}\PY{p}{;} \PY{n}{entries}\PY{o}{.}\PY{n}{remove}\PY{p}{(}\PY{n}{i}\PY{p}{)}\PY{p}{;} \PY{n}{entries}\PY{o}{.}\PY{n}{remove}\PY{p}{(}\PY{n}{j}\PY{p}{)}
            \PY{n}{f}\PY{p}{[}\PY{n}{counter}\PY{p}{]} \PY{o}{=} \PY{n}{B}\PY{o}{.}\PY{n}{matrix\PYZus{}from\PYZus{}rows\PYZus{}and\PYZus{}columns}\PY{p}{(}\PY{n+nb}{range}\PY{p}{(}\PY{l+m+mi}{0}\PY{p}{,}\PY{n}{M}\PY{p}{)}\PY{p}{,} \PY{n}{entries}\PY{p}{)}\PY{o}{.}\PY{n}{det}\PY{p}{(}\PY{p}{)}
            \PY{n}{counter} \PY{o}{+}\PY{o}{=} \PY{l+m+mi}{1}
    \PY{k}{return} \PY{n}{gcd}\PY{p}{(}\PY{n}{f}\PY{p}{)}

\PY{c}{\PYZsh{} Returns a polynomial whose roots are exceptional congruence}
\PY{c}{\PYZsh{} classes of $r$ modulo $p$ outside of which the relevant factors of}
\PY{c}{\PYZsh{} $\smash{\overline{\theta}^\alpha\sigma_{r-\alpha(p+1)}/\overline{\theta}^{\alpha+1}\sigma_{r-(\alpha+1)(p+1)}}$ cancel, in the case $s\not\in\{2,\ldots,2m\}$.}
\PY{k}{def} \PY{n+nf}{polynomial\PYZus{}from\PYZus{}the\PYZus{}big\PYZus{}matrix}\PY{p}{(}\PY{n}{m}\PY{p}{,}\PY{n}{alpha}\PY{p}{)}\PY{p}{:}
    \PY{c}{\PYZsh{} $M$ is the number of rows and the number of columns.}
    \PY{n}{M} \PY{o}{=} \PY{n}{alpha}\PY{o}{+}\PY{l+m+mi}{1}\PY{p}{;} \PY{n}{A} \PY{o}{=} \PY{n}{construct\PYZus{}big\PYZus{}matrix}\PY{p}{(}\PY{n}{m}\PY{p}{,}\PY{n}{alpha}\PY{p}{)}
    \PY{k}{for} \PY{n}{i} \PY{o+ow}{in} \PY{n+nb}{range}\PY{p}{(}\PY{l+m+mi}{0}\PY{p}{,}\PY{n}{M}\PY{p}{)}\PY{p}{:}
        \PY{n}{A}\PY{p}{[}\PY{n}{i}\PY{p}{,}\PY{l+m+mi}{0}\PY{p}{]} \PY{o}{=} \PY{n}{A}\PY{p}{[}\PY{n}{i}\PY{p}{,}\PY{l+m+mi}{0}\PY{p}{]}\PY{o}{/}\PY{n}{p}
    \PY{n}{B} \PY{o}{=} \PY{n}{A}\PY{o}{.}\PY{n}{matrix\PYZus{}from\PYZus{}rows\PYZus{}and\PYZus{}columns}\PY{p}{(}\PY{n+nb}{range}\PY{p}{(}\PY{l+m+mi}{0}\PY{p}{,}\PY{n}{M}\PY{o}{\PYZhy{}}\PY{l+m+mi}{1}\PY{p}{)}\PY{p}{,} \PY{n+nb}{range}\PY{p}{(}\PY{l+m+mi}{0}\PY{p}{,}\PY{n}{M}\PY{p}{)}\PY{p}{)}
    \PY{n}{C} \PY{o}{=} \PY{n}{B}\PY{o}{.}\PY{n}{transpose}\PY{p}{(}\PY{p}{)}\PY{o}{.}\PY{n}{kernel}\PY{p}{(}\PY{p}{)}\PY{o}{.}\PY{n}{basis\PYZus{}matrix}\PY{p}{(}\PY{p}{)}
    \PY{n}{x} \PY{o}{=} \PY{n+nb}{sum}\PY{p}{(}\PY{p}{[}\PY{n}{A}\PY{p}{[}\PY{n}{M}\PY{o}{\PYZhy{}}\PY{l+m+mi}{1}\PY{p}{,}\PY{n}{v}\PY{p}{]}\PY{o}{*}\PY{n}{C}\PY{p}{[}\PY{l+m+mi}{0}\PY{p}{,}\PY{n}{v}\PY{p}{]} \PY{k}{for} \PY{n}{v} \PY{o+ow}{in} \PY{n+nb}{range}\PY{p}{(}\PY{l+m+mi}{0}\PY{p}{,}\PY{n}{M}\PY{p}{)}\PY{p}{]}\PY{p}{)}
    \PY{k}{return} \PY{n}{x}\PY{o}{.}\PY{n}{subs}\PY{p}{(}\PY{n}{r}\PY{o}{=}\PY{n}{s}\PY{o}{\PYZhy{}}\PY{n}{t}\PY{p}{)} \PY{k}{if} \PY{n}{C}\PY{p}{[}\PY{l+m+mi}{0}\PY{p}{,}\PY{l+m+mi}{0}\PY{p}{]}\PY{o}{==}\PY{l+m+mi}{1} \PY{k}{else} \PY{p}{(}\PY{n}{t}\PY{o}{+}\PY{l+m+mi}{1}\PY{p}{)}

\PY{c}{\PYZsh{} Returns the product of all polynomials for all choices for $\alpha$.}
\PY{k}{def} \PY{n+nf}{the\PYZus{}roots\PYZus{}for\PYZus{}all\PYZus{}big\PYZus{}matrices}\PY{p}{(}\PY{n}{m}\PY{p}{)}\PY{p}{:}
    \PY{k}{return} \PY{n+nb}{reduce}\PY{p}{(}\PY{n}{mul}\PY{p}{,}\PY{p}{[}\PY{n}{polynomial\PYZus{}from\PYZus{}the\PYZus{}big\PYZus{}matrix}\PY{p}{(}\PY{n}{m}\PY{p}{,}\PY{n}{alpha}\PY{p}{)} \PYZbs{}
                   \PY{k}{for} \PY{n}{alpha} \PY{o+ow}{in} \PY{n+nb}{range}\PY{p}{(}\PY{l+m+mi}{0}\PY{p}{,}\PY{n}{m}\PY{p}{)}\PY{p}{]}\PY{p}{,}\PY{l+m+mi}{1}\PY{p}{)}

\PY{k}{def} \PY{n+nf}{print\PYZus{}the\PYZus{}gcd\PYZus{}for\PYZus{}the\PYZus{}matrix}\PY{p}{(}\PY{n}{m}\PY{p}{,}\PY{n}{alpha}\PY{p}{,}\PY{n}{L}\PY{p}{)}\PY{p}{:}
    \PY{n}{A} \PY{o}{=} \PY{n}{construct\PYZus{}matrix}\PY{p}{(}\PY{n}{m}\PY{p}{,}\PY{n}{alpha}\PY{p}{,}\PY{n}{L}\PY{p}{)}
    \PY{n}{g} \PY{o}{=} \PY{n}{gcd\PYZus{}for\PYZus{}the\PYZus{}matrix}\PY{p}{(}\PY{n}{m}\PY{p}{,}\PY{n}{alpha}\PY{p}{,}\PY{n}{L}\PY{p}{)}
    \PY{k}{print} \PY{n}{A}
    \PY{k}{if} \PY{n}{g}\PY{o}{!=}\PY{l+m+mi}{0}\PY{p}{:}
        \PY{k}{if} \PY{n}{g}\PY{o}{.}\PY{n}{roots}\PY{p}{(}\PY{p}{)}\PY{o}{!=}\PY{p}{[}\PY{p}{]}\PY{p}{:}
            \PY{k}{for} \PY{n}{zeta} \PY{o+ow}{in} \PY{n}{g}\PY{o}{.}\PY{n}{roots}\PY{p}{(}\PY{p}{)}\PY{p}{:}
                \PY{k}{print} \PY{l+s}{\PYZdq{}}\PY{l+s}{Substituting r=}\PY{l+s+si}{\PYZpc{}d}\PY{l+s}{ in the matrix,}\PY{l+s}{\PYZdq{}}\PY{o}{\PYZpc{}}\PY{n}{zeta}\PY{p}{[}\PY{l+m+mi}{0}\PY{p}{]}\PY{p}{,} \PYZbs{}
                    \PY{l+s}{\PYZdq{}}\PY{l+s}{with m=}\PY{l+s+si}{\PYZpc{}d}\PY{l+s}{, L=}\PY{l+s+si}{\PYZpc{}d}\PY{l+s}{, alpha=}\PY{l+s+si}{\PYZpc{}d}\PY{l+s}{, yields}\PY{l+s}{\PYZdq{}}\PY{o}{\PYZpc{}}\PY{p}{(}\PY{n}{m}\PY{p}{,} \PY{n}{L}\PY{p}{,} \PY{n}{alpha}\PY{p}{)}
                \PY{k}{print} \PY{n}{substitute\PYZus{}in\PYZus{}matrix}\PY{p}{(}\PY{n}{A}\PY{p}{,}\PY{n}{zeta}\PY{p}{[}\PY{l+m+mi}{0}\PY{p}{]}\PY{p}{)}
    \PY{k}{print} \PY{n}{A}\PY{o}{.}\PY{n}{transpose}\PY{p}{(}\PY{p}{)}\PY{o}{.}\PY{n}{kernel}\PY{p}{(}\PY{p}{)}
    \PY{k}{print} \PY{l+s}{\PYZdq{}}\PY{l+s}{For m=}\PY{l+s+si}{\PYZpc{}d}\PY{l+s}{, L=}\PY{l+s+si}{\PYZpc{}d}\PY{l+s}{, alpha=}\PY{l+s+si}{\PYZpc{}d}\PY{l+s}{, the GCD of the determinants =}\PY{l+s}{\PYZdq{}}\PY{o}{\PYZpc{}}\PY{p}{(}\PY{n}{m}\PY{p}{,} \PY{n}{L}\PY{p}{,} \PY{n}{alpha}\PY{p}{)}\PY{p}{,} \PYZbs{}
                \PY{n}{factorise}\PY{p}{(}\PY{n}{g}\PY{p}{)}\PY{p}{,} \PY{l+s}{\PYZdq{}}\PY{l+s+se}{\PYZbs{}n}\PY{l+s+se}{\PYZbs{}n}\PY{l+s}{\PYZdq{}}

\PY{k}{def} \PY{n+nf}{print\PYZus{}the\PYZus{}roots\PYZus{}for\PYZus{}all\PYZus{}matrices}\PY{p}{(}\PY{n}{m}\PY{p}{,}\PY{n}{quiet}\PY{p}{)}\PY{p}{:}
    \PY{k}{for} \PY{n}{L} \PY{o+ow}{in} \PY{n+nb}{range}\PY{p}{(}\PY{l+m+mi}{1}\PY{p}{,}\PY{n}{m}\PY{o}{+}\PY{l+m+mi}{1}\PY{p}{)}\PY{p}{:}
        \PY{n}{product} \PY{o}{=} \PY{n}{r}\PY{o}{/}\PY{n}{r}
        \PY{k}{for} \PY{n}{alpha} \PY{o+ow}{in} \PY{n+nb}{range}\PY{p}{(}\PY{l+m+mi}{1}\PY{p}{,}\PY{n}{m}\PY{o}{+}\PY{l+m+mi}{1}\PY{p}{)}\PY{p}{:}
            \PY{n}{product} \PY{o}{=} \PY{n}{product}\PY{o}{*}\PY{n}{exceptional\PYZus{}cases}\PY{p}{(}\PY{n}{m}\PY{p}{,}\PY{n}{alpha}\PY{p}{,}\PY{n}{L}\PY{p}{)}
            \PY{k}{if} \PY{n}{quiet}\PY{o}{==}\PY{n+nb+bp}{False}\PY{p}{:}
                \PY{n}{print\PYZus{}the\PYZus{}gcd\PYZus{}for\PYZus{}the\PYZus{}matrix}\PY{p}{(}\PY{n}{m}\PY{p}{,}\PY{n}{alpha}\PY{p}{,}\PY{n}{L}\PY{p}{)}
        \PY{k}{print} \PY{l+s}{\PYZdq{}}\PY{l+s}{In the case when m=}\PY{l+s+si}{\PYZpc{}d}\PY{l+s}{ and s=}\PY{l+s+si}{\PYZpc{}d}\PY{l+s}{, Theta\PYZca{}bar is irreducible unless}\PY{l+s}{\PYZdq{}}\PY{o}{\PYZpc{}}\PY{p}{(}\PY{n}{m}\PY{p}{,} \PY{l+m+mi}{2}\PY{o}{*}\PY{n}{L}\PY{p}{)}\PY{p}{,} \PYZbs{}
                    \PY{l+s}{\PYZdq{}}\PY{l+s}{r is congruent to one of the following numbers modulo p:}\PY{l+s}{\PYZdq{}}
        \PY{k}{print} \PY{n}{get\PYZus{}roots}\PY{p}{(}\PY{n}{product}\PY{p}{)}\PY{p}{,} \PY{l+s}{\PYZdq{}}\PY{l+s+se}{\PYZbs{}n}\PY{l+s+se}{\PYZbs{}n}\PY{l+s+se}{\PYZbs{}n}\PY{l+s+se}{\PYZbs{}n}\PY{l+s}{\PYZdq{}}

\PY{c}{\PYZsh{} Returns a list of exceptional congruence classes of $r$}
\PY{c}{\PYZsh{} modulo $p$ outside of which $\overline{\Theta}_{r+2,a}$ is irreducible.}
\PY{k}{def} \PY{n+nf}{print\PYZus{}everything}\PY{p}{(}\PY{n}{m}\PY{p}{)}\PY{p}{:}
    \PY{n}{print\PYZus{}the\PYZus{}roots\PYZus{}for\PYZus{}all\PYZus{}matrices}\PY{p}{(}\PY{n}{m}\PY{p}{,}\PY{n+nb+bp}{True}\PY{p}{)}
    \PY{k}{print} \PY{l+s}{\PYZdq{}}\PY{l+s}{In the case when m=}\PY{l+s+si}{\PYZpc{}d}\PY{l+s}{ and s=/=2,...,}\PY{l+s+si}{\PYZpc{}d}\PY{l+s}{, Theta\PYZca{}bar is irreducible unless}\PY{l+s}{\PYZdq{}}\PY{o}{\PYZpc{}}\PY{p}{(}\PY{n}{m}\PY{p}{,} \PY{l+m+mi}{2}\PY{o}{*}\PY{n}{m}\PY{p}{)}\PY{p}{,} \PYZbs{}
                    \PY{l+s}{\PYZdq{}}\PY{l+s}{the following polynomial vanishes modulo p, with t=s\PYZhy{}r:}\PY{l+s}{\PYZdq{}}
    \PY{k}{print} \PY{n}{the\PYZus{}roots\PYZus{}for\PYZus{}all\PYZus{}big\PYZus{}matrices}\PY{p}{(}\PY{n}{m}\PY{p}{)}

\PY{k}{def} \PY{n+nf}{verify\PYZus{}conjecture\PYZus{}eight}\PY{p}{(}\PY{n}{m}\PY{p}{)}\PY{p}{:}
    \PY{n}{conjecture\PYZus{}is\PYZus{}true} \PY{o}{=} \PY{n+nb+bp}{True}
    \PY{k}{for} \PY{n}{L} \PY{o+ow}{in} \PY{n+nb}{range}\PY{p}{(}\PY{l+m+mi}{1}\PY{p}{,}\PY{n}{m}\PY{o}{+}\PY{l+m+mi}{1}\PY{p}{)}\PY{p}{:}
        \PY{n}{product} \PY{o}{=} \PY{n}{r}\PY{o}{/}\PY{n}{r}
        \PY{k}{for} \PY{n}{alpha} \PY{o+ow}{in} \PY{n+nb}{range}\PY{p}{(}\PY{l+m+mi}{1}\PY{p}{,}\PY{n}{m}\PY{o}{+}\PY{l+m+mi}{1}\PY{p}{)}\PY{p}{:}
            \PY{n}{product} \PY{o}{=} \PY{n}{product}\PY{o}{*}\PY{n}{exceptional\PYZus{}cases}\PY{p}{(}\PY{n}{m}\PY{p}{,}\PY{n}{alpha}\PY{p}{,}\PY{n}{L}\PY{p}{)}
        \PY{n}{roots} \PY{o}{=} \PY{n}{get\PYZus{}roots}\PY{p}{(}\PY{n}{product}\PY{p}{)}
        \PY{k}{if} \PY{n+nb}{min}\PY{p}{(}\PY{n}{roots}\PY{p}{)}\PY{o}{\PYZlt{}}\PY{l+m+mi}{2}\PY{o}{*}\PY{p}{(}\PY{n}{L}\PY{o}{\PYZhy{}}\PY{n}{m}\PY{p}{)} \PY{o+ow}{or} \PY{n+nb}{max}\PY{p}{(}\PY{n}{roots}\PY{p}{)}\PY{o}{\PYZgt{}}\PY{l+m+mi}{2}\PY{o}{*}\PY{n}{L}\PY{p}{:}
            \PY{n}{conjecture\PYZus{}is\PYZus{}true} \PY{o}{=} \PY{n+nb+bp}{False}
    \PY{n}{f} \PY{o}{=} \PY{n}{the\PYZus{}roots\PYZus{}for\PYZus{}all\PYZus{}big\PYZus{}matrices}\PY{p}{(}\PY{n}{m}\PY{p}{)}
    \PY{n}{g} \PY{o}{=} \PY{n+nb}{reduce}\PY{p}{(}\PY{n}{mul}\PY{p}{,}\PY{p}{[}\PY{p}{(}\PY{n}{t}\PY{o}{\PYZhy{}}\PY{n}{alpha}\PY{p}{)}\PY{o}{\PYZca{}}\PY{n}{m} \PY{k}{for} \PY{n}{alpha} \PY{o+ow}{in} \PY{n+nb}{range}\PY{p}{(}\PY{l+m+mi}{0}\PY{p}{,}\PY{n}{m}\PY{p}{)}\PY{p}{]}\PY{p}{,}\PY{l+m+mi}{1}\PY{p}{)}
    \PY{k}{if} \PY{n}{Ring}\PY{p}{(}\PY{n}{g}\PY{o}{/}\PY{n}{f}\PY{p}{)}\PY{o}{.}\PY{n}{factor}\PY{p}{(}\PY{p}{)}\PY{o}{.}\PY{n}{is\PYZus{}integral}\PY{p}{(}\PY{p}{)}\PY{o}{==}\PY{n+nb+bp}{False}\PY{p}{:}
        \PY{n}{conjecture\PYZus{}is\PYZus{}true} \PY{o}{=} \PY{n+nb+bp}{False}
    \PY{k}{if} \PY{n}{conjecture\PYZus{}is\PYZus{}true}\PY{o}{==}\PY{n+nb+bp}{True}\PY{p}{:}
        \PY{k}{print} \PY{l+s}{\PYZdq{}}\PY{l+s}{When m=}\PY{l+s+si}{\PYZpc{}d}\PY{l+s}{, the first two parts of conjecture eight are true.}\PY{l+s}{\PYZdq{}}\PY{o}{\PYZpc{}}\PY{n}{m}
    \PY{k}{if} \PY{n}{conjecture\PYZus{}is\PYZus{}true}\PY{o}{==}\PY{n+nb+bp}{False}\PY{p}{:}
        \PY{k}{print} \PY{l+s}{\PYZdq{}}\PY{l+s}{When m=}\PY{l+s+si}{\PYZpc{}d}\PY{l+s}{, the first two parts of conjecture eight are not necessarily true!}\PY{l+s}{\PYZdq{}}\PY{o}{\PYZpc{}}\PY{n}{m}
\end{Verbatim}
}

\subsubsection{Finding $\mathfrak m_w(C_1,\ldots,C_\alpha)$}

The following program can be used to find suitable constants $C_1,\ldots,C_\alpha$ as in the statement of  theorem~\ref{theorem26}, and calculate the expressions $\mathfrak m_w$, for $0 \leqslant w \leqslant 2m+1-\alpha$. In particular, calling the method \verb|find_m_w(m,alpha,L)| with $\smash{\mathtt{m}=m}$, and $\smash{\mathtt{alpha}=\alpha}$, and $\smash{\mathtt{L}=L}$, lists 
\[\mathfrak m_0(C_1,\ldots,C_\alpha), \ldots, \mathfrak m_{2m+1-\alpha} (C_1,\ldots,C_\alpha).\]

\

\textcolor{OrangeRed}{\texttt{m\_w.sage}}\\[-8pt]
{\footnotesize
\begin{Verbatim}[commandchars=\\\{\},numbers=left,firstnumber=last,stepnumber=1,codes={\catcode`\$=3\catcode`\^=7\catcode`\_=8}]
\PY{p}{[}\PY{n}{r}\PY{p}{,}\PY{n}{s}\PY{p}{,}\PY{n}{t}\PY{p}{,}\PY{n}{p}\PY{p}{,}\PY{n}{v}\PY{p}{]} \PY{o}{=} \PY{n}{var}\PY{p}{(}\PY{l+s}{\PYZsq{}}\PY{l+s}{r s t p v}\PY{l+s}{\PYZsq{}}\PY{p}{)}
\PY{n}{Ring} \PY{o}{=} \PY{n}{FractionField}\PY{p}{(}\PY{n}{PolynomialRing}\PY{p}{(}\PY{n}{QQ}\PY{p}{,}\PY{p}{[}\PY{l+s}{\PYZsq{}}\PY{l+s}{r}\PY{l+s}{\PYZsq{}}\PY{p}{,}\PY{l+s}{\PYZsq{}}\PY{l+s}{s}\PY{l+s}{\PYZsq{}}\PY{p}{,}\PY{l+s}{\PYZsq{}}\PY{l+s}{t}\PY{l+s}{\PYZsq{}}\PY{p}{,}\PY{l+s}{\PYZsq{}}\PY{l+s}{p}\PY{l+s}{\PYZsq{}}\PY{p}{]}\PY{p}{)}\PY{p}{)}

\PY{k}{def} \PY{n+nf}{factorise}\PY{p}{(}\PY{n}{g}\PY{p}{)}\PY{p}{:}
    \PY{k}{return} \PY{n}{g}\PY{o}{.}\PY{n}{factor}\PY{p}{(}\PY{p}{)} \PY{k}{if} \PY{n}{g}\PY{o}{!=}\PY{l+m+mi}{0} \PY{k}{else} \PY{n}{g}

\PY{c}{\PYZsh{} Returns $\eta(X,Y)$, from lemma \ref{lemma7wefwfwerfwf}.}
\PY{k}{def} \PY{n+nf}{eta}\PY{p}{(}\PY{n}{X}\PY{p}{,}\PY{n}{Y}\PY{p}{)}\PY{p}{:}
    \PY{k}{if} \PY{n}{X}\PY{o}{\PYZgt{}}\PY{o}{=}\PY{l+m+mi}{1} \PY{o+ow}{and} \PY{n}{Y}\PY{o}{\PYZgt{}}\PY{o}{=}\PY{l+m+mi}{0}\PY{p}{:}
        \PY{k}{return} \PY{n}{binomial}\PY{p}{(}\PY{n}{X}\PY{p}{,}\PY{n}{Y}\PY{p}{)}
    \PY{k}{if} \PY{n}{X}\PY{o}{\PYZlt{}}\PY{l+m+mi}{1} \PY{o+ow}{and} \PY{n}{Y}\PY{o}{\PYZgt{}}\PY{o}{=}\PY{l+m+mi}{0}\PY{p}{:}
        \PY{k}{return} \PY{l+m+mi}{2}\PY{o}{*}\PY{n}{binomial}\PY{p}{(}\PY{n}{X}\PY{o}{\PYZhy{}}\PY{l+m+mi}{1}\PY{p}{,}\PY{n}{Y}\PY{p}{)} \PY{k}{if} \PY{n}{X}\PY{o}{==}\PY{l+m+mi}{0} \PY{o+ow}{and} \PY{n}{Y}\PY{o}{==}\PY{l+m+mi}{0} \PY{k}{else} \PY{n}{binomial}\PY{p}{(}\PY{n}{X}\PY{o}{\PYZhy{}}\PY{l+m+mi}{1}\PY{p}{,}\PY{n}{Y}\PY{p}{)}
    \PY{k}{if} \PY{n}{X}\PY{o}{\PYZlt{}}\PY{l+m+mi}{1} \PY{o+ow}{and} \PY{n}{Y}\PY{o}{\PYZlt{}}\PY{l+m+mi}{0}\PY{p}{:}
        \PY{k}{return} \PY{n}{binomial}\PY{p}{(}\PY{n}{X}\PY{o}{\PYZhy{}}\PY{l+m+mi}{1}\PY{p}{,}\PY{n}{X}\PY{o}{\PYZhy{}}\PY{n}{Y}\PY{p}{)} \PY{k}{if} \PY{n}{X}\PY{o}{\PYZgt{}}\PY{o}{=}\PY{n}{Y} \PY{k}{else} \PY{l+m+mi}{0}
    \PY{k}{return} \PY{l+m+mi}{0}

\PY{c}{\PYZsh{} Returns the relevant matrix $\bmod{p}$ arising from theorem \ref{theorem26}, when $s<2m$.}
\PY{k}{def} \PY{n+nf}{construct\PYZus{}small\PYZus{}matrix}\PY{p}{(}\PY{n}{m}\PY{p}{,}\PY{n}{alpha}\PY{p}{,}\PY{n}{L}\PY{p}{)}\PY{p}{:}
    \PY{c}{\PYZsh{} $N$ is the number of rows, $M$ is the number of columns.}
    \PY{n}{M} \PY{o}{=} \PY{n}{alpha}\PY{o}{+}\PY{l+m+mi}{1}\PY{p}{;} \PY{n}{N} \PY{o}{=} \PY{l+m+mi}{2}\PY{o}{*}\PY{n}{m}\PY{o}{+}\PY{l+m+mi}{2}\PY{o}{\PYZhy{}}\PY{n}{alpha}\PY{p}{;} \PY{n}{A} \PY{o}{=} \PY{n}{matrix}\PY{p}{(}\PY{n}{Ring}\PY{p}{,}\PY{n}{N}\PY{p}{,}\PY{n}{M}\PY{p}{)}

    \PY{k}{for} \PY{n}{l} \PY{o+ow}{in} \PY{n+nb}{range}\PY{p}{(}\PY{l+m+mi}{0}\PY{p}{,}\PY{n}{M}\PY{p}{)}\PY{p}{:}
        \PY{k}{for} \PY{n}{w} \PY{o+ow}{in} \PY{n+nb}{range}\PY{p}{(}\PY{l+m+mi}{0}\PY{p}{,}\PY{n}{N}\PY{p}{)}\PY{p}{:}
            \PY{n}{a} \PY{o}{=} \PY{o}{\PYZhy{}}\PY{n}{binomial}\PY{p}{(}\PY{n}{r}\PY{o}{\PYZhy{}}\PY{n}{alpha}\PY{o}{+}\PY{n}{l}\PY{p}{,}\PY{l+m+mi}{2}\PY{o}{*}\PY{n}{L}\PY{o}{\PYZhy{}}\PY{n}{alpha}\PY{p}{)}\PY{o}{*}\PY{n}{binomial}\PY{p}{(}\PY{l+m+mi}{2}\PY{o}{*}\PY{n}{L}\PY{o}{\PYZhy{}}\PY{n}{r}\PY{p}{,}\PY{n}{w}\PY{p}{)} \PY{k}{if} \PY{n}{alpha}\PY{o}{\PYZgt{}}\PY{o}{=}\PY{n}{L} \PY{k}{else} \PY{l+m+mi}{0}
            \PY{n}{b} \PY{o}{=} \PY{o}{\PYZhy{}}\PY{n}{binomial}\PY{p}{(}\PY{n}{r}\PY{o}{\PYZhy{}}\PY{n}{alpha}\PY{o}{+}\PY{n}{l}\PY{p}{,}\PY{n}{l}\PY{p}{)}\PY{o}{*}\PY{n}{binomial}\PY{p}{(}\PY{l+m+mi}{0}\PY{p}{,}\PY{n}{w}\PY{p}{)} \PYZbs{}
                    \PY{o}{+}\PY{n+nb}{sum}\PY{p}{(}\PY{p}{[}\PY{p}{(}\PY{o}{\PYZhy{}}\PY{l+m+mi}{1}\PY{p}{)}\PY{o}{\PYZca{}}\PY{p}{(}\PY{n}{v}\PY{p}{)}\PY{o}{*}\PY{n}{binomial}\PY{p}{(}\PY{n}{l}\PY{o}{\PYZhy{}}\PY{n}{v}\PY{p}{,}\PY{n}{w}\PY{o}{\PYZhy{}}\PY{n}{v}\PY{p}{)}\PY{o}{*}\PY{n}{eta}\PY{p}{(}\PY{l+m+mi}{2}\PY{o}{*}\PY{n}{L}\PY{o}{\PYZhy{}}\PY{n}{alpha}\PY{o}{+}\PY{n}{l}\PY{o}{\PYZhy{}}\PY{n}{v}\PY{p}{,}\PY{n}{l}\PY{o}{\PYZhy{}}\PY{n}{v}\PY{p}{)} \PYZbs{}
                        \PY{o}{*}\PY{n}{binomial}\PY{p}{(}\PY{n}{r}\PY{o}{\PYZhy{}}\PY{n}{alpha}\PY{o}{+}\PY{n}{l}\PY{p}{,}\PY{n}{v}\PY{p}{)} \PY{k}{for} \PY{n}{v} \PY{o+ow}{in} \PY{n+nb}{range}\PY{p}{(}\PY{l+m+mi}{0}\PY{p}{,}\PY{n}{w}\PY{o}{+}\PY{l+m+mi}{1}\PY{p}{)}\PY{p}{]}\PY{p}{)}
            \PY{n}{A}\PY{p}{[}\PY{n}{w}\PY{p}{,}\PY{n}{l}\PY{p}{]} \PY{o}{=} \PY{n}{factorise}\PY{p}{(}\PY{n}{simplify}\PY{p}{(}\PY{n}{a}\PY{o}{+}\PY{n}{b}\PY{p}{)}\PY{p}{)}
    \PY{k}{return} \PY{n}{A}

\PY{c}{\PYZsh{} Returns the list $[\mathfrak{m}_0,\ldots,\mathfrak{m}_{2m+1-\alpha}]$.}
\PY{k}{def} \PY{n+nf}{find\PYZus{}m\PYZus{}w}\PY{p}{(}\PY{n}{m}\PY{p}{,}\PY{n}{alpha}\PY{p}{,}\PY{n}{L}\PY{p}{)}\PY{p}{:}
    \PY{c}{\PYZsh{} $N$ is the number of rows, $M$ is the number of columns.}
    \PY{n}{M} \PY{o}{=} \PY{n}{alpha}\PY{o}{+}\PY{l+m+mi}{1}\PY{p}{;} \PY{n}{N} \PY{o}{=} \PY{l+m+mi}{2}\PY{o}{*}\PY{n}{m}\PY{o}{+}\PY{l+m+mi}{2}\PY{o}{\PYZhy{}}\PY{n}{alpha}\PY{p}{;} \PY{n}{A} \PY{o}{=} \PY{n}{construct\PYZus{}small\PYZus{}matrix}\PY{p}{(}\PY{n}{m}\PY{p}{,}\PY{n}{alpha}\PY{p}{,}\PY{n}{L}\PY{p}{)}
    \PY{n}{B} \PY{o}{=} \PY{n}{A}\PY{o}{.}\PY{n}{matrix\PYZus{}from\PYZus{}rows\PYZus{}and\PYZus{}columns}\PY{p}{(}\PY{n+nb}{range}\PY{p}{(}\PY{l+m+mi}{0}\PY{p}{,}\PY{n}{M}\PY{o}{\PYZhy{}}\PY{l+m+mi}{1}\PY{p}{)}\PY{p}{,} \PY{n+nb}{range}\PY{p}{(}\PY{l+m+mi}{0}\PY{p}{,}\PY{n}{M}\PY{p}{)}\PY{p}{)}
    \PY{n}{C} \PY{o}{=} \PY{n}{B}\PY{o}{.}\PY{n}{transpose}\PY{p}{(}\PY{p}{)}\PY{o}{.}\PY{n}{kernel}\PY{p}{(}\PY{p}{)}\PY{o}{.}\PY{n}{basis\PYZus{}matrix}\PY{p}{(}\PY{p}{)}
    \PY{k}{print} \PY{l+s}{\PYZdq{}}\PY{l+s}{The list of constants C\PYZus{}l is}\PY{l+s}{\PYZdq{}}\PY{p}{,} \PY{n}{C}
    \PY{n}{x} \PY{o}{=} \PY{n+nb}{range}\PY{p}{(}\PY{l+m+mi}{0}\PY{p}{,}\PY{n}{N}\PY{p}{)}
    \PY{k}{for} \PY{n}{j} \PY{o+ow}{in} \PY{n+nb}{range}\PY{p}{(}\PY{l+m+mi}{0}\PY{p}{,}\PY{n}{N}\PY{p}{)}\PY{p}{:}
        \PY{n}{x}\PY{p}{[}\PY{n}{j}\PY{p}{]} \PY{o}{=} \PY{n}{factorise}\PY{p}{(}\PY{n+nb}{sum}\PY{p}{(}\PY{p}{[}\PY{n}{A}\PY{p}{[}\PY{n}{j}\PY{p}{,}\PY{n}{v}\PY{p}{]}\PY{o}{*}\PY{n}{C}\PY{p}{[}\PY{l+m+mi}{0}\PY{p}{,}\PY{n}{v}\PY{p}{]} \PY{k}{for} \PY{n}{v} \PY{o+ow}{in} \PY{n+nb}{range}\PY{p}{(}\PY{l+m+mi}{0}\PY{p}{,}\PY{n}{M}\PY{p}{)}\PY{p}{]}\PY{p}{)}\PY{p}{)}
    \PY{k}{return} \PY{n}{x}
\end{Verbatim}
}
\phantomsection\addcontentsline{toc}{section}{Acknowledgments}\section*{Acknowledgments}

I would like to thank my advisor, Professor Kevin Buzzard, for his helpful remarks, suggestions, and support. This work was supported by Imperial College London.

\def\bysame{\leavevmode\hbox to3em{\hrulefill}\thinspace}\newcommand{\NUMBERRRR}{\#}\newcommand{\PAGESS}{}

{%
\phantomsection}

\end{document}